\documentclass{article}
\usepackage[sort]{natbib}
\usepackage[a4paper, margin=1in]{geometry}

\usepackage{titlesec}
\titleformat{\paragraph}
{\normalfont\normalsize\bfseries}{\theparagraph}{1em}{}
\titlespacing*{\paragraph}
{0pt}{3.25ex plus 1ex minus .2ex}{1.5ex plus .2ex}
\usepackage{authblk}

\usepackage{graphicx} 

\usepackage{amsmath}
\usepackage{amsthm}
\usepackage{amssymb}

\usepackage{float}
\usepackage{caption}
\usepackage{subcaption}

\usepackage{booktabs}

\usepackage{algorithm}%
\usepackage{algorithmicx}%
\usepackage{algpseudocode}%

\usepackage{hyperref}
\usepackage[dvipsnames]{xcolor}
\hypersetup{colorlinks=true, allcolors=NavyBlue}

\usepackage[capitalise]{cleveref}

\usepackage{pifont}

\newtheorem{theorem}{Theorem}[section]

\newtheorem{definition}{Definition}[section]

\newtheorem{lemma}{Lemma}[section]
\newtheorem{remark}{Remark}[section]%
\newtheorem{conjecture}{Conjecture}[section]

\newcommand{\bbm}{\boldsymbol{\bar \mu}}

\newcommand{\bbn}{\boldsymbol{\bar \nu}}

\newcommand{\A}{\boldsymbol{A}}
\newcommand{\X}{\boldsymbol{X}}

\renewcommand{\H}{\bs{H}}

\renewcommand{\S}{\mathbf{S}}
\renewcommand{\O}{\mathcal{O}}

\newcommand{\td}{\tilde\delta}

\newcommand{\RR}{\mathbf{R}}

\newcommand{\Z}{\bs{Z}}
\newcommand{\E}{\mathbf{E}}

\usepackage{amsmath}

\newcommand{\rank}{\mathrm{rank}}

\newcommand{\suchthat}{\mathrm{s.t.}\;}
\newcommand{\tr}{\mathrm{Tr}}

\renewcommand{\d}{\bs d}

\newcommand{\e}{\bs {1}}
\newcommand{\I}{\bs{I}}

\newcommand{\bs}{\boldsymbol}
\renewcommand{\u}{\bs{u}}
\renewcommand{\v}{\bs{v}}
\newcommand{\x}{\bs {x}}
\newcommand{\y}{\bs {y}}
\newcommand{\Y}{\bs {Y}}
\newcommand{\z}{\bs{z}}
\newcommand{\1}{\bs{1}}

\newcommand{\W}{\bs{W}}

\newcommand{\bu}{\boldsymbol{\bar{u}}}
\newcommand{\bv}{\boldsymbol{\bar{v}}}
\newcommand{\bX}{\boldsymbol{\bar{X}}}
\newcommand{\bY}{\boldsymbol{\bar{Y}}}

\newcommand{\R}{\bs{R}}
\newcommand{\Q}{\bs{Q}}

\newcommand{\BL}{\boldsymbol{\Lambda}}
\newcommand{\BX}{\bs{\Xi}}
\newcommand{\0}{\bs{0}}

\newcommand{\EX}{\mathbf{E}}
\DeclareMathOperator{\Var}{Var}
\DeclareMathOperator{\prob}{Pr}

\title{Provably Finding a Hidden Dense Submatrix among Many Planted Dense Submatrices via Convex Programming}

\author[1]{Valentine Olanubi}
\author[2]{Phineas Agar}
\author[3]{Brendan Ames}
\affil[1]{Department of Mathematics, University of Alabama, Tuscaloosa, AL, USA, {vaolanubi@crimson.ua.edu}}
\affil[2]{Department of Mathematics, University of Alabama, Tuscaloosa, AL, USA, {paagar@crimson.ua.edu}}
\affil[3]{{School of Mathematical Sciences}, University of Southampton, Southampton, UK, {b.ames@soton.ac.uk}}

\begin{document}

 \maketitle

\begin{abstract}
We consider the densest submatrix problem, which seeks the submatrix of fixed size of a given binary matrix that contains the most nonzero entries. This problem is a natural generalization of fundamental problems in combinatorial optimization, e.g., the densest subgraph, maximum clique, and maximum edge biclique problems, and has wide application in the study of complex networks.
Much recent research has focused on the development of sufficient conditions for exact solution of the densest submatrix problem via convex relaxation. The vast majority of these sufficient conditions establish identification of the densest submatrix within a graph containing exactly one large dense submatrix hidden by noise. The assumptions of these underlying models are not observed in real-world networks, where the data may correspond to a matrix containing many dense submatrices of varying sizes.

We extend and generalize these results to the more realistic setting where the input matrix may contain \emph{many} large dense subgraphs. Specifically, we establish sufficient conditions under which we can expect to solve the densest submatrix problem in polynomial time for random input matrices sampled from a generalization of the stochastic block model. Moreover, we also provide sufficient conditions for perfect recovery under a deterministic adversarial model. Numerical experiments involving randomly generated problem instances and real-world collaboration and communication networks are used to empirically verify the theoretical phase-transitions to perfect recovery given by these sufficient conditions.
\end{abstract}

\subsection*{Data Availability Statement}

The real-world network datasets analysed in this study are publicly available. 
The Jazz Collaboration Network, Zachary's Karate Club, Dolphins, and Les 
Mis\'{e}rables networks are widely available through standard network analysis 
libraries and repositories. The character interaction network data for George 
R.R.\ Martin's \textit{A Song of Ice and Fire} series is available at 
\url{https://networkofthrones.com/}.

Python, MATLAB, and R implementations of the ADMM algorithm (Algorithm~1) 
used to produce the numerical results in this study are available in 
open-source repositories at \url{https://github.com/bpames/DENSUB} (Python), 
\url{https://github.com/pbombina/admmDSM} (MATLAB), and 
\url{https://cran.r-project.org/web/packages/admmDensestSubmatrix/index.html} 
(R). Numerical experiments were conducted on the IRIDIS High Performance 
Computing Facility at the University of Southampton.

\subsection*{Author Contributions and Conflicts of Interests}
All authors contributed equally to this work.
The authors declare that they have no conflicts of interest.
\section{Introduction}

We consider the \emph{densest submatrix problem (DSM)}, in which one aims to identify the submatrix of given size containing a maximum number of nonzero entries; see \cite{lanciano2024survey} for a survey of the historical and recent study of the DSM.
The densest submatrix problem can be considered a generalization of the \emph{densest bipartite subgraph problem}:
given a bipartite graph \( G = (V, U, E) \) and integers \( k_1 \) and \( k_2 \), the objective is to extract a subgraph induced by \( k_1 \) nodes from \( V \) and \( k_2 \) nodes from \( U \) that maximizes the number of edges. This further generalizes the \emph{densest subgraph problem}: given graph $G$ and integer $k$, find the $k$-node subgraph containing maximum number of edges. 

The problem of discovering dense subgraphs has garnered significant attention due to its broad applicability in domains such as bioinformatics~\cite{malod2010maximum, gomez2008identification}, financial and economic network analysis~\cite{boginski2006mining}, social network mining~\cite{balasundaram2011clique, pattillo2011clique}, telecommunications~\cite{jain2003impact}, clustering~\cite{schaeffer2007graph}, computational biology, and image processing. In many such contexts, densely connected substructures often correspond to functionally meaningful modules, such as communities in social graphs or coherent biological or transactional clusters.

From a computational standpoint, the densest subgraph problem of fixed size is known to be NP-hard~\cite{karp2009reducibility} and is considered difficult to approximate within any constant factor in general graphs~\cite{andersen2007finding}. This intractability arises from its close relationship to the maximum clique problem, a classical NP-hard problem. In fact, it has been shown that unless certain complexity-theoretic assumptions fail, no polynomial time approximation schemes (PTAS) exist for the densest \( k \)-subgraph problem~\cite{feige2002relations, khot2006ruling, alon1998finding, alon2011inapproximability}.

Despite this worst-case intractability, there is a growing body of literature establishing sufficient conditions for exact solution of the densest submatrix problem, or special cases such as the maximum clique and densest subgraph problem, via spectral methods and/or convex relaxation~\cite{ames2015guaranteed, candes2011robust, chen2014statistical, doan2013finding, ames2011nuclear, feige2000finding, bombina2020convex, sohn2025sharp, dadon2024detection, omer2023harnessing,omer2025maximum}. Most of these results establish that if the given problem instance corresponds to a single dense submatrix hidden by sparse noise, then this hidden submatrix is the exact solution of some convex optimization problem.

Our proposed convex relaxation adopts a nuclear norm minimization strategy, analogous to that used in robust PCA (RPCA)~\cite{chandrasekaran2011rank}, to capture low-rank structure while tolerating sparse noise. This adaptation enables recovery guarantees in planted submatrix settings that generalize and improve upon earlier works.
For example,
Ames and Vavasis~\cite{ames2011nuclear} proposed a nuclear norm relaxation for the maximum-edge biclique problem, providing recovery guarantees based on Chernoff bounds and spectral norm analysis. Ames~\cite{ames2015guaranteed} further improved this approach by combining nuclear norm minimization with an $\ell_1$-penalty. Their model decomposes the adjacency matrix into low-rank and sparse components, linking the formulation to RPCA. 
Bombina and Ames~\cite{bombina2020convex} established tight recovery guarantees for planted dense $m \times n$ submatrices under constant signal assumptions, proving that perfect recovery is possible when $\min\{m,n\} \ge \Omega((\log N)^{3/2})$ under sparse noise. They also demonstrated phase transition phenomena for convex formulations. However, their analysis was restricted to homogeneous or block-constant models and did not account for general heterogeneous structures where each block has its own distinct probability distribution.

A key common factor of these earlier results is the foundational assumption that the input matrix contains a single dense hidden submatrix. 
In contrast, we generalize these results to a more realistic setting where the input matrix may contain \emph{many dense submatrices}.
Specifically,
we consider a heterogeneous rectangular stochastic block model, where each block is governed by its Bernoulli probability distribution. Specifically, the observed binary matrix \( \A \in \{0,1\}^{M \times N} \) is constructed by partitioning the row indices \( \{1, \dots, M\} \) into \( k \) disjoint subsets \( U_1, \dots, U_{k} \) and the column indices \( \{1, \dots, N\} \) into \( k \) disjoint subsets \( V_1, \dots, V_{k} \). For each block pair \( (U_r, V_s) \), the entries \( A_{ij} \) with \( i \in U_r \) and \( j \in V_s \) are independently sampled from a Bernoulli distribution with parameter \( p_{rs} \), defining the local edge probability of that rectangular block.

\subsection{Contributions}

The primary contribution of our work is a sufficient condition for identification of a planted dense submatrix hidden among potentially many other dense submatrices.
We will see that this result can be specialized to provide sufficient conditions for identification of planted cliques, bicliques, and dense subgraphs in graphs containing many such combinatorial objects.
In all cases, we will recover the hidden submatrix or subgraph from the solution of a convex relaxation of the densest $m\times n$-submatrix problem, which we will formally define next.

\begin{definition}[Densest Submatrix Problem]
Given binary matrix $\A \in \{0,1\}^{M\times N}$ and integers $m \in \{1,2,\dots, M\}$ and $n \in \{1,2,\dots, N\}$, the \emph{densest $m\times n$-submatrix problem} seeks the submatrix of $\A$ with $m$ rows and $n$ columns that contains the \emph{most non-zero entries}.    
\end{definition}

Our main theorem will provide sufficient conditions under which exact recovery of the planted structure is guaranteed with high probability as defined below.

\begin{remark}[Interpretation of “with high probability”]
In this work, the phrase “with high probability” (abbreviated as \emph{w.h.p.}) refers to events whose probability converges to 1 as the problem size increases. Formally, for a sequence of random events \( \{E_n\} \), we say that \( E_n \) occurs with high probability if
$
\Pr(E_n) \to 1 
$
\emph{polynomially} as $n\to \infty$ in the sense that
$\Pr(E_n) \ge 1 - 1/p(n)$ for some polynomial function $p(n)$ with $p(n) \to \infty$ as $n\to\infty$.
\end{remark}


The planted dense matrix problem is NP-Hard by reduction from the maximum clique problem. As such, we should not expect to be able to solve it efficiently in all problem instances.
We next define a probabilistic model for random matrices with different expected densities across different blocks.
We will use this block structure and these densities to characterize ``nice" matrices, for which we should expect to be able to identify their dense submatrices.

\begin{definition}[Planted Submatrix Model]
\label{def:psm}
Consider partitions $(U_1, U_2, \dots, U_{k})$ and $(V_1, V_2, \dots, V_{k})$ of $\{1, 2,\dots, M\}$ and $\{1,2,\dots, N\}$ respectively.
We say that matrix $\A \in \{0,1\}^{M\times N}$ is sampled from the \emph{planted submatrix model} corresponding to this partition if each entry of $\A$ is independently sampled according to the Bernoulli probability distribution
\[\Pr(A_{ij} = 1) = p_{rs}, \hspace{0.25in}
    \Pr(A_{ij} = 0) = 1 - p_{rs}\]
for each $i \in U_r$, $j\in V_s$, $r\in \{1,2, \dots, k\}$ and $s \in \{1,2, \dots, k\}$ 
for probability matrix $\boldsymbol{P} = [p_{rs}] \in \mathbb{R}^{k \times k}$ with values in the interval $[0,1]$.
\end{definition}

We next state a simplified version of the planted submatrix model, which we will use to present a simplification of our main recovery guarantee; we will present a more detailed analysis and sufficient conditions for recovery in \cref{sec:DSM}.

\begin{definition}[Balanced Planted Submatrix Model]
We say that square matrix $\A\in\{0,1\}^{N\times N}$ is sampled from the \emph{balanced planted submatrix model} if it is sampled from the planted submatrix model with partitions with square blocks of the same size:
$$
    m_1 = m_2 = \cdots = m_k = n_1 = n_2 = \cdots = n_k.
$$
\end{definition}

Note that both the planted submatrix model and the balanced planted submatrix model are generalizations of the stochastic block model (SBM), which is commonly used in the analysis of clustering and community detection algorithms; see \cite{abbe2023communitydetectionstochasticblock} for a detailed discussion of clustering and community detection under the SBM. 
Under this model, the binary matrix \( \A \in \{0,1\}^{M \times N} \) is structured into rectangular blocks, where each block \( \A(U_r, V_s) \) is independently sampled from an Erdös–Rényi distribution with edge probability \( p_{rs} \). This results in a heterogeneous block structure in which each block may follow a different connectivity regime.

We illustrate the planted submatrix and balanced submatrix models in \cref{fig:heterogeneous_planted_example}.
We sample two matrices, one balanced and the other unbalanced, with $6$ especially dense submatrices in the diagonal blocks. 

\begin{figure}[t]
    \centering
    \subfloat[$\boldsymbol{A}$ sampled from the \emph{planted submatrix model}.]{
        \includegraphics[width=0.45\linewidth]{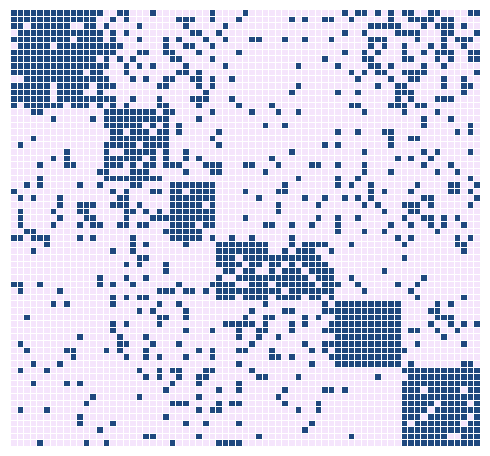}
    \label{fig:generated_bicliques}   }
    \hfill
    \subfloat[$\boldsymbol{B}$ sampled from the \emph{balanced planted submatrix model}.]{
        \includegraphics[width=0.45\linewidth]{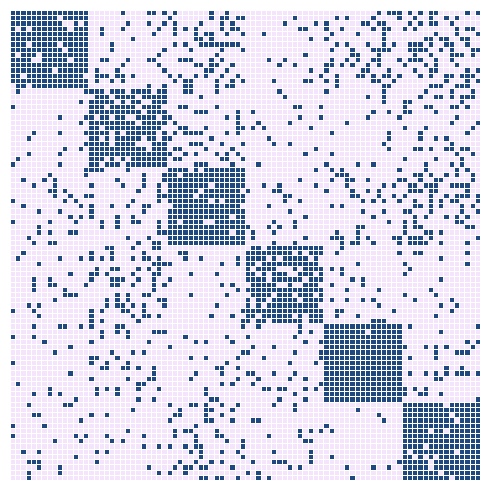}       \label{fig:generated_balancedbicliques}}
    \caption{Examples of matrices $\boldsymbol{A}$ and $\boldsymbol{B}$ sampled from the planted submatrix and balanced planted submatrix models with heterogeneous probabilities $\{p_{rs}\}$.}
    \label{fig:heterogeneous_planted_example}
\end{figure}

We are now ready to state a simplified version of our key contribution, which provides a sufficient condition on the parameters in the \emph{balanced planted submatrix} that ensures recovery of the densest submatrix via the solution of a tractable convex program.

\begin{theorem}[Perfect Recovery in the Balanced Planted Submatrix Model] \label{Theorem recov}
Let $\A \in \mathbb{R}^{M \times N}$ be sampled from the \emph{balanced planted submatrix} model with row partitions $(U_1, \ldots, U_{k})$ and column partitions $(V_1, \ldots, V_{k})$. Let $m := |U_i| = |V_i|$ for all $i=1,2,\dots, k.$ Let $p_{rs}$ denote the probability that $A_{ij} = 1$ for $i \in U_r$, $j \in V_s$.

Assume that $(U_1, V_1)$ is the planted block of interest. Define the worst-case variance proxy:
\[
\tilde{\sigma}^2 := \frac{p^*}{1-p^*} \text{ where } p^*:= \max_{(r,s)\neq (1,1)} p_{rs}.
\]
Then there exists constant $c > 0$ such that 
\begin{itemize} 
    \item the planted block $(U_1, V_1)$ is the unique densest $m \times m$ submatrix of $\A$ with high probability; and
    \item we can identify $(U_1, V_1)$ in polynomial-time using the solution of a convex program with high probability
\end{itemize}
if the following conditions hold
\begin{subequations}
\begin{align}
(\log N)^3 &\le m^2 \\
    p_{11} - p^* &\ge c \cdot \max \left\{
        \sqrt{ \frac{\tilde\sigma^2 N \log N}{m^2}
            },
        \sqrt{\frac{\sigma_{11}^2 \log N}{m}}, 
        \sqrt{\frac{\max\{\tilde\sigma^2, 1\}(\log N)^3}{m}}
    \right\}. \label{eq:balanced-gap}
\end{align}
\end{subequations}
\end{theorem}


Condition~\eqref{eq:balanced-gap} serves as a signal-to-noise ratio (SNR) balancing the signal, measured by the gap between $p_{11}$ and $p^*$, relative to noise, measured by the variances $\tilde\sigma^2_*$ and $\sigma_{11}^2$, normalized by the relative gap in the size of the planted blocks, $m^2$, and the size of the matrix given by $N$. If this SNR is sufficiently large, then we can expect to find the densest $m\times m$-submatrix in polynomial-time.

Here, we denote $(U_1, V_1)$ as the indices of the block of interest for notational simplicity and consistency through the rest of the manuscript. We assume that the planted block is indexed by $U_1, V_1$ without loss of generality, i.e., the result holds if $(U_i, V_j)$ indexes a planted block satisfying~\eqref{eq:balanced-gap} (with $p^*$ redefined as the maximum probability over all blocks except the $(i,j)$ block), for any $i,j \in \{1,2,\dots, k\}$. In this case, $(U_i, V_j)$ indexes the densest submatrix of $\A$ with high probability.

This result generalizes to the \emph{unbalanced} planted submatrix model, as well as adversarially constructed matrices. Note that Theorem~\ref{Theorem recov} omits an explicit description of the family of convex relaxations that yield the densest submatrix w.h.p.~under the hypothesis of Theorem~\ref{Theorem recov}. We will derive the relevant convex relaxation and provide these generalized sufficient conditions in the following section, including a range of regularization parameters for which the relaxation is exact. Specifically, we provide a generalization of Theorem~\ref{Theorem recov} in Theorem~\ref{thm:suff-cond-random}. A corresponding result for adversarially generated noise is found in Theorem~\ref{thm:suff-cond-adv}.

\section{The Densest Submatrix Problem}
\label{sec:DSM}

\subsection{Formulation and Relaxation}

We approach the challenge of the intractability of the general densest $m \times n$-submatrix problem by relaxing the problem as a convex program. While we do not claim that the relaxation yields globally optimal solutions for all input matrices, we  demonstrate that our convex formulation is capable of exact recovery for structured matrices generated under the planted submatrix model. In particular, we show that the proposed model recovers the planted $m \times n$ block even when it is obscured by noise in the form of spurious nonzero entries or missing signal entries elsewhere in the matrix.

We define $[M]:=\{1, 2,...,M\}$ and $[N]:=\{1, 2,...,N\}$ for every positive integer $M$ and $N$ respectively. Let $\bar{U}_1 \subseteq [M]$ and $\bar{V}_1 \subseteq [N]$ denote the row and column indices of a submatrix of $\A$ such that $|\bar{U}_1| = m$ and $|\bar{V}_1|=n$. Furthermore, we define vectors $\bar{\u} \in \{0,1\}^{M}$ and $\bar{\v} \in \{0,1\}^{N}$ as the indicator vectors of $\bar{U}_1$ and $\bar{V}_1$ respectively. The support of this submatrix is given by the nonzero entries of the  matrix $\bar{\X} = \bar{\u} \bar{\v}^\top \in \{0,1\}^{M \times N}$. By construction, $\bar{\X} = \bar{\u}\bar{\v}^\top$ is a rank-one binary matrix supported on $\bar{U}_{1}\times\bar{V}_{1}$, such that $\bar{X}_{ij} = 1$ if $(i,j) \in \bar{U}_{1} \times \bar{V}_{1}$ and $\bar{X}_{ij} = 0$ otherwise.

Suppose that $\Omega$ encodes the set of zero entries in a given matrix $\A \in \mathbb{R}^{M\times N}$. Without loss of generality, we may assume that $\A$'s entries are binary. If not, we can substitute $\A$ with a binary matrix that features the same sparsity pattern, preserving the index set of the densest $m \times n$-submatrix. Our objective is to find a rank-one matrix $\X$ with $m n$ nonzero entries that minimizes the number of disagreements with $\A$ on $\Omega$.

The observed binary matrix $\A \in \{0,1\}^{M \times N}$ may differ from $\bar{\X}$ due to noise or missing entries. 
We define the disagreement matrix $\bar{\Y} = P_\Omega(\bar{\X})$ as the projection of $\bar{\X}$ onto the set $\Omega$, i.e., a binary matrix supported only on indices where $\bar{\X}$ and $\A$ disagree due to $A_{ij} = 0$ while $\bar{X}_{ij} = 1$.
This leads to the representation
\[
\tilde{\A}_{\bar{U}, \bar{V}} = \bar{\X} - P_\Omega(\bar{\X}),
\]
where $\tilde{\A}_{\bar{U}, \bar{V}}$ can be interpreted as a noise-corrected approximation of the true submatrix in $\A$. The pair $(\bar{\X}, \bar{\Y})$ constitutes the matrix representation of the recovered submatrix and its disagreement pattern.

The density of the submatrix indexed by $(\bar{U}_{1}, \bar{V}_{1})$ is given by
\[
d(\bar{\X}) = \frac{1}{m n} \left( m n- \| P_\Omega(\bar{\X}) \|_0 \right),
\]
where $\| P_\Omega(\bar{\X}) \|_0$ counts the number of  locations in $\bar{\X}$ that lie in $\Omega$, where $\A$ reports a zero.
Since $P_\Omega(\bar{\X})$ is binary, its support size can be equivalently written as:
\[
\| P_\Omega(\bar{\X}) \|_0 = \sum_{(i,j) \in \Omega} [\bar{\X}]_{ij}.
\]
This allows us to cast the planted submatrix recovery problem as a structured rank-one optimization task with explicit modeling of disagreement against the zero-pattern of the data:
\begin{equation} \label{eqn3.1}
\begin{array}{cl}
    \displaystyle\min_{\X,\Y \in\{0, 1\}^{M\times N} } & \tr(\Y\e\e^{T})  \\ 
    \suchthat & \tr(\X\e\e^{T}) = m n, ~ P_{\Omega}(\X - \Y) = \0,~ \rank (\X) = 1.
\end{array}
\end{equation}
 By penalizing the number of mismatches through the matrix $\Y$, the problem~\eqref{eqn3.1} seeks a rank-one submatrix that minimizes disagreements between $\A$ and $\X$ on $\Omega$ while preserving the target support size $m_1 n_1$.

We move the rank constraint to the objective and relax using the nuclear norm \cite{recht2010guaranteed}, and then relax binary constraints using appropriate box constraints. This yields the optimization problem
\begin{equation} \label{eqn3.2}
\begin{array}{cl}
     \min &\lVert \X \rVert_{*} +  \gamma \tr(\Y\e\e^{T})  \\
     \suchthat &\tr(\X\e\e^{T}) = mn, ~ P_{\Omega}(\X - \Y) = \0,~  \0 \le \X \le \e\e^{T}, ~ \0 \le \Y.
     \end{array}
\end{equation}
Here, $\gamma \ge 0$ denotes a regularization parameter selected to balance between the two objectives. and $\|\X\|_*$ denotes the nuclear norm of matrix $\X$, which is equal to the sum of the singular values of $\X$. 
The relaxation~\eqref{eqn3.2} is convex since it is the minimization of the sum of a norm and linear function subject to linear constraints. Moreover,~\eqref{eqn3.2} satisfies Slater's condition; indeed, $\X = \Y = m_1 n_1/(M N) \e\e^T$ is strictly feasible for~\eqref{eqn3.2}. Moreover, \eqref{eqn3.2} can be rewritten as a semidefinite program using the argument presented in~\cite{recht2010guaranteed}. Thus,~\eqref{eqn3.2} can be solved using standard algorithms for semidefinite programming such as interior point methods; we consider a first-order method for solution of~\eqref{eqn3.2} in Section~\ref{sec:ADMM}.

As previously mentioned, it is reasonable to anticipate retrieving the solution of \eqref{eqn3.1} from \eqref{eqn3.2} when matrix $\A$ consists of a dense $m_1\times n_1$ submatrix hidden among perhaps many other dense submatrices under certain assumptions about the size and density of these submatrices.
In particular, this program succeeds in recovering the planted block exactly when embedded in a noisy binary matrix under the conditions of Theorem~\ref{Theorem recov}.

\subsection{A Sufficient Condition for Perfect Recovery for Random Matrices}
\label{sec:suff-cond-random}

\cref{Theorem recov} provided a sufficient condition for recovery of planted submatrices in a random graph sampled from the \emph{balanced planted submatrix model}. We now state the generalization of this sufficient condition for the planted submatrix model with (possibly) unbalanced submatrix sizes.

\begin{theorem}\label{thm:suff-cond-random}
Let $\A \in \mathbb{R}^{M \times N}$ be sampled from the planted submatrix model with row partitions $(U_1, \ldots, U_{k})$ and column partitions $(V_1, \ldots, V_{k})$. Let $m_i = |U_i|$ and $n_i = |V_i|$, and let $p_{rs}$ denote the probability that $A_{ij} = 1$ for $i \in U_r$, $j \in V_s$.

Assume that $(U_1, V_1)$ is the planted block of interest and let $\u_1$, $\v_1$ denote the characteristic vectors of $U_1$ and $V_1$.
Define $N^* = \max\{M, N\}$, $n_{j*} = \min\{m_j, n_j\}$, and the worst-case variance proxy:
\vspace{-4mm}
\begin{align*}
\tilde{\sigma}^2 := \max_{(r,s)\neq (1,1)} \left\{ \frac{p_{rs}}{1-p_{rs}}  \right\}.
\end{align*}
Then there exist constants $c_1, c_2 > 0$ such that the planted block $(U_1, V_1)$ is the unique densest $m_1 \times n_1$ submatrix of $\A$
and $\X^* = \u_1 \v_1^T$ is the unique optimal solution of the convex relaxation~\eqref{eqn3.2} with high probability if the following conditions hold for $(r,s) \neq (1,1)$:
\begin{subequations}
\begin{align}
    m_1 n_1 &\le m_r n_s, \label{eq:rec-size}\\
    (\log N)^3 &\le \min\{m_r^2,n_s^2\}  \label{eq:rec-size2}\\
     p_{11} - p^* & \ge   \max \left\{
        \sqrt{\frac{\tilde\sigma^2 N^* \log N^*}{m_1 n_1}}, 
        \sqrt{\frac{\sigma_{11}^2 \log N^*}{n_{1*}}}, 
        \sqrt{\frac{\max\{\tilde\sigma^2, 1\} (\log N^*)^3}{n_{1*}}} 
    \right\},         \label{eq:rec-gap}
\end{align}
\end{subequations}
and the regularization parameter $\gamma$ in~\eqref{eqn3.2} belongs to the interval
\begin{equation}\label{eq:gamma-range}
\gamma \in \left(
\frac{2}{(p_{11} - p^*) \sqrt{m_1 n_1}},
\frac{c_2}{(p_{11} - p^*) \sqrt{m_1 n_1}}
\right), 
\end{equation}
where $p^* :=\max_{(r,s)\neq (1,1)} p_{rs}$.
\end{theorem}

This result shows that recovery depends critically on a separation condition between the planted density $p_{11}$ and the background noise levels $p_{rs}$. The condition balances statistical fluctuations, structural asymmetry, and problem size. When it holds, solving the convex program in~\eqref{eqn3.2} results in a solution $\X$ whose support exactly matches the planted submatrix. 
It is easy to see that \cref{thm:suff-cond-random} specializes to \cref{Theorem recov} when $\A$ is sampled from the balanced planted submatrix model.

\begin{remark}\label{rmk:area-constraint}
    The assumption that $m_1 n_1 \le m_{r} n_s$ for all $(r,s) \neq (1,1)$ is made to simplify both the statement of the sufficient conditions for perfect recovery and their proof. It is reasonable to expect perfect recovery when $(U_1, V_1)$ is sufficiently more dense in expectation than any other $m_1 \times n_1$ block of $\A$ and $m_1 \times n_1$ is sufficiently large. Assumptions~\eqref{eq:rec-size} and~\eqref{eq:rec-gap} characterize when this is satisfied in Theorem~\ref{thm:suff-cond-random}.
    An analogous result holds when we replace Assumption~\eqref{eq:rec-size} with the assumption that $p_{11}$ exceeds the average expected density of any union of blocks indexing a submatrix with at least $m_1n_1$ entries of $\A$, albeit with far more complicated formal statement and notation than that given in Theorem~\ref{thm:suff-cond-random}. We give a deterministic characterization of this more general sufficient condition in the next section.
\end{remark}

\subsection{A Sufficient Condition for Perfect Recovery for Adversarially Generated Matrices}
\label{sec:suff-cond-det}
The sufficient conditions for perfect recovery given in Section~\ref{sec:suff-cond-random} may be overly conservative. As noted in Remark~\ref{rmk:area-constraint}, the assumption that the planted submatrix of interest is smaller than all other blocks in the generative model is kept largely to ensure that this is the densest $m_1\times n_1$-submatrix with high probability if the hypothesis of \cref{thm:suff-cond-random} is satisfied.
This simplifies the statement of the sufficient conditions, as well as their proof, but is not necessary, in general, to ensure perfect recovery. In this section, we provide a sufficient condition for recovery that holds for \emph{deterministically} generated graphs. This condition deterministically encodes the relationship between planted submatrices that ensures perfect recovery in the random case, i.e., that $(U_1,V_1)$ is sufficiently more dense than all other $m_1\times n_1$ subgraphs.

\subsubsection{The Adversarial Planted Submatrix Model}
\label{sec:adv-psm}

We start by describing an adversarial analogue of the planted submatrix model (\cref{def:psm}):
\begin{itemize}
\item 
Partition $[M]$ and $[N]$ into $k$ blocks given by $(U_1, U_2, \dots, U_k)$ and $(V_1, V_2, \dots, V_k)$, respectively.
Let $m_i := |U_i|$ and $n_i := |V_i|$ for each $i=1,2,\dots, k$.
We'll associate these with a partition of the rows and columns of binary matrix $\A \in \mathbb{R}^{M\times N}$.
\item 
Pick one block, indexed by $(U_1, V_1)$ without loss of generality,  and set $A_{ij} =1$ for all $i \in U_1, j \in V_1$.
\item
An adversary is then allowed to do the following:
\begin{itemize}
    \item Delete, i.e., set equal to $0$, up to $(1-\td) m_1$ entries in the $j$th column of $\A[U_1, V_1]$ for each $j \in V_1$. Note that this implies that at least $\td m_1$ entries of each column of $\A[U_1, V_1]$ remain equal to $1$.    
    \item Delete up to $(1 - \td) n_1$ entries in the $i$th row of $\A[U_1, V_1]$ for each $i \in U_1$.
    \item In total, the adversary deletes $\bar{r}_{11}$ entries in $\A[U_1, V_1]$; note that $r_{11} := m_1 n_1 - \bar r_{11}$ entries of $\A[U_1, V_1]$ are equal to $1$.
    \item Add, i.e., set equal to $1$, up to $r_{ii}$ nonzero entries in each diagonal block $\A[U_i, V_i]$, $i=2,3,\dots, k$.
    \item Add up to $\delta m_1$ nonzero entries in the $j$th column of $\A[U_1, V-V_1]$ for each $j \in V-V_1$, with at most $r_1$ nonzero entries in $\A[U_1, V-V_1]$.
    \item Add up to $\delta n_1$ nonzero entries in the $i$th row of $\A[U-U_1, V_1]$ for each $i \in U-U_1$, with at most $r_2$ nonzero entries in $\A[U-U_1, V_1]$.
    \item Add up to $r_3$ nonzero entries among $\{(i,j): i \in U_q, j \in V_s, \text{such that } q, s \ge 2, q \neq s\}$.
\end{itemize}
\end{itemize}

This construction models a two-person game where an adversary attempts to hide a planted all-ones matrix by a combination of creation of other dense blocks and setting entries within the planted matrix equal to $0$.

\subsubsection{Perfect Recovery Under the Adversarial Planted Submatrix Model}
\label{sec:adv-psm-rec}

It is reasonable to assume that we should have perfect recovery of the planted submatrix under the adversarial construction described in \cref{sec:adv-psm} if the adversary is allowed to introduce a limited amount of noise in the form of entry additions and deletions.
The following theorem provides a characterization of the limits on adversarial noise that ensures that the planted submatrix is recoverable as the unique solution of~\eqref{eqn3.2}.

\begin{theorem}\label{thm:suff-cond-adv}
Suppose that matrix $\A$ is constructed as described in \cref{sec:adv-psm} with planted submatrix indexed by $(U_1, V_1)$.
Let $\u_1$ and $\v_1$ denote the characteristic vectors of $U_1$ and $V_1$ respectively.
Then there exists constant $c$, depending only on $\delta, \td$ such that if 
\begin{align}
    2 \td - \delta  &> 1 \label{eq:adv-gap-cond} \\ 
    \max\{ r_1, r_2, r_3, r_{22}, r_{33}, \dots, r_{kk}, \bar{r}_{11} \} &\le c m_1 n_1 \label{eq:adv-size-cond}
\end{align}
then $(U_1, V_1)$ indexes the densest $m_1 \times n_1$-submatrix of $\A$ and $(\u_1 \v_1^T, P_\Omega(\u_1 \v_1^T))$ is the unique optimal solution of~\eqref{eqn3.2} with 
\begin{equation}
    \label{eq:adv-gamma-def}
    \gamma = \frac{1}{2\td - \delta - 1} \frac{1}{\sqrt{m_1 n_1}}.
\end{equation}
\end{theorem}

Essentially \cref{thm:suff-cond-adv} indicates that we have perfect recovery of $\A[U_1, V_1]$, unless the adversary is able to create a denser $m_1\times n_1$-submatrix through a combination of 
\begin{itemize} 
\item addition of $\Omega(m_1n_1)$ nonzero entries outside $\A[U_1, V_1]$;
\item deletion of $\Omega(m_1n_1)$ nonzero entries within $\A[U_1, V_1]$; or 
\item adding $\Omega(m_1)$ entries within a column of $\A[U_1, V - V_1]$ and/or $\Omega(n_1)$ entries within a row of $\A[U-U_1, V_1]$.
\end{itemize}

\subsection{Relationship with Existing Recovery Guarantees}
\label{sec:rec-examples}

The sufficient conditions for recovery given by Theorem~\ref{thm:suff-cond-random} and Theorem~\ref{thm:suff-cond-adv} can be shown to be extensions of several existing sufficient conditions for perfect recovery of dense subgraphs, planted cliques, and dense submatrices.
Indeed, consider the case when $k=2$ and set 
$$
    q := p_{11}, \hspace{0.25in} p = p_{12} = p_{21} = p_{22}.
$$
In this case, the sufficient condition given in Theorem~\ref{thm:suff-cond-random} immediately specializes to that given in \cite[Theorem 3.1]{bombina2020convex}. Moreover, placing further assumptions on $q$ and $p$, e.g., $q=1$, yields the recovery guarantees for the planted clique and maximum edge biclique problems proposed in~\cite{ames2011nuclear}.
Similarly, the sufficient conditions for perfect recovery under the adversarially constructed setting given in Theorem~\ref{thm:suff-cond-adv} directly specialize the sufficient conditions for perfect recovery of planted cliques and bicliques given in \cite[Section 4.1 and Section 5.1]{ames2011nuclear}.

Our results are closely related to similar recovery guarantees for biclustering under generalizations of the stochastic block model. 
Specifically, several recent results \cite{ames2014guaranteed,sudoso2025semidefinite} establish sufficient conditions under which we can find \emph{all} planted dense submatrices, corresponding to biclusters in the bipartite graph with adjacency matrix $\A$, for $\A$ sampled from the planted submatrix model;  the planted submatrix model corresponds to a special case of the stochastic block model (SBM) here.
In this context, Theorem~\ref{thm:suff-cond-random} and Theorem~\ref{thm:suff-cond-adv} give sufficient conditions for perfect recovery of a \emph{single} planted dense submatrix of $\A$ sampled from the planted submatrix model. This translates to a sufficient condition for perfect recovery of the \emph{single densest bicluster} of a given size in the biclustering setting. 
Comparing these results suggests that there are instances sampled from the planted submatrix model where the planted dense $m\times n$-submatrix corresponds to one of the biclusters in the corresponding SBM. 
In this case, biclustering should reveal the densest $m\times n$-submatrix, along with all other biclusters. In this sense, we can view the densest submatrix problem as that of identifying a single hidden dense bicluster or submatrix, among potentially many biclusters with heterogeneous edge probabilities and outliers.

The sufficient condition for solution of the $m\times n$-densest submatrix problem includes a broader range of matrices than the sufficient conditions for perfect recovery of biclusters given in~\cite{ames2014guaranteed,sudoso2025semidefinite}  in the following ways:
\begin{itemize}
    \item 
    The sufficient conditions given in Theorem~\ref{thm:suff-cond-random} hold for a much larger number of outliers than those for biclusters. Indeed, sufficient conditions for perfect biclustering under the SBM typically restrict the number of outlier nodes to $\O(\min\{\min\{m_i,n_i\}\})$, where $(m_i,n_i)$ is the size of the $i$th planted bicluster. The sufficient conditions given in Theorem~\ref{thm:suff-cond-random} ensure recovery of the densest $m\times n$-submatrix for instances $\A \in \{0,1\}^{M\times N}$ of the planted submatrix model with up to $\O(\min\{M,N\})$ outlier nodes. 
    \item 
    Moreover, our construction of the planted submatrix model potentially allows multiple relatively dense blocks in each row or column of blocks, while the stochastic block model assumes one especially dense block per row and column (typically indexed as the diagonal block).
    \item 
    Solution of~\eqref{eqn3.2} requires selection of the submatrix sizes $(m,n)$, while most biclustering methods do not need to be provided block sizes. However, many biclustering methods  require the number of biclusters $k$ as a parameter, which is not needed by~\eqref{eqn3.2}.
\end{itemize}
This implies that we are able to recover the densest $m\times n$-submatrix using the convex relaxation~\eqref{eqn3.2} in a wider range of instances $\A$ than if we identify all biclusters, including the densest $m\times n$-submatrix, in the bipartite graph with adjacency matrix $\A$.

\subsection{Identifying Multiple Planted Submatrices}
\label{sec:multiple-cliques-examples}

When a matrix contains multiple densest $m \times n$-submatrices, the sufficient conditions for perfect recovery given by Theorem~\ref{thm:suff-cond-random} and Theorem~\ref{thm:suff-cond-adv} do not hold.
However, we observe that we may still identify the densest submatrices from the solution of~\eqref{eqn3.2} in practice in Section~\ref{sec:real-data}.
In this case, we observe that any convex combination of the matrix representations of the densest submatrices is an optimal solution of~\eqref{eqn3.2}.
It seems reasonable to expect that there are sufficient conditions similar to those given in Sections~\ref{sec:suff-cond-random} and~\ref{sec:suff-cond-det} under which we should expect to identify (potentially) multiple densest submatrices with high probability. 

We can formalize this observation as the following conjecture.

\begin{conjecture}\label{conj:multiple-bicliques-conjecture}
    Let $\A \in \mathbb{R}^{M\times N}$ be sampled from the planted submatrix model with row partitions $(U_1, \dots, U_k)$ and column partitions $(V_1, \dots, V_k)$ such that 
    \begin{itemize}
    \item 
        $(U_1, V_1)$, $(U_2, V_2)$, \dots $(U_\ell, V_\ell)$ all index $m\times n$-submatrices of $\A$ with maximum density;
    \item 
        $p_{ii} \ge q$ for all $i \in \{1,2,\dots,\ell\}$ for some $q \in (0,1]$;        
    \end{itemize}
    for all $r \in \{1,2,\dots, k-1\}$.
    Let $\X_r = \u_r \v_r^T$ be the rank-one matrix representation of the diagonal block $(U_r, V_r)$ for each $r=1,2,\dots, \ell$.    
    Then there exist constants $c_1, c_2, c_3 > 0$ such if 
    \[ 
        (q - \tilde p) \min \{m,n\} \ge c_1 \max
        \left\{
            \sqrt{\tilde \sigma^2 N^* \log N^*}, 
            \sqrt{q(1-q) m^* \log N^*}, 
            (\log N^*)^{3/2}
        \right\},
    \]
    where $m^* = \max\{m,n\}$, $N^*=\max\{M,N\}$, 
    \begin{align*}
        \tilde \sigma^2 &:= \max 
            \left\{ \frac{p_{rs}}{1-p_{rs}}: (r,s) \notin \{(1,1), (2,2), \dots, (\ell, \ell) 
            \right\}, \text{ and }
    \\
        \tilde p &:= \max 
            \Big\{ p_{rs}: (r,s) \notin \{(1,1), (2,2), \dots, (\ell, \ell) 
            \Big\},
    \end{align*}
    then any convex combination of $\X_i = \u_i \v_i^T$,
    $$
        \sum_{i=1}^\ell \lambda_i \X_i = \sum_{i=1}^\ell \lambda_i \u_i \v_i^T,
    $$
    where $0\le \lambda_i \le 1$ and $\sum \lambda_i = 1$,
    is optimal for~\eqref{eqn3.2} with high probability for
    regularization parameter $\gamma$ satisfying 
    \[    
     \frac{c_2}{(q -\tilde p) \sqrt{mn}} < \gamma 
        <  \frac{c_3}{(q -\tilde p) \sqrt{mn}}.
    \]
\end{conjecture}

\section{Derivation of the Recovery Guarantees in the Random Case}
\label{sec:random-proof}

The recovery guarantees established in this section are grounded in a rigorous duality-based analysis of the convex relaxation formulated in \eqref{eqn3.2}. In particular, we seek to characterize the conditions under which the solutions \( \boldsymbol{\bar{X}} = \bu \bv^{T} \) and \( \bar{\Y} = P_{\Omega}(\bar{\X}) \) not only satisfy feasibility but also exhibit optimality in the context of recovering the planted densest \( (m_1, n_1) \)-subgraph. This is achieved by constructing an explicit dual certificate that satisfies the Karush-Kuhn-Tucker (KKT) conditions for~\eqref{eqn3.2} given in the following theorem.


\subsection{Optimality Conditions} \label{sec: Optimality Conditions}

\begin{theorem}\label{thm:KKT}
The solutions \( \boldsymbol{\bar{X}} = \bu \bv^{T} \) and  \( \bar{\Y} = P_{\Omega}(\bar{\X}) \) are the unique optimal solutions of \eqref{eqn3.2}, where \( \bu \) and \( \bv \) represent the characteristic vectors of \( U_1 \) and \( V_1 \), respectively, if and only if there exist dual multipliers \( \lambda \geq 0 \), \( \boldsymbol{\Lambda} \in \mathbb{R}_{+}^{M \times N} \), \( \BX \in \mathbb{R}_{+}^{M \times N} \), and \( \W \in \mathbb{R}^{M \times N} \) satisfying the following conditions:
\begin{subequations}\label{eq: KKT}
\begin{align}
    \frac{\boldsymbol{\bar {u}} \boldsymbol{\bar v}^T}  {\sqrt{m_{1}n_{1}}} + \W - \lambda \e\e^T 
            + \gamma \e\e^T - \BX + \BL     &= \0 \label{8a}\\ 
    \Xi_{ij} &= \gamma \;\; \text{for all } (i,j) \notin \Omega\label{8e} \\
    \tr (\BL^{T}(\bX - \e\e^T)) &= 0 \label{8b}\\ 
    \tr (\BX^{T}\bY) &=0 \label{8c}\\ 
    \W^{T}\boldsymbol{\bar u} = \0, \quad \W\boldsymbol{\bar v} = \0, \quad \| \W \| &< 1. \label{8d}
\end{align} 
\end{subequations}
\end{theorem}

Here, $\lambda$, $\BX$, $\BL$ are Lagrange multipliers corresponding to the constraints $\tr(\X\e\e^T) = mn$, $\Y\ge 0$, and $\X \le 1$, respectively, in~\eqref{eqn3.2}; the remaining constraints in~\eqref{eqn3.2} and the corresponding Lagrange multipliers can be eliminated when formulating and applying the KKT conditions at $(\bu, \bv)$. Conditions~\eqref{8a} and~\eqref{8e} are the stationary conditions  with respect to $\X$ and $\Y$.
We note that the subdifferential of the nuclear norm at $\bu \bv^T$ is given by 
$$
    \partial \|\bu \bv^T \|_* = \left\{ 
        \frac{\boldsymbol{\bar {u}} \boldsymbol{\bar v}^T}  {\sqrt{m_{1}n_{1}}} + \W : 
        \W^{T}\boldsymbol{\bar u} = \0, \quad \W\boldsymbol{\bar v} = \0, \quad \| \W \| \le 1
        \right\},
$$
which corresponds to the first term in the stationarity condition~\eqref{8a} and the inequalities~\eqref{8d}.
We next consider a particular choice of these Lagrange multipliers and subgradient, and then determine conditions ensuring that the system of equations~\eqref{eq: KKT} is satisfied.

\subsection{A Proposed Choice of Lagrange Multipliers} \label{sec: multipliers}

Theorem \ref{thm:KKT} presents a specialization of the Karush-Kuhn-Tucker (KKT) optimality conditions for the convex program~\eqref{eqn3.2} (see Section 5.5.3 of \cite{boyd2004convex}). Our derivation of sufficient conditions for perfect recovery relies on solving the nonlinear system~\eqref{eq: KKT}, which ensures both optimality and uniqueness of the solution given by Theorem~\ref{thm:KKT}. This solution corresponds to the planted submatrix with rows and columns indexed by \(U_1\) and \(V_1\), respectively. Let \(\bu\) and \(\bv\) denote the characteristic vectors of \(\bar U_1\) and \(\bar V_1\).

Under the assumptions of Theorem~\ref{thm:suff-cond-random}, we show that the solutions \( \bar{\X} = \bu \bv^{T} \) and \( \bar{\Y} = P_{\Omega}(\bar{\X}) \) are the unique optimal solutions of \eqref{eqn3.2}, and that the submatrix given by the index sets \( (U_1, V_1) \) is the densest submatrix. Moreover, we demonstrate that \( \bar{\X} = \bu \bv^{T} \) and \( \bar{\Y} = P_{\Omega}(\bar{\X}) \) satisfy the sufficient conditions for optimality derived from the KKT framework.

The analysis is grounded in tools from convex optimization, spectral graph theory, and probabilistic analysis/sharp nonasymptotic bounds on the norm of random matrices with independent entries \cite{bandeira2016sharp, boucheron2013concentration}. By examining the primal-dual optimality conditions, we provide a concrete method to verify when a rank-one solution of the form \( \X = \u\v^{T} \), where \( \u, \v \) corresponds to the characteristic vectors of the planted submatrix, satisfies the optimality criteria. The core idea is to carefully construct the dual variables to satisfy the complementary slackness and stationarity conditions of the KKT system while ensuring the feasibility of both primal and dual solutions.

This approach allows us to establish precise recovery guarantees for the combinatorial problem of submatrix localization under partial observation. By bridging the discrete structural properties of the graph and continuous convex relaxations, we provide a robust theoretical foundation for the proposed algorithmic framework under the probabilistic model.


\subsubsection{Choice of $\W$ and $\BX$}
\label{sec: W cases}

We can construct a solution of the KKT conditions~\eqref{eq: KKT} by constructing the matrices $\W$ and $\BX$ according to the  following cases.

\begin{itemize}
    \item 
        \textbf{Case 1:} If $(i,j) \in {\bar U}_{1} \times {\bar V}_{1} - \Omega$, then~\eqref{8a} indicates that
        \begin{align*}
            W_{ij} & = \lambda - \gamma - \frac{1}{\sqrt{m_{1}n_{1}}} + \Xi_{ij} -  \Lambda_{ij} \\
                 &= \lambda - \frac{1}{\sqrt{m_{1}n_{1}}}  -  \Lambda_{ij} \\
                 &= \Bar{\lambda} -  \Lambda_{ij},
        \end{align*}
        where $\Bar{\lambda} :=  \lambda - \frac{1}{\sqrt{m_{1}n_{1}}}$, 
        if we choose 
        $
            \Xi_{ij} = \gamma.
        $

    \item 
        \textbf{Case 2:} If $i \in U_1$, $j \in V_1$ and  $(i,j) \in \Omega $, then
        \begin{equation*}
            X_{ij} = Y_{ij} =  \frac{1}{ \sqrt{m_{1}n_{1}}}.
        \end{equation*}
        The complementary slackness condition \eqref{8c} is only satisfied in this case if 
        $
            \Xi_{ij} = 0.
        $
        In this case,
        \begin{equation*}
            W_{ij}  = \Bar{\lambda} - \gamma  -\Lambda_{ij}.
        \end{equation*}
    \item 
        \textbf{Case 3:} Suppose that 
        $i \in U_r,$ $j \in V_s$ for some $r, s \ge 2$.
        Then we choose
        \begin{equation} \label{eq: Case 3}
            W_{ij}  = \lambda \cdot
            \begin{cases}                 
                1, & \text{if } (i,j) \notin \Omega, \\
                - \frac{p_{rs}}{1 - p_{rs}},  & \text{if } (i,j) \in \Omega.
            \end{cases}
        \end{equation}
        We choose $\Xi_{ij} = W_{ij} - \lambda + \gamma$ according to~\eqref{8a} and~\eqref{8e}.
        
    \item \textbf{Case 4:}
        Suppose that $i \in U_1$, $j \in V_s$ for $s \ge 2$.
        In this case, we choose $W_{ij}$ and $\Xi_{ij}$ as follows so that~\eqref{8a} and~\eqref{8e} are satisfied:
        \begin{align} \label{eq:Case-4-W}
            W_{ij}  &= 
            \begin{cases}
                \lambda, & \text{if } (i,j) \notin \Omega, \\
                -\lambda \left( \frac{\nu_{j}}{m_1-\nu_{j}} \right), & \text{if } (i,j) \in \Omega,
            \end{cases}
            \\ 
            \Xi_{ij} &= 
            \begin{cases} 
            \gamma, & \text{if } (i,j) \notin \Omega, \\
            \gamma - \lambda \left( \frac{m_1}{m_1-\nu_{j}}\right), & \text{if } (i,j) \in \Omega,
            \end{cases}
        \end{align}
        where $m_1 := \lvert U_1 \rvert$ and $\nu_j$ denotes the number of nonzero entries of the $j$th column of $\boldsymbol{A}$ in rows indexed by $U_1$.

    \item \textbf{Case 5:}
        Similarly, if $i \in U_r$, $r \ge 2$, $j \in V_1$  we choose
        \begin{align*}
            W_{ij}  &= 
            \begin{cases}
                \lambda, & \text{if } (i,j) \notin \Omega, \\
                -\lambda \left( \frac{\mu_{i}}{n_1-\mu_{i}} \right), & \text{if } (i,j) \in \Omega,
            \end{cases}
            \\ 
            \Xi_{ij} &=
            \begin{cases} 
            \gamma, & \text{if } (i,j) \notin \Omega, \\
            \gamma - \lambda \left( \frac{n_1}{n_1-\mu_{i}}\right), & \text{if } (i,j) \in \Omega,
            \end{cases}
        \end{align*} 
        where $n_1 = \lvert V_1 \rvert$ and $\mu_i$ denotes the number of nonzero entries of the $i$th row of $\boldsymbol{A}$
        in columns indexed by $V_1$.

\end{itemize}
This exhausts all possible values of $W_{ij}$ and $\Xi_{ij}$. We will use this choice to choose the remaining Lagrange multipliers.

\subsubsection{Construction of $\BL$}
\label{sec:BL-choice}

\renewcommand{\y}{\boldsymbol{y}}
\renewcommand{\z}{\boldsymbol{z}}

By construction, the proposed Lagrange multipliers satisfy the stationarity condition~\eqref{8a} and the complementary slackness conditions~\eqref{8b} and~\eqref{8c}.
Moreover, for any $i \in U - U_1$, we have 
\begin{equation*}
    [\W \bu]_i = \sum_{j\in V_1} W_{ij} = \mu_i \lambda - (n-\mu_i) \lambda \left(\frac{\mu_i}{n_1 - \mu_i} \right) = 0.
\end{equation*}
By a symmetric argument, we have $[\W^T \bv]_j = 0$ for all $j \in V - V_1$. 

We next propose a choice of $\Lambda_{ij}$ for $i \in U_1,$ $j \in V_1$ ensuring that $[\W\bu]_i = 0$ and $[\W^T \bv]_j = 0$.
Verifying when these choices of $\BL$ and $\BX$ are nonnegative and $\|\W\| < 1$ with high probability will complete the proof.

We parametrize the (unknown) nonzero entries of $\BL$ as
\begin{equation}
    \label{eq: yz}
    \Lambda_{ij} =: y_i + z_j
\end{equation}
for each $(i,j) \in U_1 \times V_1$. That is, 
\begin{equation*}
    \BL(U_1, V_1) = \y \e^T + \e \z^T
\end{equation*}
for some $\y \in \RR^{m_1},$ $\z\in \RR^{n_1}$.
Under this parametrization of $\BL$, the conditions $\W\bu = \0$ and $\W^T\bv = \0$ are satisfied if $\y$ and $\z$ are solutions of the linear system 
\begin{equation} \label{9a}
    \left(
        \begin{array}{cc}
            n_{1} \I & \e\e^T\\ 
            \e\e^T & m_{1} \I
        \end{array}
    \right)
    \left(
        \begin{array}{c}
            \y \\
            \z
        \end{array}
    \right) = 
    \left(
        \begin{array}{c}
            -\gamma \boldsymbol{\bar \mu} + n_{1} \Bar{\lambda} \e \\
            -\gamma \boldsymbol{\bar \nu} + m_{1} \Bar{\lambda} \e
        \end{array}
    \right), 
\end{equation} 
where 
\begin{align*}
   \Bar \mu_{i} &:= n_{1} - \mu_{i}, \hspace{0.35in}
   \Bar \nu_{j} := m_{1} - \nu_{i} 
\end{align*}
for all $i \in U_1$, and $j \in V_1$.

The coefficient matrix of this linear system is 
singular with null space spanned by $(\e,-\e) \in \RR^{m_1 + n_1}.$
However, the unique solution of the nonsingular linear system 
\begin{equation} \label{9a-2}
    \left(
        \left(
            \begin{array}{cc}
                n_{1} \I & \e\e^T\\ 
                \e\e^T & m_{1} \I
            \end{array}
        \right)
        + 
        \theta
        \begin{pmatrix} \e \\ -\e \end{pmatrix}        
        \begin{pmatrix}
            \e & - \e^T
        \end{pmatrix}
    \right)            
    \left(
        \begin{array}{c}
            \y \\
            \z
        \end{array}
    \right) = 
    \left(
        \begin{array}{c}
            -\gamma \boldsymbol{\bar \mu} + n_{1} \Bar{\lambda} \e \\
            -\gamma \boldsymbol{\bar \nu} + m_{1} \Bar{\lambda} \e
        \end{array}
    \right),
\end{equation} 
obtained by perturbing~\eqref{9a} by $\theta(\e,-\e)$
is a solution of~\eqref{9a}
for all $\theta > 0$.
Choosing $\theta = 1$, the coefficient matrix of~\eqref{9a-2} becomes 
\begin{align*}
    \left(
        \begin{array}{cc}
            n_{1} \I & \e\e^T\\ 
            \e\e^T & m_{1} \I
        \end{array}
    \right) +
    \theta
    \begin{pmatrix} \e \\ -\e \end{pmatrix}        
        \begin{pmatrix}
            \e & - \e^T
        \end{pmatrix}    
    &= 
    \left(
        \begin{array}{cc}
            n_1 \I & \e\e^T\\ 
            \e\e^T & m_1 \I
        \end{array}
    \right) + 
    \theta \left(
        \begin{array}{cc}
            \e\e^T & -\e\e^T\\ 
            -\e\e^T & \e\e^T
        \end{array}
    \right)\\
    &= \left(
        \begin{array}{cc}
            n_1 \I + \theta \e\e^T & \e\e^T - \theta \e\e^T\\ 
            \e\e^T - \theta \e\e^T & m_1 \I + \theta \e\e^T
        \end{array}
    \right)\\
    &= \left(
        \begin{array}{cc}
            n_1 \I + \e\e^T & 0\\ 
            0 & m_1 \I + \e\e^T
        \end{array}
    \right).
\end{align*}
In this case,~\eqref{9a-2} reduces to
\begin{align*}
    (n_1 \I + \e\e^T) \y &= -\gamma \boldsymbol{\bar \mu} + n_{1} \Bar{\lambda} \e  \\     
     (m_1 \I + \e\e^T) \z &= -\gamma \boldsymbol{\bar \nu} + m_{1} \Bar{\lambda} \e, 
\end{align*}
which has unique solution given by 
\begin{align}
\y &= (n_1 \I + \e\e^T)^{-1} (-\gamma \boldsymbol{\bar \mu} + n_{1} \Bar{\lambda} \e) \label{9b}\\
\z &= (m_1 \I + \e\e^T)^{-1} (-\gamma \boldsymbol{\bar \nu} + m_{1} \Bar{\lambda} \e). \label{9c}
\end{align}
Applying the Sherman-Morrison-Woodbury-Formula \cite[Theorem 1.1]{deng2011generalization} and \cite[Section 2.1.4]{van1996matrix}
to Equation~\eqref{9b} yields
\begin{align}
    \y &= 
            \frac{1}{n_1}
            \left(
                \I - \frac{1}{m_1 + n_1} \e \e^T 
            \right)
            (-\gamma \boldsymbol{\bar \mu} + n_{1} \Bar{\lambda} \e)
        \notag \\ 
        &=
        \frac{1}{n_1} 
        \left(
            -\gamma \boldsymbol{\bar\mu} 
            +
            \left(
                \frac {\gamma \boldsymbol{\bar\mu}^T \e}{m_1 + n_1} 
            \right)
            \e 
            + 
            \left(
                \frac{n_1^2}{m_1 + n_1}
            \right)
            \bar\lambda \1
        \right). 
        \label{eq: y formula}
\end{align}
and
\begin{equation} \label{eq: z formula}
\z = \frac{1}{m_1} 
        \left(
            -\gamma \boldsymbol{\bar\nu} 
            +
            \left(
                \frac {\gamma \boldsymbol{\bar\nu}^T \e}{m_1 + n_1} 
            \right)
            \e 
            + 
            \left(
                \frac{m_1^2}{m_1 + n_1}
            \right)
            \bar\lambda \1
        \right). 
\end{equation}

\subsection{Nonnegativity of $\BL$}

The entries of $\bbm$ and $\bbn$ are binomial random variables corresponding to $n_{1}$ and $m_{1}$ independent Bernoulli trials, respectively, each with probability of success $1 - p_{11}$. Thus,
\begin{align*}
    \EX [\bar\mu_i] &= (1-p_{11}) n_1, & \Var[\bar\mu_i] &= p_{11}(1-p_{11}) n_1, \\ 
    \EX [\bar\nu_i] &= (1-p_{11}) m_1, & \Var[\bar\nu_i] &= p_{11}(1-p_{11}) m_1.
\end{align*}
Moreover, the number of nonzero entries in the $U_1 \times V_1$ block of $\boldsymbol{A}$ is equal to $\bbm^T \e = \bbn^T \e$. This is also binomially distributed, corresponding to $m_1 n_1$ independent Bernoulli trials with probability of success $1- p_{11}$:
\begin{align*}
    \EX [\bbm^T\e] &= (1-p_{11}) m_1 n_1, & \Var[\bbm]_i &= p_{11}(1-p_{11}) m_1 n_1.
\end{align*}
Applying linearity of expectation implies that 
\begin{align}
    \mathbf{E}[y_{i}] 
    &=   
        \frac{n_{1} \Bar{\lambda} }{m_1 + n_1} 
            + \frac{\gamma }{n_1(n_{1} + m_{1})}E[\boldsymbol{\bar \mu}^{T} \e] 
            - \frac{\gamma }{n_{1}}E[\boldsymbol{\bar \mu}]_i 
    \notag \\
    &=
        \frac{n_{1}\Bar{\lambda} e}{m_1 + n_1} 
            + \frac{\gamma m_1}{m_1 + n_1}(1 - p_{11}) 
            - \gamma(1 - p_{11})
    \notag \\
    &= 
    \left( 
        \frac{n_1}{m_1 + n_1} 
    \right)
    \left(
        \bar\lambda - \gamma (1 - p_{11})
    \right)
    \label{eq: E y}
\end{align}
for all $i \in U_1$.
Following an identical calculation, we obtain 
\begin{equation}    \label{eq: E z}
   \mathbf{E}[z_{j}] = \frac{m_{1}}{m_1 + n_1}(\Bar{\lambda} - \gamma (1 - p_{11}))
\end{equation}
for all $j \in V_1$.
For any $\tau > 0$, selecting 
\begin{equation}\label{eq: E lam 1}
\lambda = \frac{1}{\sqrt{m_{1}n_{1}}} + \gamma (1 - p_{11}) + \gamma \tau
\end{equation} 
ensures that the entries of $\BL$ are strictly positive in expectation.


Next, our focus shifts to determining $\tau$ in order to guarantee that both $\y$ and $\z$ have positive entries with high probability. To achieve this, we will repeatedly apply a the following specialization of the classical Bernstein inequality \cite[Section 2.8]{boucheron2013concentration} to bound the sum of independent Bernoulli random variables.

\begin{lemma} \label{lem4a}
    Consider a sequence of $k$ independent $\{0, 1\}$ Bernoulli random variables, denoted by $x_{1}, . . ., x_{k}$,
    each having a success probability of $\rho$. 
    Let $s = \sum_{i=1}^{k}x_{i}$ represent the binomially distributed random variable indicating the number of successes. 
    Then
    \begin{align}\label{eq: Bernstein}
            \prob\left( \lvert s - \rho k \rvert > 6\max \left\{ \sqrt{\rho(1 - \rho)k\log t}, \log t \right\}\right) \le 2t^{-6}
    \end{align}
    for all $t > 0$.
\end{lemma}

Applying Lemma \ref{lem4a} with $t = N^* = \max\{M, N\}$ and  the union bound shows that
\begin{align}
         \lvert {\bar \mu}_{i} - (1 - p_{11})n_{1}  \rvert 
            & \le 6\max \left\{ 
                    \sqrt{\sigma_{11}^{2} n_{1}\log{N^*}}, 
                    \log{N^*} 
                    \right\}, \label{eq: mu Bern}\\
        \lvert {\bar \nu}_{j} - (1 - p_{11})m_{1}  \rvert 
            & \le 6\max \left\{ 
                \sqrt{\sigma_{11}^{2}m_{1}\log{N^*}}, 
                \log {N^*} 
                \right\}, \label{eq: nu Bern} \\
        \lvert \bbm^T\e - m_1 n_1 (1-p_{11}) \rvert 
            & \le 6 \max \left\{ 
                \sqrt{\sigma_{11}^2 m_1 n_1 \log N^*}, 
                \log {N^*} 
                \right\}
                \label{eq: mu 1 Bern}
\end{align}
for all $i \in U_{1}$ and $j \in V_{1}$ with high probability.
Combining~\eqref{eq: mu Bern},~\eqref{eq: nu Bern}, and~\eqref{eq: mu 1 Bern} and applying the triangle inequality shows that
\begin{align}
    \lvert y_{i} - \mathbf{E}[y_{i}] \rvert 
    &\le 
        \frac{\gamma}{n_{1}} \left(
            \lvert (\boldsymbol{\bar \mu}_{i} - \mathbf{E}[\boldsymbol{\bar \mu}_{i}])\rvert 
            + {\frac{1}{m_1 + n_1}} \lvert  \boldsymbol{\bar \mu}^{T}e - \mathbf{E}[\boldsymbol{\bar \mu}^{T}e] \rvert\right)\nonumber\\
    &\le 
        6\gamma\max \left\{ \sqrt{\frac{\sigma_{11}^{2}\log{N^*}}{n_{1}}}, \frac{\log{N^*}}{n_{1}} \right\} \nonumber\\
        &\quad\quad+ \frac{6\gamma}{m_1 + n_1}\max\left\{
        \sqrt{\frac{\sigma_{11}^{2} m_{1}\log{N^*}}{n_{1}}}, \frac{\log{N^*}}{n_{1}} \right\}\nonumber\\
    &= 
    \O \left(
        \gamma \max 
            \left\{ 
                \sqrt{\frac{\sigma_{11}^{2}\log{N^*}}{n_{1}}}, 
                \frac{\log{N^*}}{n_{1}} 
            \right\}
        \right)
        \label{eq4.1}
\end{align}
for all $i \in U_1$ w.h.p. By a symmetric argument
\begin{align}
    \lvert z_{j} - \mathbf{E}[z_{j}] \rvert 
    & = \O \left( 
        \gamma \max 
        \left\{ \sqrt{\frac{\sigma_{11}^{2}\log{N^*}}{m_{1}}}, 
        \frac{\log{N^*}}{m_{1}} 
        \right\}
        \right)\label{eq4.2}
\end{align}
for all $j \in V_1$ w.h.p.
Substituting \eqref{eq4.1} and \eqref{eq4.2} into Equation \eqref{eq: yz} shows that
\begin{align*}
    \Lambda_{ij} 
        &= y_{i} + z_{j}\nonumber\\ 
        &\ge
            \mathbf{E}[y_{i}] - \lvert y_{i} -\mathbf{E}[y_{i}]\rvert + \mathbf{E}[z_{j}] -  \lvert z_{j} -\mathbf{E}[z_{j}]  \rvert \nonumber\\
        &= 
            \frac{n_{1}}{m_1 + n_1}(\Bar{\lambda}  - \gamma (1 - p_{11})) -\lvert y_{i} -\mathbf{E}[y_{i}]\rvert\nonumber\\ 
            &\hspace{0.5in} +
                \frac{m_{1}}{m_1 + n_1}(\Bar{\lambda} - \gamma (1 - p_{11})) -  \lvert z_{j} -\mathbf{E}[z_{j}]  \rvert \nonumber\\
        &\ge 
            \Bar{\lambda}  - \gamma (1 - p_{11}) - c_\tau \gamma \max \left\{ \sqrt{\frac{\sigma_{11}^{2}\log{N^*}}{n_{1*}}}, \frac{\log{N^*}}{n_{1*}} \right\}\nonumber\\
        &= 
            \gamma \tau - c_\tau \gamma \max \left\{ \sqrt{\frac{\sigma_{11}^{2}\log{N^*}}{n_{1*}}}, \frac{\log{N^*}}{n_{1*}}\right\}
\end{align*}
for some constant $c_\tau >0$ w.h.p.~for all $i \in U_{1}$ and $ j \in V_{1}$.
Therefore, choosing sufficiently large $\tau > 0$  ensures that $\BL$ is nonnegative w.h.p.
Specifically choosing 
\begin{align}
    \tau = c_\tau \max \left\{ \sqrt{\frac{\sigma_{11}^{2}\log{N^*}}       {n_{1*}}}, \frac{\log{N^*}}{n_{1*}}\right\} \label{eq4.3}
\end{align}
guarantees that every entry of $\BL$ is nonnegative w.h.p.

 \subsection{Nonnegativity of $\BX$}
 \label{sec:nonneg-Xi}

We next derive conditions on the regularization parameter $\gamma$ so that the entries of the dual variable $\BX$ are nonnegative with high probability.

It is essential to recall that ${\Xi}_{ij}$ takes values 0 and $\gamma$ for all $(i, j)$ except in Case 3 and the $(i,j)\in \Omega$ subcase in Cases 4 and 5 considered in the choice of $\W$ and $\BX$ in Section~\ref{sec: W cases}. It remains to establish conditions under which $\Xi_{ij} \ge 0$ w.h.p.~in these remaining cases. 

\subsubsection{Nonnegativity in Case 3}

Consider $i \in U_r$ and $j \in V_s$ for $r,s \ge 2$.


If $(i,j) \notin \Omega$, then we must have $\Xi_{ij} = \gamma > 0$ by~\eqref{8e}.
On the other hand, suppose that $(i,j) \in \Omega$.
In this case,
\begin{align*}
 \Xi_{ij} &= W_{ij} - \lambda + \gamma = \gamma - \lambda \left( \frac{p_{rs}}{1 - p_{rs}} + 1 \right) = \gamma - \frac{\lambda}{1-p_{rs}}.
\end{align*}
Substituting $\lambda = \frac{1}{\sqrt{m_1n_1}} + \gamma (1-p_{11}) + \gamma \tau$, we have 
\begin{align*}
    (1 - p_{rs}) \Xi_{ij} 
    &= \gamma ( 1- p_{rs}) - \frac{1}{\sqrt{m_1n_1}}  -\gamma (1-p_{11}) - \gamma \tau \\ 
    &= \gamma ( p_{11} - p_{rs} - \tau ) - \frac{1}{\sqrt{m_1n_1}}.
\end{align*}
Under the assumption
\begin{equation}\label{eq:tau-gap-3}
    p_{11} - p_{rs} \ge 2 \tau 
    = 2 c_\tau \max \left\{ \sqrt{\frac{\sigma_{11}^{2}\log{N^*}}       {n_{1*}}}, \frac{\log{N^*}}{n_{1*}}\right\},
\end{equation}
we have 
$$
(1 - p_{rs}) \Xi_{ij} \ge \frac{\gamma}{2}(p_{11} - p_{rs}) -  \frac{1}{\sqrt{m_1n_1}}.
$$
This implies that $\Xi_{ij} \ge 0$ if
\begin{equation}\label{eq:gap-3}
 \frac{\gamma}{2}(p_{11} - p_{rs}) \ge  \frac{1}{\sqrt{m_1n_1}}. 
\end{equation}

\subsubsection{Nonnegativity in Cases 4 and 5}

We start with Case 4 in the construction of $\W$ and $\BX$. 
Fix $i \in U_1$ and $j \in V_s$ for $s \ge 2$.
Then,
\begin{align}\label{eq4.4}
    \Xi_{ij} &= \gamma - \lambda \frac{m_1}{m_1 - \nu_j}.
\end{align}
Note that $\Xi_{ij}$ is a monotonically decreasing function of $\nu_j$ on the interval $[0, m_1)$.
Therefore, $\Xi_{ij}$ is minimized when $\nu_j$ is maximized over this interval. 
Moreover,
the Bernstein inequality implies that
\begin{align}
    \nu_j
    &\le p_{1s}m_{1} 
       + 
            6\max \left\{ \sqrt{\sigma_{p_{1s}}^{2}m_{1}\log{N^*}},         \log{N^*} \right\}. \label{eq4.5}
\end{align}
with high probability.
Substituting~\eqref{eq4.3} into~\eqref{eq4.4} and rearranging we see that $\Xi_{ij} \ge 0$ w.h.p.~if
\begin{align*}
     \gamma \left(p_{11} - p_{1s} - \tau - 6\max\left\{\sqrt{\sigma_{1s}^2 \frac{\log N^{*}}{m_1}}, \frac{\log N^{*}}{m_1} \right\} \right) &\ge \frac{1}{\sqrt{m_1 n_1}}.
\end{align*}
If 
\begin{equation} \label{eq:gap-case4}
    p_{11} - p_{1s} \ge 2 \left( \tau + 6\max\left\{\sqrt{\sigma_{1s}^2 \frac{\log N^*}{m_1}}, \frac{\log N^*}{m_1} \right\} \right),
\end{equation}
then we have $\Xi_{ij} \ge 0$ if
\begin{equation*}
    \frac{\gamma}{2} (p_{11} - p_{1s} ) \ge \frac{1}{\sqrt{m_1 n_1}}.
\end{equation*}
Rearranging shows that $\Xi_{ij} \ge 0$ w.h.p.~if 
\eqref{eq:gap-case4} holds and we choose
\begin{equation}
    \label{eq:gamma-4}
    \gamma \ge \frac{2}{(p_{11} - p_{1s}) \sqrt{m_1 n_1}}.
\end{equation}


Establishing nonnegativity of $\BX$ under {Case 5} in Section~\ref{sec: W cases} follows a similar argument.
Suppose that $i \in U_1$, $j \in V_r$ for $r \ge 2$ such that $(i,j) \in \Omega$. Then $\Xi_{ij} \ge 0$ with high probability if 
\begin{equation} \label{eq:gap-case5}
    p_{11} - p_{r1} \ge 2 \left( \tau + 6\max\left\{\sqrt{\sigma_{1s}^2\frac{\log N^*}{n_1}}, \frac{\log N^*}{n_1} \right\} \right)
\end{equation}
and 
\begin{equation}
    \label{eq:gamma-5}
    \gamma \ge \frac{2}{(p_{11} - p_{r1}) \sqrt{m_1 n_1}}.
\end{equation}

\noindent This establishes that $\BX$ has nonnegative entries in all cases with high probability.

\subsection{A Bound on Matrix $\boldsymbol{W}$}

 We conclude the proof by establishing a sufficient condition involving the parameters $m_{1}, n_{1}, M, N, p_{rs}$ to guarantee that the constructed $\W$, as described above, satisfies the condition $\lVert \W \rVert \le 1$ with high probability. 
 
\subsubsection{A Decomposition of $\boldsymbol{W}$}
\renewcommand{\S}{\boldsymbol{S}}

We establish an upper bound on $\lVert \W \rVert$ by utilizing the decomposition $\W = \gamma \boldsymbol{R} + \lambda \boldsymbol{S}$ with matrices $\R$ and $\S$ constructed as follows.
\begin{itemize} 
\item  For each $i \in U_1$, $j\in V_1$, we choose 
    $\displaystyle R_{ij} = \frac{W_{ij}}{\gamma}$ and $S_{ij} = 0$;
\item We set
    $R_{ij} = 0$ 
    $\displaystyle S_{ij} = \frac{W_{ij}}{\lambda}$
for all $(i,j) \notin U_1 \times V_1$.
\end{itemize}

\newcommand{\n}[1]{\lVert{#1}\rVert}
\newcommand{\nW}{\n{\W}}
\newcommand{\nR}{\n{\boldsymbol{R}}}
\newcommand{\nS}{\n{\S}}

\noindent We will bound $\n\W$ by obtaining bounds on each of $\nR$ and $\nS$ and then applying the triangle inequality.
To do so, we will make repeated use of the following bound on the following matrix concentration inequality, which provides a bound on the spectral norm of random matrices with mean-zero entries.

 \begin{lemma} \label{lem:BvH}
     Let $\A = [a_{ij}] \in \mathbf{R}^{m\times n} $ be a random matrix with entries satisfying   $\E[a_{ij}] = 0$, $\Var[a_{ij}] \le \sigma^2$, and $\lvert a_{ij} \rvert \le B$    
     for all $i=1,\dots, m$ and $j=1,\dots, n$ for some positive constants $\sigma^2$ and $B$.
     Let 
     \begin{align*}
     \sigma_{1} &:= \max_{i} \sqrt{\sum_{j} \mathbf{E}[a_{ij}^{2}]} \le  \sigma \sqrt{n} \\ 
     \sigma_{2} &:= \max_{j} \sqrt{\sum_{i} \mathbf{E}[a_{ij}^{2}]} \le \sigma\sqrt{m} 
     \end{align*}
     There exist a constant $C>0$ such that
     \[
        \Pr\left(\n\A \ge C \max \left\{\sigma_{1} + \sigma_{2} , \sqrt{B\log{t}} \right\}\right)
     \le \max\{ n,m \} t^{-7}
     \]
     for all $t > 0$.
     
 \end{lemma}

Lemma~\ref{lem:BvH} is a special case of the matrix concentration inequality given in \cite[Remark 4.12]{van2017dimension}. Note there is no assumption of symmetry, i.e., that $a_{ij} = a_{ji}$, in the hypothesis of Lemma~\ref{lem:BvH}.

\subsubsection{A Bound on $\boldsymbol{R}$}

The following lemma gives a bound on $\nR$.

\begin{lemma} \label{lem: R bound}
    If the matrix $\boldsymbol{W}$ is constructed based on Cases 1 through 6 for a matrix $\A$ sampled from the planted dense submatrix model, then there exists a constant $c_R$ such that
    \begin{equation}\label{eq: R bound}
    \| \R \| \le c_R                     
                \max \left\{ 
                    \sqrt{ \sigma_{11}^2 n_1^* \log N}, 
                    \sqrt{\frac{n_1^*}{n_{1*}}} \log N
                \right\}
    \end{equation}
    with high probability,
    where $N^* = \max \{M,N\},$ $n_1^* = \max\{m_1, n_1\}$, and $n_{1*} = \min\{m_1, n_1\}.$
\end{lemma}


\renewcommand{\y}{\boldsymbol{y}}
\renewcommand{\z }{\boldsymbol{z}}
\begin{proof}


    By construction, $\|\R \| = \|\R(U_1,V_1)\|$.
    Note that 
    \begin{equation*}
        \R(U_1,V_1) = \frac{1}{n_1} \bs{\bar\mu}\e^T + \frac{1}{m_1} \e \bs{\bar\nu}^T - \frac{\bs{\bar\mu}^T \e}{m_1 n_1} \e \e^T - \H
    \end{equation*}
    by our earlier formulas for $\y$ and $\z$,
    where $\H$ is defined by $H_{ij} = 1$ if $(i,j) \in \Omega$ and $H_{ij} = 0$ otherwise. 
    This suggests that we can decompose $\R(U_1,V_1)$ as $\R(U_1,V_1) = \Q_{1} + \Q_{2} + \Q_{3} + \Q_{4}$, where
    \begin{align*}
        \Q_{1} &= (1-p_{11}) \e \e^{T} - \H\\
        \Q_{2} &= \frac{1}{n_{1}}\boldsymbol{\bar \mu} \e^{T} - (1-p_{11}) \e \e^{T}\\
        \Q_{3} &= \frac{1}{m_{1}}\e\boldsymbol{\bar \nu}^{T} - (1-p_{11}) \e \e^{T}\\
        \Q_{4} &= \left(\frac{(1-p_{11})m_{1}n_{1} - \boldsymbol{\bar \mu}^{T} \e}{m_{1}n_{1}}\right) \e \e^{T}\\
    \end{align*}
    We start by bounding $\Q_{1}$. 
    Notice that $\Q_{1} = (1-p_{11}) \e \e^{T} - \H$ is centered with mean-zero entries and variance  $\sigma_{{11}}^2 = p_{11}(1-p_{11})$. 
    Applying Lemma~\ref{lem:BvH} with $B=1$ and $t= N^* = \max\{N,M\}$, we have
    \begin{align} 
        \lVert \Q_{1} \rVert &= \lVert (1-p_{11}) \e \e^{T} - \H \rVert
        =  \mathcal{O} \left(
            \max \left\{
                \left(\sqrt{n_{1}}+\sqrt{m_{1}} \right)\sigma_{{11}}, \sqrt{\log{N^*}} 
            \right\}
            \right) \notag\\ 
            &= 
            \mathcal{O} \left(
            \max \left\{
                \sqrt{n_{1}^*\sigma^2_{{11}}}, \sqrt{\log{N^*}} 
            \right\}
            \right) \label{eq: Q1}
    \end{align}
    Next, taking $\Q_{2} = \frac{1}{n_{1}}(\boldsymbol{\bar \mu} \e^{T} - (1-p_{11}) \e \e^{T}$ and applying Lemma \ref{lem4a}, we have
    \begin{align}
        \lVert \Q_{2}\rVert^{2}
        &= \frac{1}{n_1} \| \boldsymbol{\bar \mu} - (1-p_{11})n_1 \| \|\e\| \notag \\ 
        &= \frac{1}{\sqrt{n_1}} \| \boldsymbol{\bar \mu} - (1-p_{11})n_1 \| \notag \\ 
        &= \O \left( \sqrt{m_1} \cdot 
                \frac{1}{\sqrt{n_1}} 
                \max \left\{
                    \sqrt{\sigma^2_{11} n_1 \log N}, 
                    \log N
                \right\}
            \right) \notag \\ 
        &= \O \left(                 
                \max \left\{
                    \sqrt{\sigma^2_{11} m_1 \log N}, 
                    \sqrt{\frac{m_1}{n_1}}\log N
                \right\}
            \right)
          \label{eq: Q2}
    \end{align}
    since 
    \begin{equation*}
        \lvert \boldsymbol{\bar \mu}_i - (1 - p_{11}) n_1 \rvert = 
        \O \left(                 
                \max \left\{
                    \sqrt{\sigma^2_{11} n_1 \log N}, 
                    \log N
                \right\}
            \right)
    \end{equation*}
    w.h.p.~for all $i \in U_1$ by Bernstein's inequality.\\
    By an identical argument
    $\Q_{3} = \frac{1}{m_{1}}(\e\mathbf{\boldsymbol{\bar \nu}^{T}} - (1-p_{11}) \e\e^T)$ satisfies
    \begin{align}
        \|\Q_3\|  &=             
            \O  \left( 
                \max \left\{
                    \sqrt{\sigma^2_{11} m_1 \log N}, 
                    \log N
                \right\}
            \right)
            \label{eq: Q3}            
    \end{align}
    with high probability
    It remains to bound $\|\Q_4\|$.
    We first note that 
    \begin{equation*}
        \|\Q_4\| = \sqrt{m_1 n_1} \left\lvert \frac{(1-p_{11})m_{1}n_{1} - \boldsymbol{\bar \mu}^{T}\e}{m_{1}n_{1}} \right\rvert.
    \end{equation*}
    Moreover, $\boldsymbol{\bar \mu}^{T}\e$ is a binomially distributed random variable corresponding to $m_1n_1$ independent Bernoulli trials with probability of success $1-p_{11}$.
    Applying Lemma~\ref{lem4a} one more time shows that 
    \begin{align}
        \lVert \Q_{4}\rVert  &= 
            \O \left(
            \max \left\{ 
                \sqrt{\sigma_{11}^{2}\log{N}}, 
                \frac{\log N}{\sqrt{m_{1}n_{1}}} 
            \right\}
            \right)\label{eq: Q4}
    \end{align}
    with high probability.
    
    Applying the triangle inequality, \eqref{eq: Q1}, \eqref{eq: Q2}, \eqref{eq: Q3}, \eqref{eq: Q4}, and the union bound shows that     
    \begin{align}
        \lVert \R\rVert  &= 
            \O \left(
                \max \left\{ 
                    \sqrt{ \sigma_{11}^2 n_1^* \log N}, 
                    \sqrt{\frac{n_1^*}{n_{1*}}} \log N
                \right\}
            \right).
        \label{eq: R1 bound}
    \end{align}
    with high probability.    
\end{proof}


\subsubsection{A Bound on $\boldsymbol{S}$}
\label{sec:S-bound}


We have the following bound on the second matrix, $\boldsymbol S$, in the decomposition of $\boldsymbol{W}$.

\begin{lemma}\label{lem: S bound}
    If the matrix $\boldsymbol{W}$ is constructed according to  Cases 1 through 6 for a matrix $\boldsymbol {A}$ sampled from the planted dense submatrix model, then there exist constant $c_S$ such that
\begin{equation} \label{eq: S Bound}
    \|\boldsymbol{S}\| \le c_S
            \max \left\{ \sqrt{\tilde \sigma^2 N^* \log N^*},
                \sqrt{\max\{1, \tilde \sigma^2\} \log N^*}, 
                (\log N^*)^{3/2} \right\}
\end{equation}
with high probability.
\end{lemma}

\renewcommand{\S}{\boldsymbol{S}}
The rest of this section consists of a proof of Lemma~\ref{lem: S bound}.
To establish the desired bound on $\S$, we initially approximate $\S$ with a random matrix having mean-zero entries. Specifically, we define $\S_1$ as the random matrix constructed in the following manner:
\begin{itemize}
    \item 
        For each $i \in U_r$, $j \in V_s$ such that $r,s \ge 2$, we let 
        $$
            [\S_1]_{ij} = S_{ij} = \frac{W_{ij}}{\lambda} = 
            \begin{cases}                 
                1, & \text{if } (i,j) \notin \Omega, \\
                - \frac{p_{rs}}{1 - p_{rs}},  & \text{if } (i,j) \in \Omega.
            \end{cases}
        $$
    \item 
        For $i \in U_1$, $j \in V_s$, $s \ge 2$, we independently sample $[\S_1]_{ij}$ from the Bernoulli distribution such that 
        $$
            [\S_1]_{ij} = 
            \begin{cases}
                1, &\text{with probability $p_{1s}$}, \\ 
                -\frac{p_{1s}}{1 - p_{1s}}, & \text{with probability } 1 - p_{1s}.
            \end{cases}
        $$        
    \item
    For $i \in U_r$, $j \in V_1$, $r\ge 2$, we sample $[\S_1]_{ij}$ from the Bernoulli distribution such that 
    $$
        [\S_1]_{ij} = \psi_{r1} \times 
        \begin{cases}
            1, &\text{with probability $p_{r1}$}, \\ 
            -\frac{p_{r1}}{1 - p_{r1}}, & \text{with probability } 1 - p_{r1}.
        \end{cases}
    $$   
\end{itemize}
$\newcommand{\ts}{\tilde\sigma^2}$

\noindent Note that entries in the $(U_r, V_s)$-block are random variables with variance 
\[\tilde{\sigma}_{rs}^{2}:= \frac{p_{rs}}{1 - p_{rs}}\]
for all $(r,s) \neq (1,1).$
By construction, the entries of $\S_1$ are mean-zero independent random variables, each having variance bounded above by $\tilde\sigma^2$, 
where 
\[
\tilde\sigma^2 = \max_{\substack{(r,s) \neq (1,1)}} \left\{\tilde \sigma^2_{rs}\right\}
\]
Applying Lemma \ref{lem:BvH} with $B= \max\{1, \tilde\sigma^2 \}$ and $t= N^* = \max\{N,M\}$, shows that
\begin{equation}\label{eq: S1 bound}
    \lVert \S_1 \rVert = 
    \mathcal{O} \left( 
        \max \left\{ 
            \sqrt{\tilde\sigma^2 N^*}, 
            \sqrt{B \log N^*} 
            \right\} \right)
\end{equation}
with high probability.

The rest of the proof demonstrates that $\S$ is well approximated by $\S_1$. 
That is, we finish the proof by bounding $\lVert \S - \S_1 \rVert$. 
To do so, we consider the matrices $\S_2$ and $\S_3$ defined as follows:
\begin{align*}
   [\S_2]_{ij} &= \begin{cases}
        [\S - \S_1]_{ij},&\text{if } i \in U_1, j \in V_s, s \ge 2, \\ 
        0, &  \text{otherwise}.
    \end{cases} \\ 
   [\S_3]_{ij} &= \begin{cases}
        [\S - \S_1]_{ij},&\text{if } i \in U_r, j \in V_1, r \ge 2, \\ 
        0, &  \text{otherwise}.
    \end{cases}
\end{align*}


To bound the norm of $\S_2$ and the norm of $\S_3$, we will apply the following lemma, which offers bounds on the spectral norm of random matrices with this particular form. A proof of Lemma~\ref{lem4.4} is given in Appendix~\ref{app:MBI}.

\renewcommand{\A}{\boldsymbol{A}}
\newcommand{\T}{\boldsymbol{\Theta}}
\begin{lemma}\label{lem4.4}
Let $V_1, V_2, \dots, V_K$ be a partition of $[N]$.
Let $\T$ be a random $n\times N$ matrix such that
\begin{equation} \label{eq:theta-def}
\Theta_{ij} = 
\begin{cases}
    1, & \text{with probability } p_s \\ 
    -\frac{p_s}{1-p_s}, & \text{with probability } 1 - p_s
\end{cases}
\end{equation}
for all $i \in [n]$, $j \in V_s$ for probability $p_s \in [0,1]$.
Let $\tilde{\T}$ be the random matrix defined by
\begin{equation}\label{eq:tilde-theta-def}
    [\Tilde{\T}]_{ij} := 
    \begin{cases}
            1 & \text{if } \Theta_{ij} = 1\\
            \frac{-n_{j}}{n - n_{j}}, & \text{otherwise,}
    \end{cases}
\end{equation}
for each $i \in [n]$ and $j \in V_s$,
where $n_{j}$ is the number of $1$s in the $j$th column of $\T$. 
Then there is constant $c > 0$ such that       
\begin{align*}
    \prob
    \left( \lVert \T - \Tilde{\T} \rVert \ge 
        c \max \left\{ \sqrt{\hat \sigma^2 N \log N}, (\log N)^{3/2} \right\}
        \right) = \mathcal{O} \left(N^{-5}\right),
\end{align*}
where 
\begin{equation*}
    \hat\sigma^2 := 
    \max_s \left\{ \frac{p_s}{1-p_s} \right\}.
\end{equation*}
\end{lemma}              

Note that $\S_2$ has the same form as $\T - \tilde\T$ in Lemma~\ref{lem4.4}. 
Indeed,
$S_{ij}$ is sampled according to the distribution defining $\T$ given in~\eqref{eq:theta-def} and $[\S_1]_{ij}$ is sampled according to~\eqref{eq:tilde-theta-def} with respect to $p_s = p_{1s}$ for each $j \in V_s$.
It follows immediately that 
\begin{align} 
    \|\S_2\| &=
        \O\left( 
            \max \left\{ \sqrt{\tilde \sigma^2 (N-n_1) \log (N-n_1)}, (\log (N-n_1))^{3/2} \right\}
            \right) \notag \\ 
        &= \label{eq: S2}
        \O \left(
            \max \left\{ \sqrt{\tilde \sigma^2 N^* \log N^*}, (\log N^*)^{3/2} \right\}
        \right).
\end{align}
with high probability.
By an identical calculation,
\begin{align} 
    \|\S_3\| &=
        \O\left( 
            \max \left\{ \sqrt{\tilde \sigma^2 (M-m_1) \log (M-m_1)}, (\log (M-m_1))^{3/2} \right\}
            \right) \notag \\ 
        &= \label{eq: S3}
        \O \left(
            \max \left\{ \sqrt{\tilde \sigma^2 N^* \log N^*}, (\log N^*)^{3/2} \right\}
        \right)
\end{align}
with high probability.
Combining \eqref{eq: S1 bound}, \eqref{eq: S2}, and~\eqref{eq: S3} shows that 
\begin{equation*}
\|\S\| = \O \left(
            \max \left\{ \sqrt{\tilde \sigma^2 N^* \log N^*},
                \sqrt{\max\{1, \tilde \sigma^2\} \log N^*}, 
                (\log N^*)^{3/2} \right\}
        \right)
\end{equation*}
with high probability.

\subsubsection{Completing the Bound on $\W$}

Without loss of generality, we assume $N = N^*$ and $m_1 \le n_1$for notational simplicity.
Applying the triangle inequality with the bounds given by Lemma~\ref{lem: R bound} and Lemma~\ref{lem: S bound} shows that 
\begin{align} 
    \|\W\| &\le \gamma \|\R\| + \lambda \|\S\| \notag \\ 
    &\le \gamma~c_R \left(
                \max \left\{ 
                    \sqrt{ \sigma_{11}^2 n_1^* \log N}, 
                    \sqrt{\frac{n_1^*}{n_{1*}}} \log N
                \right\}
            \right) \notag\\
        &\hspace{0.25in}+ \lambda~c_S \max \left\{
                            \sqrt{\tilde\sigma_*^2 N \log N},
                            (\log N)^{3/2}
                            \right\}    \label{eq: W1}
\end{align}
with high probability.
Recall that 
\begin{align*}
    \lambda &= \frac{1}{\sqrt{m_1 n_1}} + \gamma (1- p_{11}) + \gamma \tau  \\ 
    & = \frac{1}{\sqrt{m_1 n_1}} + 
        \O \left( \gamma
            \max \left\{ 
                (1- p_{11}), 
                \frac{\sigma_{11}^2 \log N}{n_{\min}},
                \frac{\log N}{n_{\min}}
                \right\}
            \right) \\ 
    & = \frac{1}{\sqrt{m_1 n_1}} + \O(\gamma)  = \O(\gamma)
\end{align*}
if we choose 
$$
    \gamma = \O \left(\frac{1}{(p_{11} - p^*)\sqrt{m_1 n_1}} \right)
$$
by our choice of $\tau$ and the assumption that $(\log N)^{3/2} = \O(n_{\min})$ in~\eqref{eq:rec-size2}.

It follows that $\|\W \| \le 1$ and~\eqref{eq:gap-3},~\eqref{eq:gap-case4}, and~\eqref{eq:gap-case5}, which ensure nonnegativity of $\bs{\Lambda}$ and $\bs{\Xi}$, are satisfied with high probability if we choose 
 $\gamma \le c_\gamma/((p_{11} - p^*)\sqrt{m_1 n_1})$ for sufficiently small $c_\gamma > 0$ 
and 
$$
    p_{11} - p_{rs} \ge c_1 \max 
    \left\{
        \sqrt{\frac{\tilde\sigma^2 N^* \log N^*}{m_1 n_1}}, 
        \sqrt{\frac{\sigma_{11}^2 \log N^*}{\min\{m_1,n_1\}}}, 
        \sqrt{\frac{\max\{\tilde\sigma^2, 1\} (\log N^*)^3}{\min\{m_1, n_1\}}}
    \right\}
$$
for all $(r,s) \neq (1,1).$
This completes the proof.

\section{Proof of Recovery in the Adversarial Case}
\label{sec:ADMM}

The proof the sufficient condition in the adversarial case follows a similar structure that of \cref{sec:random-proof}.
The proof relies on verifying that a particular choice of Lagrange multipliers satisfies the uniqueness and optimality conditions given in \cref{thm:KKT}. As in \cref{sec:random-proof}, we start with a choice of $\W$.

\begin{itemize}
    \item 
        \textbf{Case 1:} If $(i,j) \in {\bar U}_{1} \times {\bar V}_{1} - \Omega$, then~\eqref{8a} indicates that
        \begin{align*}
             \Xi_{ij} &= \gamma &
             W_{ij} & = \bar{\lambda} - \Lambda_{ij}.
        \end{align*}

    \item 
        \textbf{Case 2:} If $(i ,j) \in  \Omega \cap ({\bar U}_{1} \times {\bar V}_{1})  $, then
        \begin{align*}
             \Xi_{ij} &= 0 &
             W_{ij} & = \bar{\lambda} -\gamma - \Lambda_{ij}.
        \end{align*}
    \item 
        \textbf{Case 3:} Suppose that 
        $i \in U_q,$ $j \in V_s$ for some $q, s \ge 2$.
        Then we choose
        \begin{align*}
            W_{ij}  &=  
            \begin{cases}                 
                \lambda, & \text{if } (i,j) \notin \Omega, \\
                0,  & \text{if } (i,j) \in \Omega,
            \end{cases}
        & 
            \Xi_{ij} &=
            \begin{cases}
                \gamma, & \text{if } (i,j) \notin \Omega, \\ 
               \gamma - \lambda, & \text{if } (i,j) \in \Omega
            \end{cases}
        \end{align*}
        according to~\eqref{8a}.
        
    \item \textbf{Case 4:}
        Suppose that $i \in U_1$, $j \in V - V_1$.
        \begin{align*} 
            W_{ij}  &= 
            \begin{cases}
                \lambda, & \text{if } (i,j) \notin \Omega, \\
                -\lambda \left( \frac{\nu_{j}}{m_1-\nu_{j}} \right), & \text{if } (i,j) \in \Omega,
            \end{cases}
            &
            \Xi_{ij} &=
            \begin{cases} 
            \gamma, & \text{if } (i,j) \notin \Omega, \\
            \gamma - \lambda \left( \frac{m_1}{m_1-\nu_{j}}\right), & \text{if } (i,j) \in \Omega,
            \end{cases}
        \end{align*}
        where $m_1 := \lvert U_1 \rvert$ and $\nu_j$ denotes the number of nonzero entries of the $j$th column of $\boldsymbol{A}$ in rows indexed by $U_1$.
    \item \textbf{Case 5:}
        Similarly, if $i \in U - U_1$, $j \in V_1$, we choose
        \begin{align*} 
            W_{ij}  &= 
            \begin{cases}
                \lambda, & \text{if } (i,j) \notin \Omega, \\
                -\lambda \left( \frac{\mu_{i}}{n_1-\mu_{i}} \right), & \text{if } (i,j) \in \Omega,
            \end{cases}
            &
            \Xi_{ij} &= 
            \begin{cases} 
            \gamma, & \text{if } (i,j) \notin \Omega, \\
            \gamma - \lambda \left( \frac{n_1}{n_1-\mu_{i}}\right), & \text{if } (i,j) \in \Omega.
            \end{cases}
        \end{align*} 
\end{itemize}

\subsection{Nonnegativity of $\BL$ in the Adversarial Case}
\label{sec:adv-nonneg-Lambda}

As in \cref{sec:BL-choice}, we choose 
$$
\Lambda_{ij} = y_i + z_j
                 = \bar{\lambda} - \gamma \bigg ( \frac{\bar{\mu}_{i}}{n_1} +  \frac{\bar{\nu}_{j}}{m_1} -  \frac{ \bar r_{11}}{m_1n_1}\bigg)\\
$$
for all $(i,j) \in \bar{U}_1 \times \bar{V}_1$
so that $\W^T\bu = \0$ and $\W \bv = \0$, where $r_{11}$ denotes the number of $0$ entries in the $(\bar{U}_1, \bar{V}_1)$-block of $\A$.

By assumption, we have 
$\bar\mu_i \ge (1 - \tilde{\delta})n_1$ and $\bar\nu_j \ge (1 - \tilde{\delta})m_1$ for all $(i,j)$.
It follows that 
\begin{align*}
    \Lambda_{ij} &\ge \bar{\lambda} - 2(1 - \tilde{\delta} )\gamma \ge 0 
\end{align*}
if 
\begin{equation}\label{eq:adv-lam-bound}
\lambda \ge \frac{1}{m_1 n_1} + 2(1 - \tilde{\delta} )\gamma.
\end{equation}

\subsection{Nonnegativity of $\BX$ in the Adversarial Case}
\label{sec:adv-nonneg-xi}

As in \cref{sec:nonneg-Xi}, we need only consider Cases 3,4, and 5 in the construction of $\BX$.

We begin with Case 3. Suppose that $(i,j) \in \Omega \cap U_q \times V_s$. In this case, we have 
$$
    \Xi_{ij} = \gamma - \lambda \ge 0
$$
if $\gamma \ge \lambda$. Enforcing \eqref{eq:adv-lam-bound} shows that it suffices to choose $\gamma$ satisfying
\begin{equation} \label{eq:adv-xi-3}
    (2 \tilde \delta - 1) \gamma \ge \frac{1}{\sqrt{m_1 n_1}}.
\end{equation}
Note that \eqref{eq:adv-xi-3} is only satisfiable if $\tilde\delta > 1/2$.

We next consider Case 4, i.e., when $i \in U_1$, $j \in V- V_1$ and $(i,j) \in \Omega$.
In this case,
$$
\Xi_{ij} = \gamma - \lambda \left( \frac{m_1}{m_1-\nu_{j}}\right) \ge \gamma -  \frac{\lambda}{1 - \delta}
$$
by the assumption that $\nu_j \le \delta m_1$. Again, enforcing \eqref{eq:adv-lam-bound} shows that $\Xi_{ij}$ in this case if 
\begin{equation}
    \label{eq:adv-xi-4}
    (2 \tilde \delta - \delta - 1) \gamma \ge \frac{1}{\sqrt{m_1 n_1}}.
\end{equation}
Finally, a symmetric argument shows that $\Xi_{ij} \ge 0$ if \eqref{eq:adv-xi-4} is satisfied when $i\in U - U_1$, $j\in V_1$, and $(i,j)\in\Omega$.
Choosing $\gamma$ according to \eqref{eq:adv-gamma-def} ensures that \eqref{eq:adv-xi-3} and \eqref{eq:adv-xi-4} are satisfied.
Moreover, we choose
$$
    \lambda = \frac{1}{m_1 n_1} + \frac{2(1-\td)}{2 \td - \delta - 1} \cdot \frac{1}{\sqrt{m_1 n_1}} 
    :=  \frac{1}{m_1 n_1} + \frac{\tilde c (1 -\td)}{\sqrt{m_1 n_1}},
$$
where $\tilde c := 2/(2\td - \delta - 1)$, which satisfies~\eqref{eq:adv-lam-bound}.

\subsection{A Bound on $\W$ in the Adversarial Case}
\label{sec:adv-W}

It remains to establish conditions when $\|\W\| \le 1$ under the assumptions given in \cref{sec:suff-cond-det}.
To do so, we decompose $\W$ as $\pmb{W} = \pmb{W}_{1} + \pmb{W}_{2} + \pmb{W}_{3}$, where
\begin{align*}
        \pmb{W}_{1}[U_q, V_s] &= \begin{cases}
                \W[U_q, V_q], & \text{if } ~q=s, \\
                0, & o/w,
            \end{cases} 
            \\ 
        \pmb{W}_{2}[U_q, V_s] &= \begin{cases}
                \W[U_1, V_s], & \text{if } ~q=1, s\ge 2 \\
                \W[U_q, V_1], & \text{if } ~q\ge 2, s=1 \\
                0, & o/w,
            \end{cases}
            \\ 
        \pmb{W}_{3}[U_q, V_s] &= \begin{cases}
                \W[U_q, V_s], & \text{if } ~q,s\ge 2, q\ne s \\
                0, & o/w,
            \end{cases}
\end{align*}

We will bound each of $\|\bs{W_1}\|, \|\bs{W_2}\|,$ and $\|\bs{W_3}\|$ individually and use the triangle inequality to bound $\|\W\|$. We start with $\|\bs{W_2}\|$ and the following initial bound:
\begin{align*}
    \| \pmb{W}_2 \|^{2} \le \| \pmb{W}_2 \|^{2}_{F} = \sum^{k}_{s=2}\left(  \|\pmb{W}[U_{1}, V_{s}]\|^{2}_{F} + \|\pmb{W}[U_{s},V_{1}]\|^{2}_{F} \right).
\end{align*}
The first summand is bounded above by
\begin{align*}
  \sum^{k}_{s=2} \|\pmb{W}[U_1, V_s]\|^{2}_{F} &= \sum^{k}_{i\in U_1} \sum^{k}_{j\in V- V_1}W_{ij}^{2}\\
  &=  \lambda^{2} \sum^{k}_{j\in V- V_1}\bigg ( v_{j} + (m_1 -v_j)\frac{v_j ^{2}}{(m_1 - v_j)^2}\bigg)\\
  &=  \lambda^{2} \sum^{k}_{j\in V- V_1} v_{j} \bigg ( 1 + \frac{v_j}{m_1 - v_j}\bigg)\\
  &\le \lambda^{2} \sum^{k}_{j\in V- V_1} \frac{v_j}{1- \delta} \le \lambda^{2} \frac{{r}_1}{1- \delta}.
\end{align*}
By an identical argument, we have
\begin{align*}
    \sum^{k}_{s=2} \|\pmb{W}[U_{s},V_{1}]\|^{2}_{F}\le \lambda^{2} \frac{{r}_2}{1- \delta}.
\end{align*}
Combining these two bounds, we have 
\begin{equation}
    \label{eq:adv-W2}
    \|\bs{W_2}\|^2 \le \frac{\lambda^2}{1- \delta}\left( r_1 +  r_2 \right).
\end{equation}
We bound $\bs{W_3}$ using a similar argument as
\begin{equation}\label{eq:adv-W3}
    \| \pmb{W}_3 \|^{2} \le \| \pmb{W}_3 \|^{2}_{F} \leq\lambda^2 {r}_{3}.
\end{equation}
It remains to bound $\bs{W_1}$.

To do so, we note that $\bs{W_1}$ is block diagonal and, therefore, 
\begin{align*}
    \| \pmb{W}_1 \| = \max_{q=1,\dots, k} \|\pmb{W}[U_q, V_q]\|. 
\end{align*}
For $q\ge 2$, there are at most $r_{qq}$ entries of $\pmb{W}[U_q, V_q]$ equal to $\lambda$, while the remaining entries are zeros. Thus,
\begin{equation}
    \label{eq:adv-wqq}
        \|\pmb{W}[U_q, V_q]\|^{2} \le \|\pmb{W}[U_q, V_q]\|^{2}_{F}
    \le \lambda^{2}r_{qq}.
\end{equation}
On the other hand,
\begin{align*}
    \pmb{W}(u_{1},v_{1}) = \bar{\lambda}\1\1 - \1\y^{T} +\1\z^{T} - \gamma \bar{\A}(u_{1},v_{1}).
\end{align*}
We decompose $\bs{W_1}$ further as $\bs{W_1}[U_1, V_1] = \bs{\tilde W_1} + \bs{\tilde W_2}$, where 
\begin{align*}
\tilde{\pmb{W}}_{1} &= \gamma \left( \frac{\bar{\pmb{\mu}}_1\1^{T}}{n_1} + \frac{\1\bar{\pmb{\nu}}_1^{T}}{m_1} \right),\\
    \tilde
    {\pmb{W}}_{2}&= \bar\lambda \1\1^T + \gamma \bigg (\frac{\bar{r}_{11}}{m_1 n_1}\1\1^{T}  - \bar{\A}(u_{1},v_{1})\bigg)    .
\end{align*}

\subsubsection{A Bound on $\bs{\tilde W_1}$}
\label{sec:adv-w11}

We want to bound the maximum eigenvalue of $\bs{\tilde \W_1}^T \bs{\tilde\W_1}$ to bound the spectral norm of $\bs{\tilde W_1}$.
Taking the product $\bs{\tilde \W_1}^T \bs{\tilde\W_1}$ we have 
\begin{align*}
    \bs{\tilde \W_1}^T \bs{\tilde\W_1} &= \frac{ \|\bar{\bs{ \mu_1}} \|^2 }{ n_1^2 } \1 \1^T 
    + \frac{\bs{\mu_1}^T \1}{m_1 n_1}  \left( \bar{\bs{\nu_1}} \1^T + \1 \bar{\bs{\nu_1}}^T \right) 
    + \frac{1}{m_1} \bar{\bs{\nu_1}} \bar{\bs{\nu_1}}^T.
\end{align*}

We will bound each term separately.
First, note that 
$$
\bar{\pmb{\mu}}_1^{T}\bar{\pmb{\mu}}_{1} = \sum_{i\in U_1}\mu_{i}^{2} \le \|\bar{\pmb{\mu}}_1\|_{1}\|\bar{\pmb{\mu}}_1\|_{\infty} = \bar{r}_{11} (1- \tilde{\delta})n_1,
$$
which implies that the first summand has maximum eigenvalue at most 
\begin{equation} \label{eq:adv-w11-s1}
\bar{r}_{11} (1- \tilde{\delta})
\end{equation}
since $\|\1 \1^T \| = n_1.$
By a similar argument, we have 
\begin{equation}\label{eq:adv-w11-s3}
\left\| \frac{1}{m_1} \bar{\bs{\nu_1}} \bar{\bs{\nu_1}}^T \right\| 
    \le \bar{r}_{11} (1- \tilde{\delta}). 
\end{equation}
It remains to bound the remaining term in the decomposition.

Note that $\bar{\bs{\nu_1}} \1^T + \1 \bar{\bs{\nu_1}}^T$ has rank at most 2, with nontrivial eigenvalues
$$
 \bar{\pmb{\nu}}_{1}^T\1 \pm \|\bar{\pmb{\nu}}_{1}\|_{2} \|\1\|_{2} =     \bar{\pmb{\nu}}_{1}\1^{T} \pm \|\bar{\pmb{\nu}}_{1}\|_{2}\sqrt{n_1}. 
$$
We know that 
\begin{align*}
    \bar{\pmb{\nu}}_{1}^T\1 &\le (1- \tilde{\delta})n_1\\ \|\bar{\pmb{\nu}}_{1}\|_{2}^{2} &\le \bar{r}_{11} (1- \tilde{\delta})n_1.
\end{align*}
Therefore, 
$$
\|\bar{\pmb{\nu}}_{1}\1^{T} + \1\bar{\pmb{\nu}}_{1}^{T} \|_{2} 
    \le (1- \tilde{\delta})n_1 + \sqrt{\bar{r}_{11} (1- \tilde{\delta})n_1}
    \le 2 \bar{r}_{11},
$$
since $(1-\tilde \delta)n_1 \le \bar r_{11}$.
This implies that 
\begin{equation} \label{eq:adv-w11-s2}
\left\|\frac{\bar{\pmb{\mu}}_1^{T}\1}{m_1 n_1}\bigg(\1\bar{\pmb{\nu}}_1^{T} + \bar{\pmb{\nu}}_1\1^{T} \bigg) \right\|
    \le \frac{(1-\tilde \delta) m_1}{m_1} \|\bar{\pmb{\nu}}_{1}\1^{T} + \1\bar{\pmb{\nu}}_{1}^{T} \|_{2} 
    \le 4 (1 - \tilde \delta) \bar r_{11}.
\end{equation}
Combining \eqref{eq:adv-w11-s1}, \eqref{eq:adv-w11-s3}, and \eqref{eq:adv-w11-s2} and applying the triangle inequality shows that 
\begin{equation}\label{eq:adv-w11-all}
\| \bs{\tilde W_1} \| \le 2 \sqrt{ ( 1- \tilde \delta) \bar{r}_{11}}.
\end{equation}

\subsubsection{A Bound on $\bs{\tilde W_2}$}
\label{sec:adv-w22}

By construction, 
$$
    [\bs{\tilde W_{2}}]_{ij} = 
    \begin{cases} 
        2 (1- \td) + \frac{\bar{r}_{11}}{m_1 n_1}, & \text{if } (i,j) \notin \Omega, \\ 
        2(1-\td) + \frac{r_{11}}{m_1 n_1}, & \text{if } (i,j) \in \Omega. 
    \end{cases}
$$
Here, we have chosen 
$$
\lambda = \frac{1}{m_1 n_1} + 2 ( 1 -\td) \gamma. 
$$
Recall that the set $\Omega \cap U_1 \times V_1$ has cardinality $\bar{r}_{11} = m_1 n_1 - r_{11}.$ It follows that 
\begin{align}
\|\bs{\tilde W_2} \|^2 \le \|\bs{\tilde W_2} \|_F^2 
    &= r_{11} \left(  2 (1- \td) + \frac{\bar{r}_{11}}{m_1 n_1} \right)  
            + (m_1 n_1 - r_{11}) \left(  2(1-\td) + \frac{r_{11}}{m_1 n_1} \right)\notag  \\ 
    &= 4(1 - \td)^2 m_1 n_1 + \frac{r_{11}  (m_1 n_1 - r_{11})}{m_1 n_1} \notag\\  
    &\le 4(1 - \td)^2 m_1 n_1  + \bar{r}_{11}. \label{eq:adv-w22-all}
\end{align}

\subsubsection{Finishing the Bound on $\W$}
\label{sec:bound-on-W}

We can obtain the desired bound on $\|\W\|$ by applying the triangle inequality and combining the partial bounds given by~\eqref{eq:adv-W2}, \eqref{eq:adv-W3}, \eqref{eq:adv-w11-all}, and~\eqref{eq:adv-w22-all}:
\begin{align}
    \|\W\| &\le \|\bs{W_1}\| + 
        \|\bs{W_2}\| + \|\bs{W_3}\|  \notag \\ 
    & \le \max\left\{ 
            \max_{q\ge 2} \left\{\lambda \sqrt{r_{qq}} \right\}, 
            \gamma \left( 
                2 \sqrt{(1-\td) \bar r_{11}} 
                    + \sqrt{4(1-\td)^2 m_1 n_1 + \bar{r}_{11}  } 
            \right) 
            \right\}\notag \\ 
    & \hspace{0.4in} + 
    \frac{\lambda}{\sqrt{1-\delta}} \sqrt{r_1 + r_2} + \lambda \sqrt{r_3}   \label{eq:adv-W-intermediate}
\end{align}
Note that 
$$
\|\bs{W_2} \| + \|\bs{W_3}\| = \left(\frac{1}{m_1 n_1} + \frac{\tilde c (1 -\td)}{\sqrt{m_1 n_1}} \right)
    \left( \sqrt{\frac{r_1 + r_2}{1-\delta}} + \sqrt{r_3} \right)
< \frac{1}{2}
$$
if $\max\{r_1, r_2, r_3\} < c m_1 n_1$ for sufficiently small constant $c$.
Similarly, note that 
$$
    \lambda \sqrt{r_{qq}} < \frac{1}{2}
$$
if $r_{qq} < c m_1 n_1$ for sufficiently small $c$ for all $q = 2, 3, \dots, k$. 

Finally, 
\begin{align*}
     \|\bs{\tilde{W}_1} \| + \|\bs{\tilde W_2}\| &
        \le \gamma \left(1 + 2 \sqrt{1 - \td} \right) \sqrt{\bar r_{11}} 
            + 2(1-\td) \sqrt{m_1 n_1} \\
        &= \frac{1 + 2 \sqrt{1- \td }}{2\td - \delta - 1} \sqrt{\frac{\bar r_{11}}{m_1 n_1} } 
            + \frac{2 (1-\td)}{2 \td - \delta - 1} 
         < \frac{1}{2} 
\end{align*}
if $\bar{r}_{11} < c m_1 n_1$ for sufficiently small $c$ and $\td$ sufficiently large.
Putting everything together shows that $\|\W\| < 1$ if~\eqref{eq:adv-size-cond} holds with sufficiently small $c$.
\section{Numerical Experiments}
\label{sec:ADMM}

\subsection{An Algorithm Based on the Alternating Direction Method of Multipliers}

The Alternating Direction Method of Multipliers (ADMM) is a powerful optimization technique particularly well-suited for large-scale and constrained problems. In this section, we present an overview of how the ADMM is applied to solve our relaxation of the densest submatrix~\eqref{eqn3.2}, highlighting its decomposition capabilities and iterative structure that enable efficient computation. Comprehensive details on the ADMM can be found in \cite{boyd2011distributed}. The methodology is first derived analytically, followed by a rigorous empirical evaluation using both synthetically generated matrices and real-world datasets representing collaboration and communication networks.

The {ADMM} algorithm is used to solve optimization problems of the form:
\begin{equation}\label{eq:gen-ADMM-prob}
\begin{array}{cl}
     \min & f(\x) + g(\y)\\
     \suchthat &  \bs A \bs x + \bs By = \bs C.
\end{array}    
\end{equation}
This can be generalized to multi-block problems of the form
\begin{equation}\label{eq:multi-block-ADMM}
\begin{array}{cl}
\min & \sum_{i=1}^k f(\x_i) \\ 
\suchthat & \sum_{i=1}^k \A_i \x_i = \bs{C}
\end{array}
\end{equation}
That is, the ADMM is used to solve optimization problems with separable objective function subject to linear coupling constraints.

We apply a multi-block ADMM to solve~\eqref{eqn3.2} as proposed in~\cite{bombina2020convex}.
We can rewrite the convex program given in~\eqref{eqn3.2}
to fit the framework~\eqref{eq:multi-block-ADMM} by introducing auxiliary variables $\Q, \W, \Z$:
\begin{equation} \label{eqn:lag 1}
\begin{array}{cl}
     \min &\lVert \X \rVert_{*} +  \gamma \lVert \Y \rVert_{1} + \mathbb{I}_{\Omega_{Q}}(\boldsymbol{Q}) + \mathbb{I}_{\Omega_{W}}(\W) + \mathbb{I}_{\Omega_{Z}}(\Z) \\
     \suchthat &\X + \Y = \boldsymbol{Q},  ~\X = \W, ~\X = \Z,
\end{array}
\end{equation}
where \(\Omega_Q\), \(\Omega_W\), and \(\Omega_Z\) denote the sets
\begin{align*}
\Omega_Q &:= \{\boldsymbol{Q} \in \mathbb{R}^{M \times N} : P_{\Omega}(\Q) = \0\}, \\
\quad \Omega_W &:= \{\W \in \mathbb{R}^{M \times N}  : \e^T \W\e = m_{1}n_{1}\},\\
\Omega_Z &:= \{\Z \in \mathbb{R}^{M \times N}  : 0 \le Z_{ij} \leq 1 \, \forall (i,j)\}.
\end{align*}
Here, \(\mathbb{I}_{\mathcal{S}} : \mathbb{R}^{M \times N} \to \{0, +\infty\}\) represents the indicator function of the set \(\mathcal{S} \subseteq \mathbb{R}^{M \times N}\), where \(\mathbb{I}_{\mathcal{S}}(\X) = 0\) if \(\X \in \mathcal{S}\), and \(+\infty\) otherwise. 

To solve \eqref{eqn:lag 1}, we employ an iterative approach based on the ADMM.
In each iteration, we sequentially update each primal decision variable by minimizing the augmented Lagrangian to approximate the dual functional gradient and then apply an approximate dual gradient ascent step.

We define the augmented Lagrangian \( \mathcal{L}_\tau \) as follows:
\begin{align*}
\mathcal{L}_\tau &= \|\X\|_{*} + \gamma \|\Y\|_{1} + \mathbb{I}_{\Omega_Q}(\boldsymbol{Q}) + \mathbb{I}_{\Omega_W}(\W) + \mathbb{I}_{\Omega_Z}(\Z) \\
&\quad + \mathrm{Tr}(\bs\Lambda_Q^\top (\X + \Y - \boldsymbol{Q})) + \mathrm{Tr}(\bs\Lambda_W^\top (\X - \W)) + \mathrm{Tr}(\bs\Lambda_Z^\top (\X - \Z)) \\
&\quad + \frac{\tau}{2} \left( \|\X+ \Y - \boldsymbol{Q}\|_F^2 + \|\X - \W\|_F^2 + \|\X - \Z\|_F^2 \right),
\end{align*}
where \( \tau \) is a regularization parameter selected to ensure that \( \mathcal{L}_\tau \) is strongly convex in each of the primal variables.

In each iteration, the primal variables are updated sequentially using a Gauss–Seidel strategy to minimize \( \mathcal{L}_\tau \) with respect to each of $\X$, $\Y$, $\bs Q$, $\W$, and $\Z$, followed by updates to the dual variables \(\bs \Lambda_Q \), \(\bs \Lambda_W \), and \(\bs \Lambda_Z \). The updates of the auxiliary variables \( \boldsymbol{Q} \), \( \W \), and \( \Z \) reduce to orthogonal projections onto the constraint sets \( \Omega_Q \), \( \Omega_W \), and \( \Omega_Z \), respectively; each of these projections can be computed in closed form.

The subproblems for updating \( \X \) and \( \Y \) admit closed-form solutions via the elementwise soft-thresholding operator \( S_\phi : \mathbb{R}^n \to \mathbb{R}^n \) defined as
\[
(S_\phi(\x))_i = 
\begin{cases}
x_i - \phi, & \text{if } x_i > \phi, \\
0, & \text{if } |x_i| \leq \phi, \\
x_i + \phi, & \text{if } x_i < -\phi.
\end{cases}
\]
The algorithm is terminated once the primal and dual residuals,
$\|\X^{(t+1)} - \W^{(t+1)}\|_F$, $\|\X^{(t+1)} - \Z^{(t+1)}\|_F$, $\|\W^{(t+1)} - \W^{(t)}\|_F$, $\|\boldsymbol{Q}^{(t+1)} - \boldsymbol{Q}^{(t)}\|_F$, $\|\Z^{(t+1)} - \Z^{(t)}\|_F$, $\|\boldsymbol{\Lambda}_Q^{(t+1)} - \boldsymbol{\Lambda}_Q^{(t)}\|_F$,$ \|\boldsymbol{\Lambda}_W^{(t+1)} - \boldsymbol{\Lambda}_Z^{(t)}\|_F$ and \(\|\boldsymbol{\Lambda}_Z^{(t+1)} - \boldsymbol{\Lambda}_Q^{(t)}\|_F\), are all within desired error bounds. 
The algorithmic steps are summarized in Algorithm~\ref{alg:densub}.

\begin{remark}[Convergence of Alg.~\ref{alg:densub}]
It has been established that ADMM exhibits linear convergence when applied to the minimization of convex separable functions under mild regularity assumptions. 
For example, Problem~\eqref{eqn:lag 1} satisfies the hypothesis of Theorem~3.1 in~\cite{hong2017linear}, which establishes that the Algorithm~\ref{alg:densub} is expected to converge linearly to a global optimum of Problem (\ref{eqn3.2}). 
\end{remark}

\begin{algorithm}[t]
\begin{algorithmic}
\caption{DENSUB: ADMM for Densest Submatrix/Subgraph Recovery}
\label{alg:densub}
\State 
{\bf Input:} adjacency matrix $\A \in \mathbb{R}^{M \times N}$, Target submatrix size $(m, n)$, Regularization parameter $\gamma$, Augmented Lagrangian penalty $\tau$,
    Stopping tolerance $\epsilon$, Maximum iterations $\texttt{maxiter}$

\State
{\bf Output:} Primal variables $\X, \Y, \Q$, number of iterations $\texttt{iter}$.

\State{\bf Initialization:}
Initialize $\W = \X = \Y = \Z = (mn)/(MN) \cdot \1 \1^T$.
Set dual variables $\bs\Lambda_Q = \bs\Lambda_Z = \bs\Lambda_W = \bs0$.
Set $\texttt{iter} = 0$; $\texttt{convergence} = \texttt{false}$.

\While{\texttt{convergence} == \texttt{false}}
    \State
    Increment iteration counter: $\texttt{iter} \leftarrow \texttt{iter} + 1$. 

    \State 
    {\bf Update $\Q$:} 
        \State 
        $\Q \leftarrow (\X - \Y + (1/\tau) \bs \Lambda_Q)$
        \State 
        Elementwise mask: $\Q \leftarrow \Q \circ \A$ (Hadamard product)

    \State {\bf Update $\X$:}
        $\X \leftarrow \texttt{MatrixShrink} 
            \left( 
                \frac{1}{3} \left(
                    \Y + \Q + \Z + \W - \frac{1}{\tau}(
                        \bs\Lambda_Q + \bs \Lambda_W + \bs \Lambda_Z)
                    \right), 
                    \frac{1}{3 \tau}                 
            \right)$

    \State {\bf Update $\Y$:}
    $\Y \leftarrow \max \left( \X - \Q - \gamma \frac{1}{\tau} \e\e^T + \frac{1}{\tau} \bs \Lambda_Q, \0 \right)$

    \State{\bf Update $\W$:}
    $\W \leftarrow \X + \mu \bs \Lambda_W$ then
    shift to enforce $\sum W_{ij} = mn$ by adding constant offset. 

    \State{\bf Update $\Z$:}
    $\Z \leftarrow \texttt{clip}(\X + \mu \bs\Lambda_Z, \0, \e\e^T)$ (clip between 0 and 1)    

    \State{\bf Update dual variables:}
        \State $\bs\Lambda_Q \leftarrow \bs\Lambda_Q + \tau (\X - \Y - \Q)$
        \State $\bs\Lambda_W \leftarrow \bs\Lambda_W + \tau (\X - \W)$  
        \State $\bs\Lambda_Z \leftarrow \bs\Lambda_Z + \tau (\X - \Z)$

    \State{\bf Check for convergence:}
        \State 
        Primal residual: $\epsilon_P = \max \left\{ \|\X-\Z\|_F, \|\X-\W\|_F, \|\X-\Y-\Q\|_F \right\} / \|\X\|_F$
        \State 
        Dual residual: 
        $\epsilon_D = \max \left\{ \|\Z-\Z_{\text{old}}\|_F, \|\W-\W_{\text{old}}\|_F, \|\Q-\Q_{\text{old}}\|_F \right\} / \|\X\|_F$
        \If{($\epsilon_P < \epsilon$ and $\epsilon_D < \epsilon$) or $\texttt{iter} \geq \texttt{maxiter}$}
            \State Set $\texttt{convergence} = \texttt{true}$
        \EndIf
        
\EndWhile
    
\end{algorithmic}
\end{algorithm}

\begin{remark}[Computational Cost of Alg.~\ref{alg:densub}]
We need $\O(MN)$ floating point operations (flops) to update all primal and dual variables except $\X$ during each iteration of Alg.~\ref{alg:densub}. Indeed, each of these variable updates reduces to entry-wise operations on $M\times N$ real matrices. Updating $\X$ is the most computationally expensive step of each iteration and requires the calculation of the singular value decomposition of a dense matrix. This requires $\O(\min\{M,N\} MN)$ flops per iteration. Thus, each iteration of Alg.~\ref{alg:densub} uses $\O(\min\{M,N\} MN)$ flops.    
\end{remark}

\subsection{Analysis of Synthetic Data}

We perform two experiments using matrices sampled from the planted submatrix model
to illustrate the theoretical sufficient conditions for perfect recovery of the densest submatrix given by Theorem~\ref{thm:suff-cond-random}. In each, we consider random matrices sampled from the planted submatrix model while varying problem parameters such as the density and size of the planted blocks, and record the number of times that we correctly identify a particular submatrix from the solution of~\eqref{eqn3.2}. This allows us to empirically model the probability of recovery of the hidden submatrix for each problem setting and, hence, illustrate the phase transitions to perfect recovery given in Theorem~\ref{thm:suff-cond-random} and Theorem~\ref{thm:suff-cond-adv}.

We have prepared Python\footnote{\href{https://github.com/bpames/DENSUB}{https://github.com/bpames/DENSUB}}, MATLAB\footnote{\href{https://github.com/pbombina/admmDSM}{https://github.com/pbombina/admmDSM}}, and R\footnote{\href{https://cran.r-project.org/web/packages/admmDensestSubmatrix/index.html}{https://cran.r-project.org/web/packages/admmDensestSubmatrix/index.html}} implementations of Alg.~\ref{alg:densub}. We use the Python implementation in the following experiments. All experiments were completed using a single node of the University of Southampton's Iridis 5 HPC cluster.

\subsubsection{Experiment 1}
\label{sec:expt-1}

We first consider $500\times 500$ matrices sampled from the planted dense submatrix model with $k = \lfloor M/m \rfloor$ blocks constructed as follows for each $(q,m)$ pairs with 
\begin{align*}
q &\in \{0.30, 0.35, 0.4, 0.45, 0.5, 0.55, 
               0.6, 0.65, 0.7, 0.75, 0.85, 0.9, 
               0.95, 1\}, \\
    m &\in \{50, 75, 100, 125, 150, 175, 200, 225, 250\}:
\end{align*}

\begin{itemize}
    \item 
    We partition the rows into $k-1$ sets $U_1, U_2, \dots, U_{k-1}$ of size $m$. The remaining $M - (k-1)m$ rows are assigned to the remaining block $U_k$. We partition the columns symmetrically.
    \item 
    We sample  $\A$ from the symmetric planted dense submatrix model according to this partition with probabilities
    \begin{align*} 
        p_{11} &= q,  \\
        p_{jj} &= \frac{1}{4(j - 1)} \;\;\text{for } j =2, 3, \dots, k,\\ 
        p_{ij} &= p_{ii} p_{jj} \;\; \text{if } i \neq j.
    \end{align*}
    Here, we sample $\A$ symmetrically so that $a_{ij} = a_{ji}$ for all $i, j \in [M]=[500]$.
\end{itemize}

This construction ensures that the $(U_1, U_1)$ block is the densest $m \times m$-submatrix in expectation. Moreover, Conditions~\eqref{eq:rec-size} and~\eqref{eq:rec-size2} hold by our choice of $m$ and $M$. 
If the gap $q - p_{22}$ is sufficiently large then~\eqref{eq:rec-gap} is satisfied and we should expect that the matrix representation of $(U_1,U_1)$ is the unique optimal solution for~\eqref{eqn3.2}. When this gap is small, i.e., $q$ is close to $1/4$, then the sufficient conditions given by Theorem~\ref{thm:suff-cond-random} are not satisfied and we may not have recovery.

This is illustrated by the examples given in Figure~\ref{fig:E1-Examples}. Each of the matrices $\A$ have a densest block indexed by $(U_1,U_1) = ([m],[m])$
for $m=100$. The remaining four diagonal blocks are less dense and the off-diagonal blocks are relatively sparse. We can create an undirected graph $G$ with adjacency matrix equal to each sampled matrix $\A$. These graphs have densest $m$-subgraphs indexed by $U_1$ in expectation.
When the gap between $q$ and $p_{22}$, the density of the next densest diagonal block, is small, we have two planted $(m,m)$-submatrices with very similar densities. We cannot expect to recover the densest $m$-subgraph or densest $m \times m$-submatrix in this case (see Figs.~\ref{fig:E1-hard-A} and~\ref{fig:E1-hard-G}). However, when the gap is larger, the corresponding graph $G$ consists of a single dense $m$-subgraph obscured by sparse noise in the form of diversionary edges and nodes. In this case, we should expect to identify the planted dense submatrix and subgraph from the solution of~\eqref{eqn3.2} (see Figs.~\ref{fig:E1-A} and Figs.~\ref{fig:E1-G}).

To empirically verify this predicted behaviour,
we generate $10$ problem instances $\A$ corresponding to each $(q,m)$-pair. For each of these matrices, we solve~\eqref{eqn3.2} with 
$$
\gamma=\frac{6}{m(q-\max\{p_{ij}\})}
$$
using the Python implementation of Alg.~\ref{alg:densub}.
We use the augmented Lagrangian parameter $\mu=2$, stopping tolerance $10^{-4}$, and maximum number of iterations $2000$ in each call to the solver.
We round the entries of $\X$ to the nearest integer and compare to the matrix representation of the planted submatrix indexed by $(U_1, U_1)$. We say that we have recovered the planted submatrix if 
$$
    \frac{\|\X - \X_0\|_F}{\|\X_0\|_F} < 10^{-3}. 
$$
We should expect a sharp transition from no recoveries in the 10 trials to recovery in all 10 trials as both $m$ and $q$ increase, as suggested by the sufficient condition~\eqref{eq:rec-gap}.

The results of this experiment are summarized in Figure~\ref{fig:E1-recovery}. 
Figure~\ref{fig:E1-recovery} includes the theoretical phase transition curve 
$$
    q = 0.25 +  \max \left\{\sqrt{\frac{M \log M}{3m}}, \frac{\log M}{m} \right\}
$$
used to mimic~\eqref{eq:rec-gap}.
The observed recovery rates closely match that predicted by Theorem~\ref{thm:suff-cond-random}. Indeed, when $q$ and $m$ satisfy~\eqref{eq:rec-gap}, the solution computed by Alg.~\ref{alg:densub} is exactly the matrix representation of the planted submatrix in all trials.

\begin{figure}[t]
    \centering
    \subfloat[Easy Problem $\A$]{    
        \includegraphics[width=0.23\textwidth]{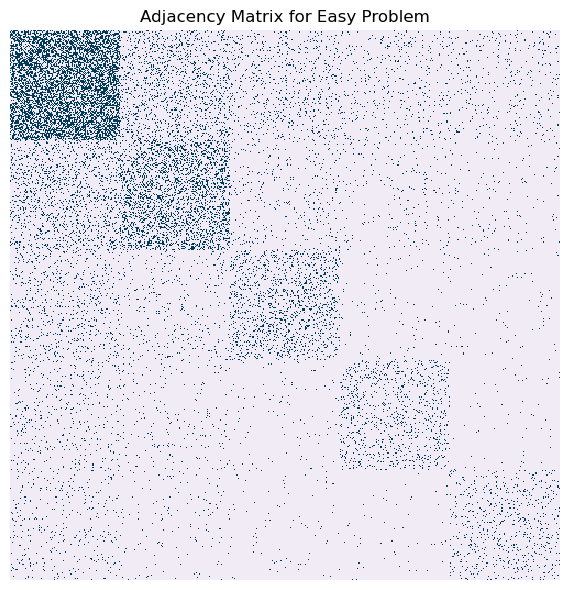}
        \label{fig:E1-A}
        }    
    \hfill
    \subfloat[Easy Problem $G$]{       
        \includegraphics[width=0.23\textwidth]{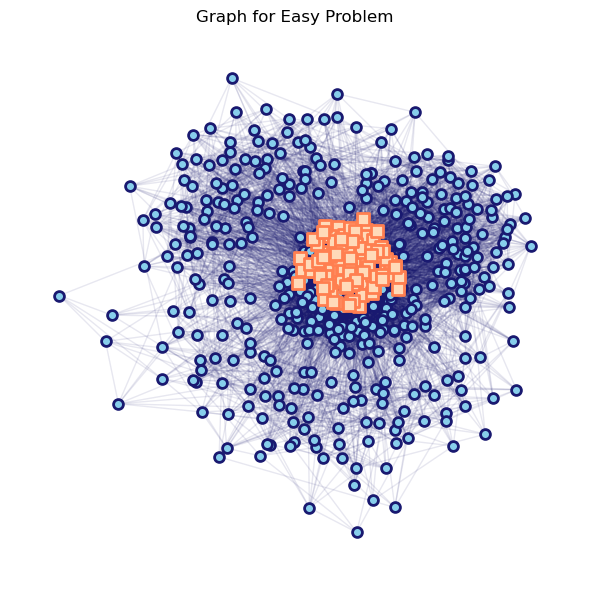}        
        \label{fig:E1-G}
        }    
    \subfloat[Hard Problem $\A$]{
        \includegraphics[width=0.23\textwidth]{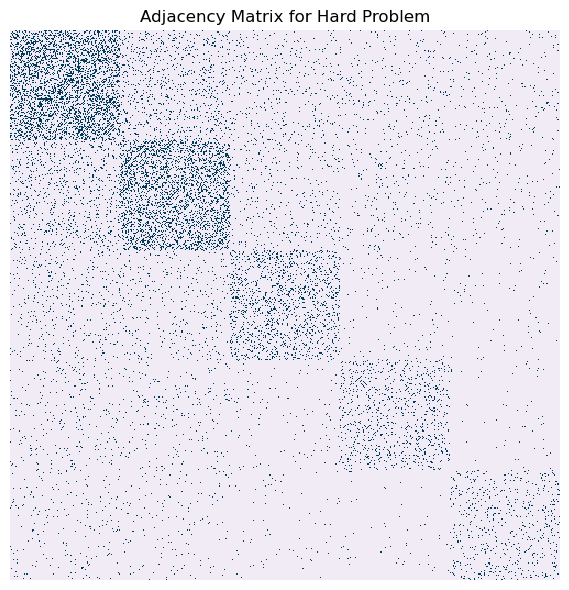}
        \label{fig:E1-hard-A}
        }
    \subfloat[Hard Problem $G$]{
        \includegraphics[width=0.23\textwidth]{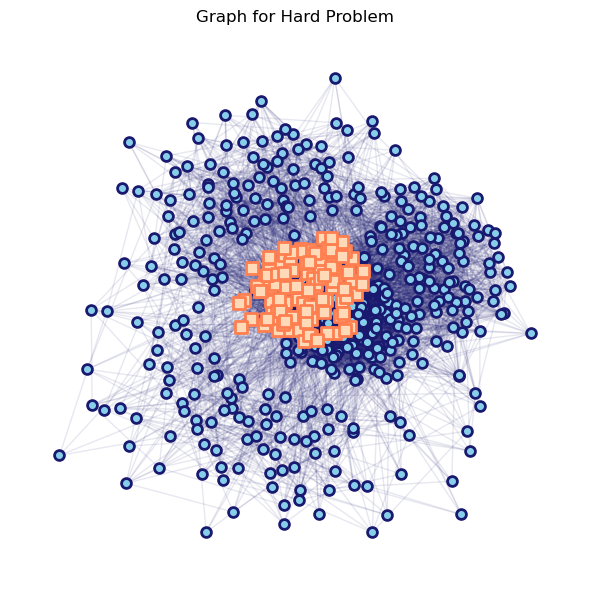}        
        \label{fig:E1-hard-G}
        }    
    \caption{Examples of easy and hard to solve instances of the densest submatrix problem sampled in Experiment 1. Here, {\color{orange}orange squares} indicate the nodes $U_1$ in the graph $G$ with adjacency matrix $\A$ that induce the planted subgraph.}
    \label{fig:E1-Examples}
\end{figure}
\subsubsection{Experiment 2}
\label{sec:expt-2}

We next consider problem instances of the densest $m \times m$-submatrix problem constructed as follows for each $(q,m)$-pair in
\begin{align*}
q &\in \{0.25, 0.30, 0.35, 0.4, 0.45, 0.5, 0.55, 
               0.6, 0.65, 0.7, 0.75, 0.85, 0.9, 
               0.95, 1 \} \\
    m&\in \{200, 225, 250, 275, 300, 325, 350, 375\}.
\end{align*}
\begin{itemize}
    \item 
    We divide $[M] = [500]$ into $k=2$ blocks $U_1, U_2$ of size $m$ and $M-m=500-m$, respectively. 
    \item 
    We sample matrix $\A$ from the planted dense submatrix model with probabilities 
    \begin{align*}
        p_{11} = p_{22} &= q \\ 
        p_{12} = p_{21} &= 0.25.
    \end{align*}
\end{itemize}
By construction, $(U_1, U_1)$ indexes the densest $m \times m$-submatrix and $U_1$-subgraph of the corresponding undirected graph $G$ when $m > M/2$.
If $m < M/2$ then the matrix contains several $(m,m)$-submatrices as dense as that indexed by $(U_1,U_1)$; in this case, it is impossible to guarantee that we will recover the matrix representation of $(U_1, U_1)$ as the optimal solution of~\eqref{eqn3.2}; see Figs.~\ref{fig:E2-impossible-A} and~\ref{fig:E2-impossible-G} for an example of such a matrix and graph.
On the other hand, the sufficient conditions for recovery given by Theorem~\ref{thm:suff-cond-random} and~\eqref{eq:rec-gap}, in particular, may not hold if 
\begin{itemize}
    \item $q \approx 0.25$, as illustrated by Figs.~\ref{fig:E2-hard-A} and~\ref{fig:E2-hard-G}; or 
    \item $m \approx M-m$, as illustrated by Figs.~\ref{fig:E2-hard2-A} and~\ref{fig:E2-hard2-G}.
\end{itemize}
In the first case, the density in the diagonal blocks is close to that of the off-diagonal blocks and we should not expect to distinguish between them. In the second, the sizes of the two diagonal blocks are close and, again, we should not expect to recover $(U_1,U_1)$ as the index set of the densest $m \times m$-submatrix.
However, the sufficient conditions given by Theorem~\ref{thm:suff-cond-random} suggest that we will have perfect recovery of $(U_1,U_1)$ from the optimal solution of~\eqref{eqn3.2} w.h.p.~if $m$ and $q$ are sufficiently large (see Figs.~\ref{fig:E2-easy-A} and~\ref{fig:E2-easy-G}).

As before, we empirically verify the phase transition to perfect recovery by 
generating $10$ problem instances $\A$ corresponding to each $(q,m)$-pair. For each of these matrices, we solve~\eqref{eqn3.2} with 
$$
\gamma=\frac{6}{m(q-p_{12})} = \frac{6}{m(q-0.25)}
$$
using the Python implementation of Alg.~\ref{alg:densub}.
We use the augmented Lagrangian parameter $\tau=2$, stopping tolerance $10^{-4}$, and maximum number of iterations $2000$ as in the previous experiment.
We record each computed solution $\X$ after rounding and compare to the matrix representation of the planted submatrix indexed by $(U_1, U_1)$. We say that we have recovered the planted submatrix if 
$$
    \frac{\|\X - \X_0\|_F}{\|\X_0\|_F} < 10^{-3}
$$
and count the number of recoveries of $\X_0$ for each $(q,m)$-pair.

The results of this experiment are summarized in Figure~\ref{fig:E2-recovery}; we plot the theoretical phase transition curve 
$$
    q = 0.25 + 6 \, \max \left\{\sqrt{\frac{M \log M}{3m}}, \frac{\log M}{m} \right\}
$$
given by~\eqref{eq:rec-gap} with constant $c_1 = 6$.
As in Section~\ref{sec:expt-1}, we have empirical recovery rates that match those predicted by Theorem~\ref{thm:suff-cond-random}. Again, when $q$ and $m$ satisfy~\eqref{eq:rec-gap} we have perfect recovery of $\X_0$ in almost all trials. The reported recovery rates are perhaps overly pessimistic in this case. We record recovery of the submatrix indexed by $(U_1, U_1)$. When $M >> m$, this submatrix is not the unique densest $m \times m$-submatrix in expectation. In this case, the solution $\X$ is a convex combination of several $m \times m$-submatrices of similar densities. We will see in the next section that we can identify the submatrices comprising these solutions using simple rounding techniques in certain cases.

\begin{figure}[t]
    \centering
    \subfloat[Easy Problem $\A$]{
        \includegraphics[width=0.23\linewidth]{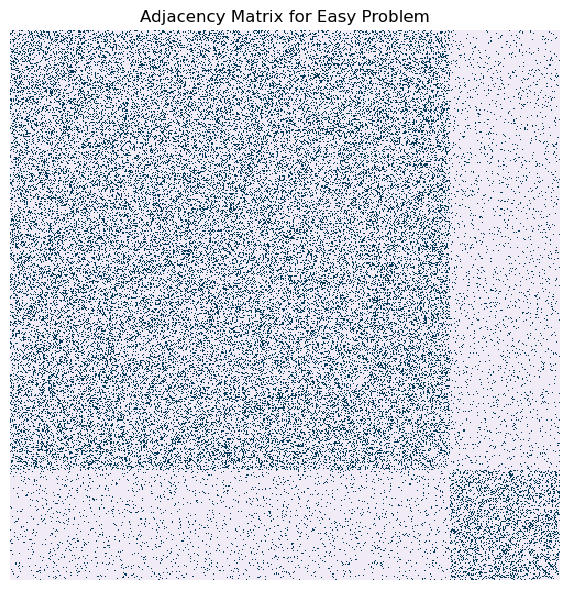}
        \label{fig:E2-easy-A}
    }
    \subfloat[Easy Problem $G$]{
        \includegraphics[width=0.23\linewidth]{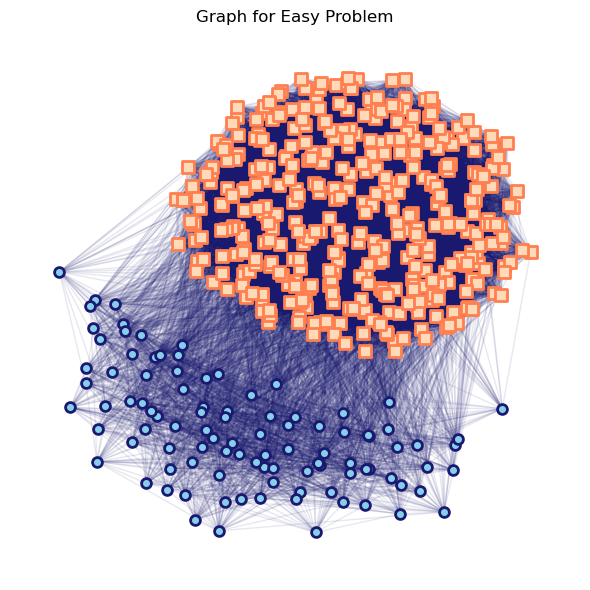}
        \label{fig:E2-easy-G}
    }
    \subfloat[Hard Problem $\A$]
        {
        \includegraphics[width=0.23\linewidth]{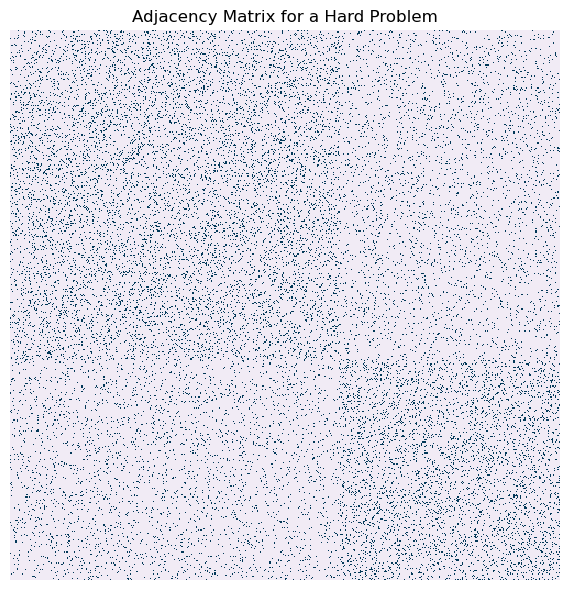}        
        \label{fig:E2-hard-A}
        }
    \subfloat[Hard Problem $G$]{
        \includegraphics[width=0.23\linewidth]{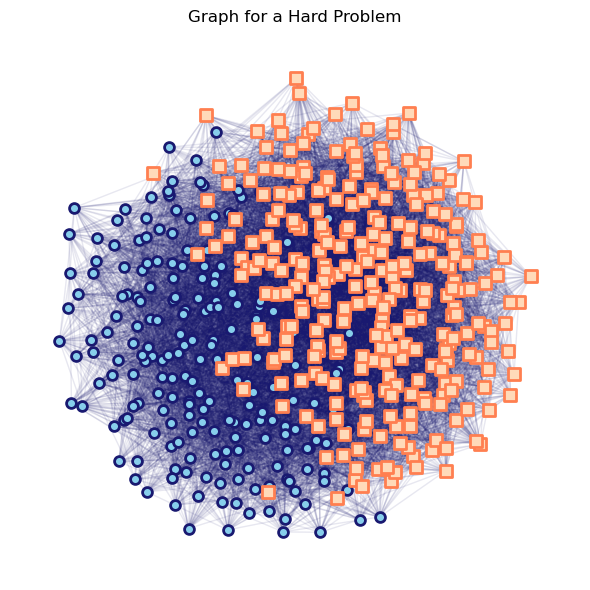}
        \label{fig:E2-hard-G}
        }

    \subfloat[Another Hard $\A$]{
        \includegraphics[width=0.23\linewidth]{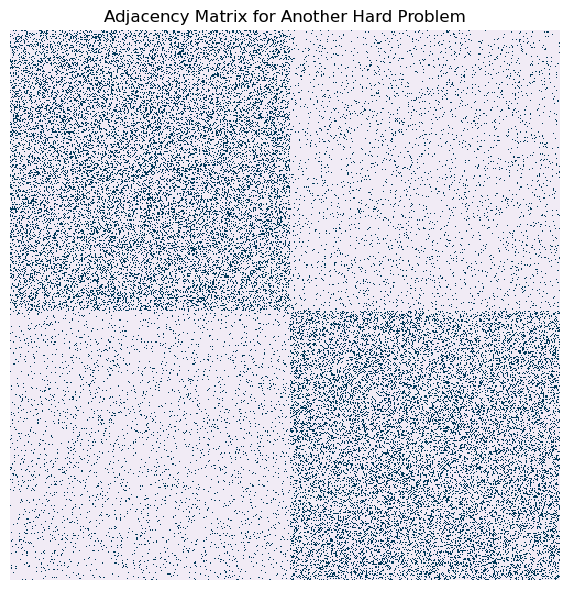}
        \label{fig:E2-hard2-A}
        }
    \subfloat[Another Hard $G$]{
        \includegraphics[width=0.23\linewidth]{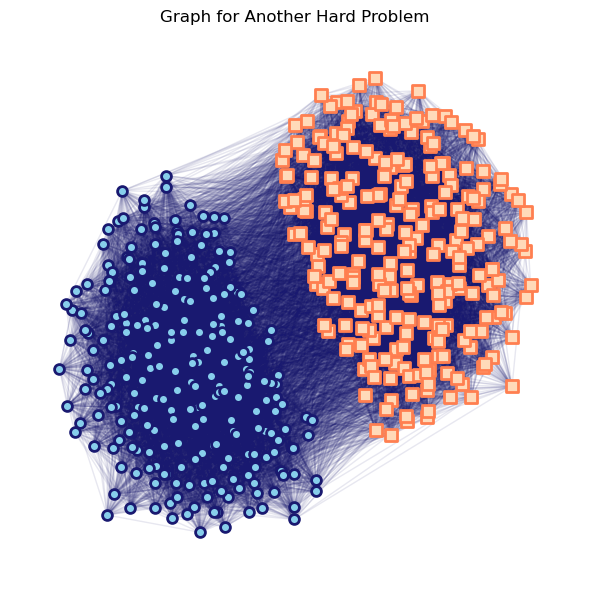}        
        \label{fig:E2-hard2-G}
        }
    \subfloat[Impossible Recovery $\A$]{
        \includegraphics[width=0.23\linewidth]{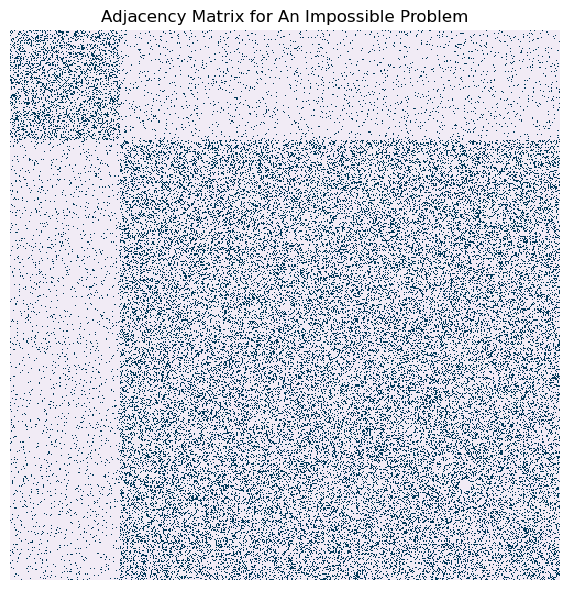}
        \label{fig:E2-impossible-A}
        }
    \subfloat[Impossible Recovery $G$]{
        \includegraphics[width=0.23\linewidth]{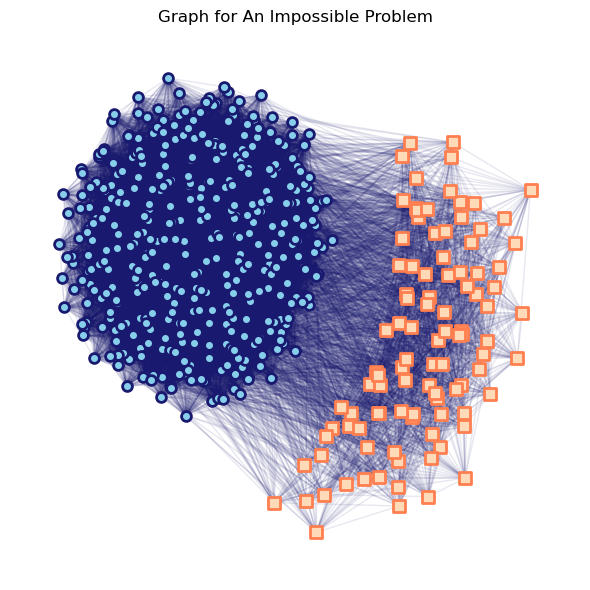}
        \label{fig:E2-impossible-G}
    }
    \caption{Examples of easy, hard, and impossible graphs for recovery of the planted submatrix sampled in Experiment 2. Here, {\color{orange}orange squares} indicate the nodes $U_1$ in the graph $G$ with adjacency matrix $\A$ that induce the planted subgraph.}
    \label{fig:E2-Examples}
\end{figure}

\begin{figure}[t]
    \centering
    \subfloat[Recovery Counts for Experiment 1]
        {
        \includegraphics[width=0.48\linewidth]{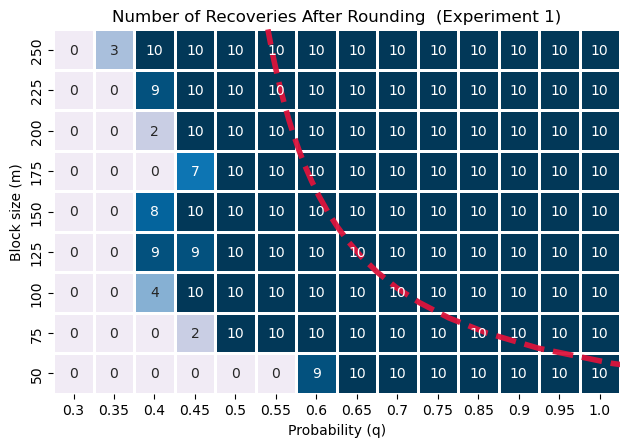}        
        \label{fig:E1-recovery}
        }
    \hfill
    \subfloat[Recovery Counts for Experiment 2]{
        \includegraphics[width=0.48\linewidth]{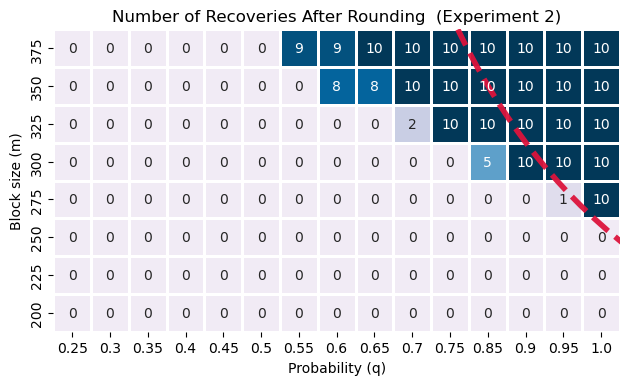}
        \label{fig:E2-recovery}
        }    
    \caption{Recovery counts for Experiments 1 and 2. Darker squares indicate more recoveries out of 10 trials for each $(q,m)$-pair.
    The dashed curve indicates the phase transition curve to perfect recovery given by the right-hand side of~\eqref{eq:rec-gap}}.
    \label{fig:E1and2-Recovery}
\end{figure}

\subsection{Identifying Maximum Cliques in Real-World Networks}
\label{sec:real-data}

\subsubsection{Benchmarking Networks}

We apply Alg.~\ref{alg:densub} to analyse several widely studied benchmarking network data sets. 
The maximum clique has been previously identified for each of these networks and is known. We will illustrate that Alg.~\ref{alg:densub} recovers the maximum clique for each of the following networks:
\begin{itemize}
    \item \emph{The Jazz Collaboration Network}~\cite{gleiser2003community} is a collaboration network where two musicians are connected if they have performed together. It contains a single maximum clique of size $30$.
    \item \emph{Zachary's Karate Club}~\cite{zachary1977information} is a social network of 34 members of a karate club in the 1970s. It contains two maximum cliques of size $5$.
    \item \emph{Dolphins}~\cite{lusseau2003bottlenose} is a social network of 62 bottlenose dolphins living of Doubtful Sound, New Zealand. It contains three maximum cliques of size $5$.
    \item \emph{Les Miserables}~\cite{knuth1993stanford} is a co-appearance network of characters in the novel \emph{Les Miserables} containing two maximum cliques of length $10$.
\end{itemize}

For each network, we apply our Python implementation of  Alg.~\ref{alg:densub} to solve the corresponding instance of~\eqref{eqn3.2}
with $\gamma = 12/m$, where $m = n$ is chosen to be the known size of the maximum clique in the network.
We use augmented Lagrangian parameter $\tau=2$ and stop the iterative process if an $\epsilon=10^{-4}$ suboptimal solution is found or $2000$ iterations have been performed.

For each of these networks, we successfully identify the maximum clique(s) present within the network. Indeed, the solution $\X$ returned by Alg.~\ref{alg:densub} when applied to the Jazz Collaboration Network is exactly the rank-one matrix representation of the maximum clique in this graph. On the other hand, each of the remaining networks contain multiple maximum cliques. In each case, the solution $\X$ returned by Alg.~\ref{alg:densub} is a convex combination of the matrix representations of each maximum clique. Simple rounding processes, e.g., taking the $m$ largest nonzero entries of the diagonal of $\X$, reveals one of the maximum cliques for each of these networks; we can easily identify the remaining cliques by inspection using the clique found in this rounding step. Figure~\ref{fig:real-data-1} provides a visualisation of these results. For each of these networks, we give a visualisation of the adjacency matrix, the optimal solution of~\eqref{eqn3.2} found by Alg.~\ref{alg:densub}, and the maximum clique found by the rounding procedure described above.

\begin{figure}[t]
    \centering
    \subfloat[Adjacency matrix of JAZZ]{        
        \includegraphics[width=0.25\linewidth]{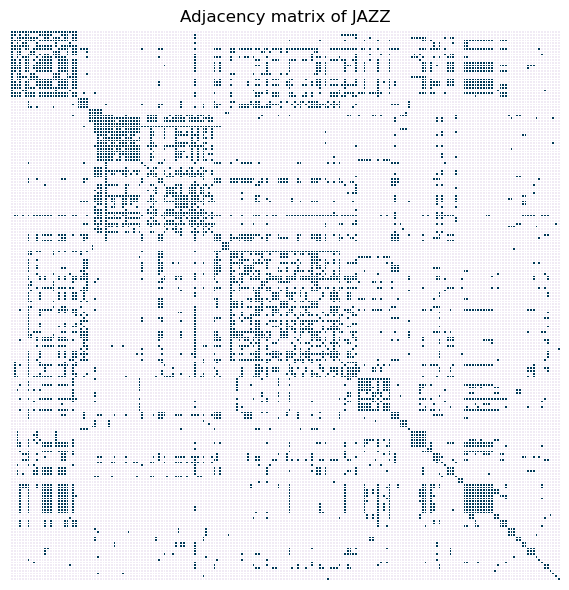}        
        \label{fig:Ajazz}
    }
    \subfloat[Recovered Solution $\X$ for JAZZ]{
        \includegraphics[width=0.25\linewidth]{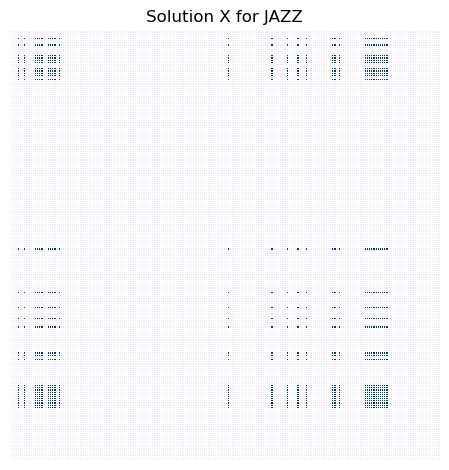}
        \label{fig:Xjazz}
    }
    \subfloat[Recovered clique in JAZZ]{
        \includegraphics[width=0.25\linewidth]{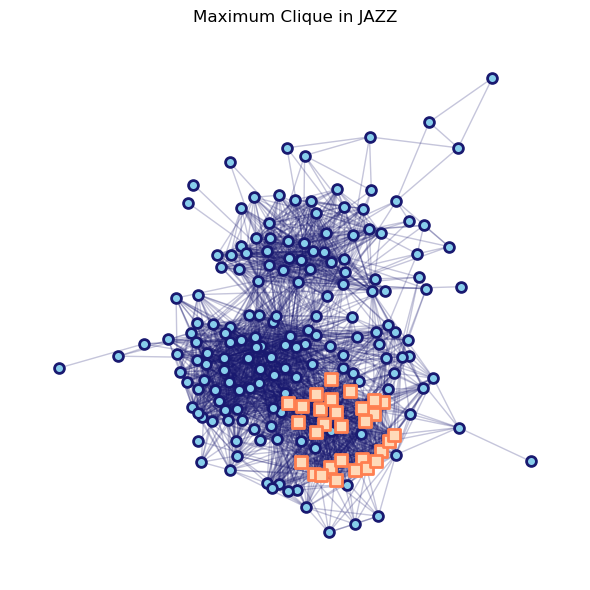}
        \label{fig:clique-jazz}
    }
    
    \subfloat[Adjacency matrix of KARATE]{
        \includegraphics[width=0.25\linewidth]{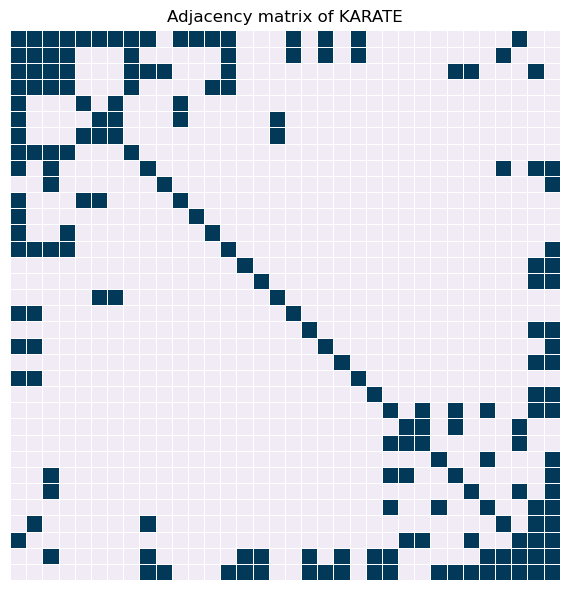}
        \label{fig:Akarate}
    }
    \subfloat[Solution $\X$ for KARATE]
        {
        \includegraphics[width=0.25\linewidth]{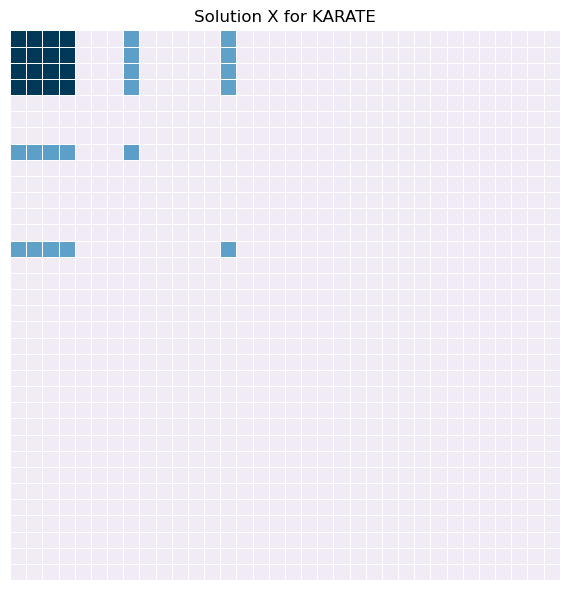}
        \label{fig:Xkarate}
    }
    \subfloat[Recovered clique in KARATE]{
        \includegraphics[width=0.25\linewidth]{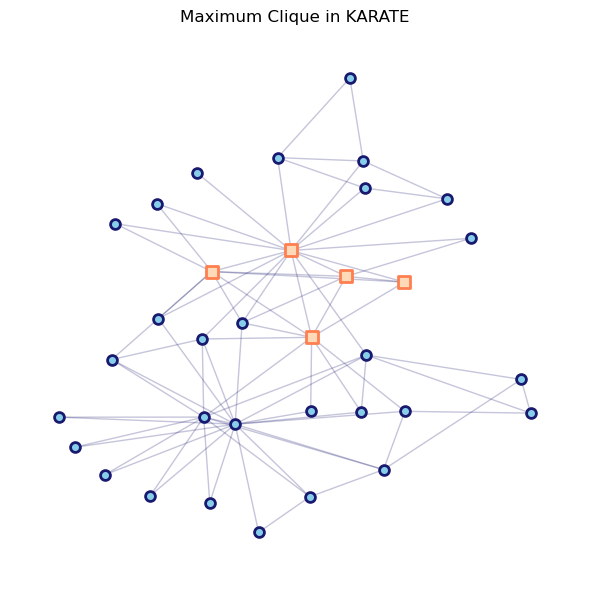}
        \label{fig:clique-karate}
    }

    \subfloat[Adjacency matrix of DOLPHINS]{
        \includegraphics[width=0.25\linewidth]{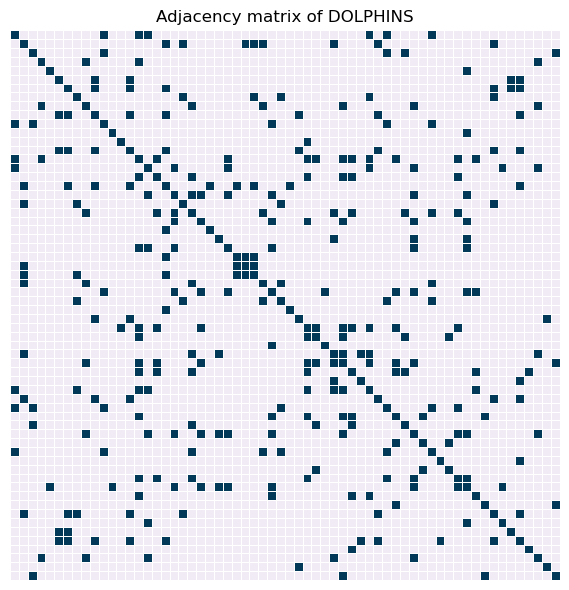}
        \label{fig:Adolphins}
    }
    \subfloat[Solution $\X$ for DOLPHINS]
        {
        \includegraphics[width=0.25\linewidth]{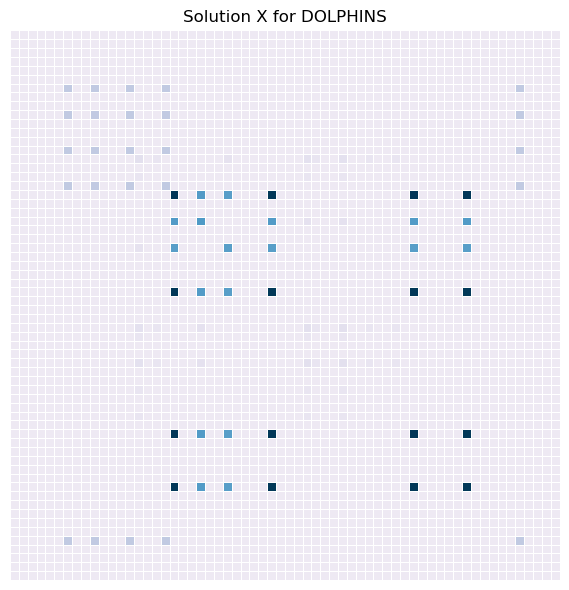}
        \label{fig:Xdolphins}
    }
    \subfloat[Recovered clique in DOLPHINS]{
        \includegraphics[width=0.25\linewidth]{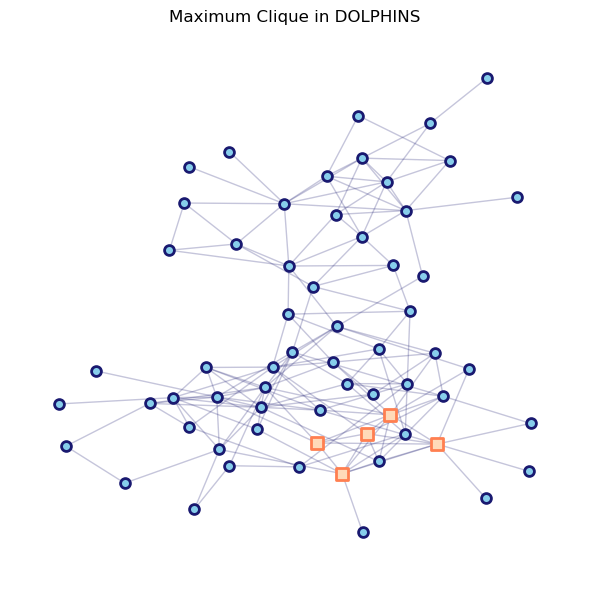}
        \label{fig:clique-dolphins}
    }

    \subfloat[Adjacency matrix of LESMIS]{
        \includegraphics[width=0.25\linewidth]{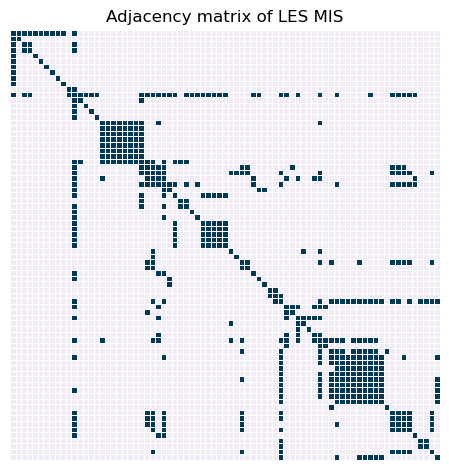}
        \label{fig:Alesmis}
    }
    \subfloat[Solution $\X$ for LESMIS]
        {
        \includegraphics[width=0.25\linewidth]{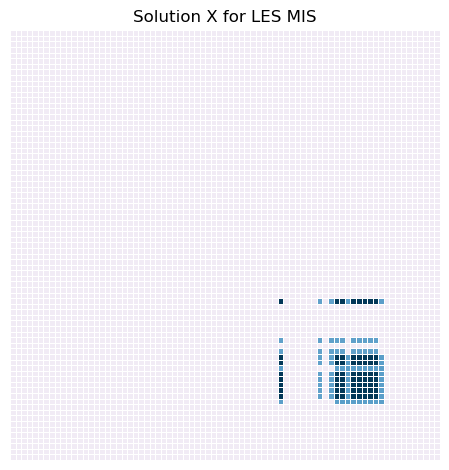}
        \label{fig:Xlesmis}
    }
    \subfloat[Recovered clique in LESMIS]{
        \includegraphics[width=0.25\linewidth]{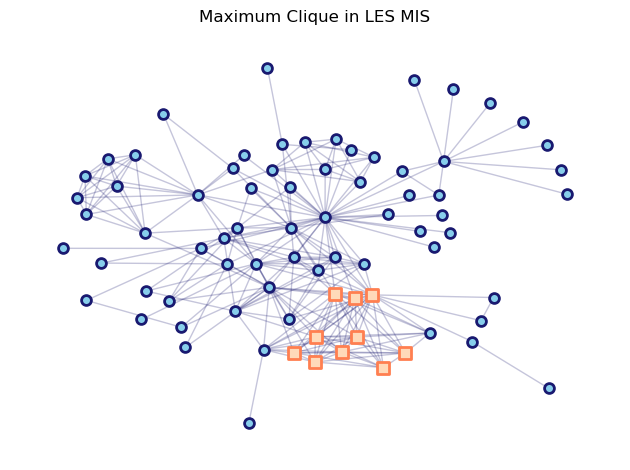}
        \label{fig:clique-lesmis}
    }
    
    \caption{Adjacency matrices of the JAZZ, KARATE, DOLPHINS, and LESMIS networks and recovered maximum cliques.}
    \label{fig:real-data-1}
\end{figure}

\subsubsection{Maximum Cliques of the \emph{A Song of Ice and Fire} Series}
\label{sec:got-data}

We conclude by applying Alg.~\ref{alg:densub} to identify the largest clique in each of the 5 novels in George R.R.~Martin's \emph{A Song of Ice and Fire (ASOIAF)} saga using the character interaction networks considered in \cite{beveridge2016network}; Each network, one for each book in series, has nodes corresponding to characters and edge weights corresponding to the number of times two character names appear within 15 words of one another, which models the number of interactions between each pair of characters in the book. A detailed discussion of the properties of these interaction networks, as well as all network data are available at~\cite{networkofthronesNetworkThrones}. A similar analysis of the television adaptation of \emph{ASOIAF}, HBO's \emph{Game of Thrones}, is presented in~\cite{beveridge2018game}.

We perform the following preprocessing steps to the character interaction network for each book:
\begin{itemize}
    \item
    We first \emph{binarize} the network by setting any edge weight greater than a pre-chosen threshold $t$ equal to $1$ and all other edge weights equal to $0$. This yields an \emph{undirected, unweighted} character interaction network, where two characters are adjacent if they have at least $t$ interactions in a given book.
    We use threshold $t=5$ for Book 5, and threshold $t=10$ for remaining the books.
    \item 
    We then form a final interaction network $G$ where two characters are adjacent if there is at least one walk of length $2$ between them in the binarized interaction network.   
\end{itemize}

For each interaction network $G$, we use the \texttt{find\_cliques}\footnote{\href{https://networkx.org/documentation/stable/reference/algorithms/generated/networkx.algorithms.clique.find_cliques.html}{https://networkx.org/documentation/stable/reference/algorithms/generated/networkx.algorithms.clique.find\_cliques.html}} function from the \texttt{NetworkX} library to calculate the set of all maximal clique of $G$ using the algorithms developed in \cite{bron1973algorithm, tomita2006worst, cazals2008note}; after calculating \emph{all} maximal cliques, we find the maximum clique as the maximal clique with largest cardinality. 
We then apply the Python implementation of  Alg.~\ref{alg:densub} to solve the corresponding instance of~\eqref{eqn3.2} 
with the adjacency matrix of $G$ as input and $\gamma = 12/m$, where $m=n$ is equal to size of the maximum clique within the character interaction network $G$. As in the experiments in Sect.~\ref{sec:real-data}, we employ augmented Lagrangian parameter $\tau=2$, stopping tolerance $10^{-4}$, and a maximum of $2000$ iterations.

We correctly identify the unique maximum clique in the interaction network $G$ for Books 1, 2, 3, and 5. Book 4 has two overlapping maximum cliques of size $5$, and we recover a convex combination $\X$ of the matrix representations of these cliques; the rounding procedure described in Sect.~\ref{sec:real-data} correctly identifies the two maximum cliques in the interaction network for Book 4.

The maximum clique in each interaction network represents the largest collection of fully connected characters in the corresponding book. Here, the membership of these cliques reflects the dynamics of the series of novels. In the first book, we observe a single large community primarily comprised of the Baratheon, Stark, and Lannister families. As the saga progresses, the action of the novels scatters these characters across many disparate locations. This corresponds to interaction networks consisting of much smaller cliques and clusters. For each book and interaction network, we successfully identified the largest of these cliques, but further analysis is required to find all of the underlying communities. This analysis aligns with the more detailed analysis of the network dynamics of the \emph{ASOIAF} series found in~\cite{beveridge2016network, beveridge2018game, networkofthronesNetworkThrones}.

\begin{figure}[t]
    \centering
    \subfloat[Adjacency matrix for Book 1]{
        \includegraphics[width=0.24\linewidth]{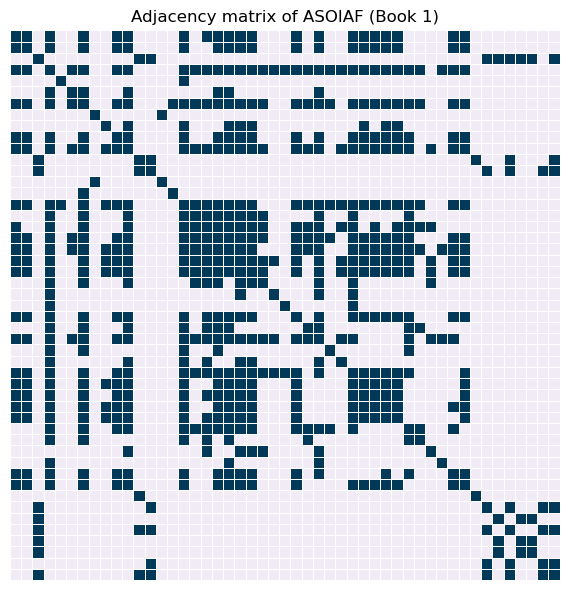}
        \label{fig:A1}
    }
    \subfloat[Solution $\X$ for Book 1]{
        \includegraphics[width=0.24\linewidth]{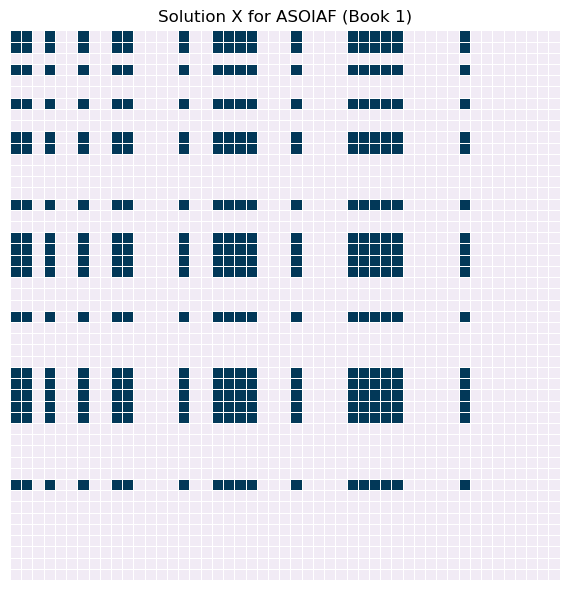}
        \label{fig:X1}
        }    
    \subfloat[Recovered clique for Book 1]{        
        \includegraphics[width=0.24\linewidth]{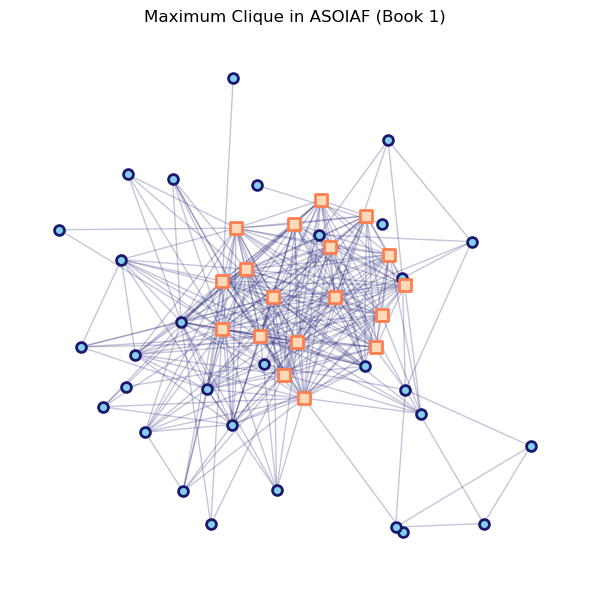}        
        \label{fig:clique-book1}
        }

    \subfloat[Adjacency matrix for Book 2]{
        \includegraphics[width=0.24\linewidth]{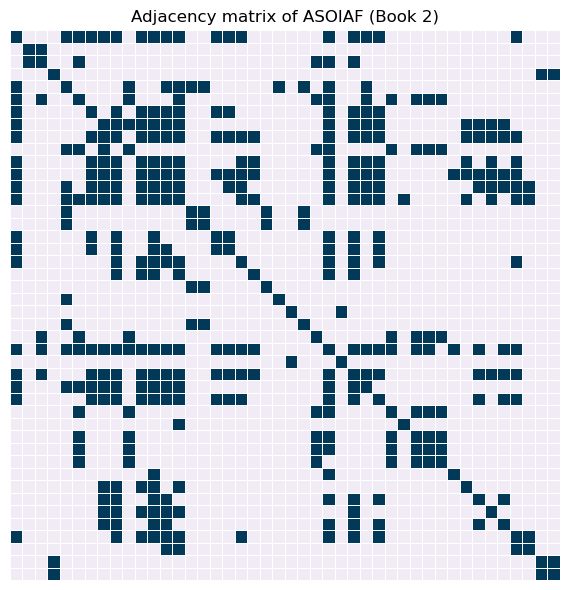}
        \label{fig:A2}
    }
    \subfloat[Solution $\X$ for Book 2]{
        \includegraphics[width=0.24\linewidth]{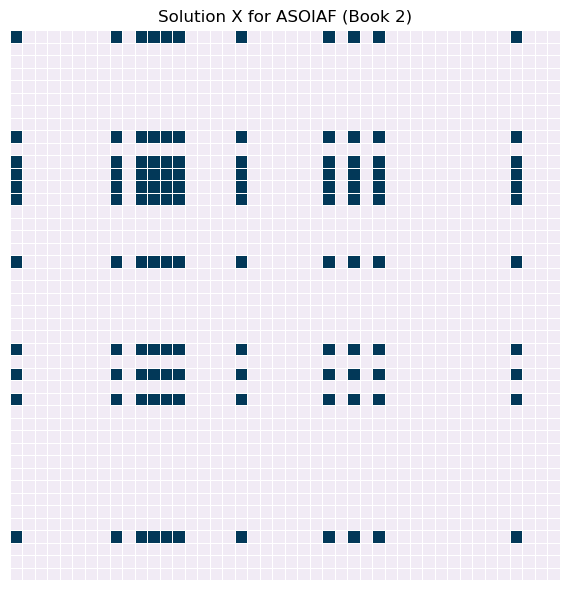}
        \label{fig:X2}
        }    
    \subfloat[Recovered clique for Book 2]{        
        \includegraphics[width=0.24\linewidth]{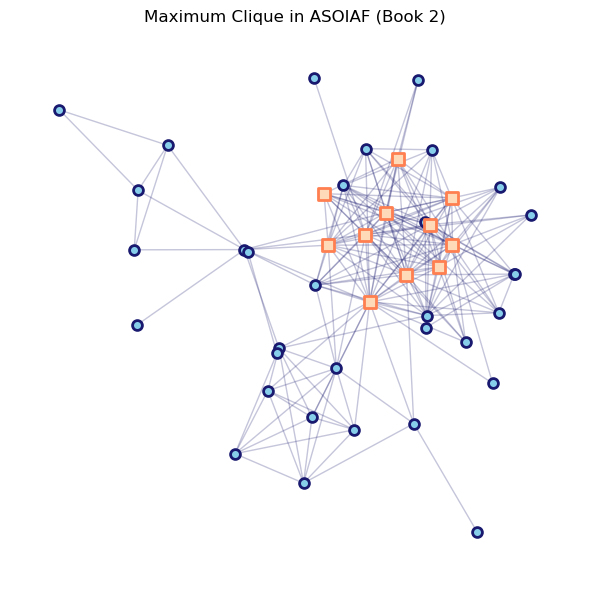}        
        \label{fig:clique-book2}
        }

    \subfloat[Adjacency matrix for Book 3]{
        \includegraphics[width=0.24\linewidth]{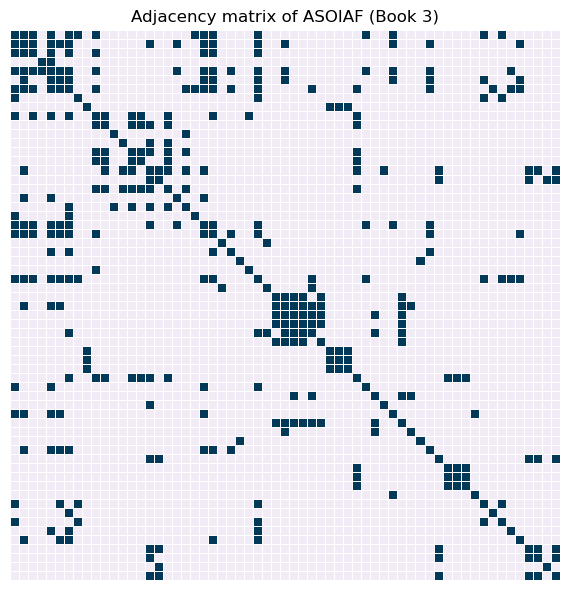}
        \label{fig:A3}
    }
    \subfloat[Solution $\X$ for Book 3]{
        \includegraphics[width=0.24\linewidth]{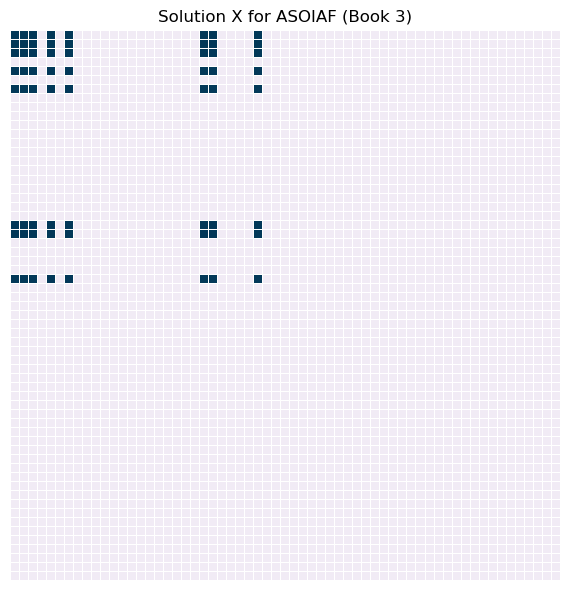}
        \label{fig:X3}
        }    
    \subfloat[Recovered clique for Book 3]{        
        \includegraphics[width=0.24\linewidth]{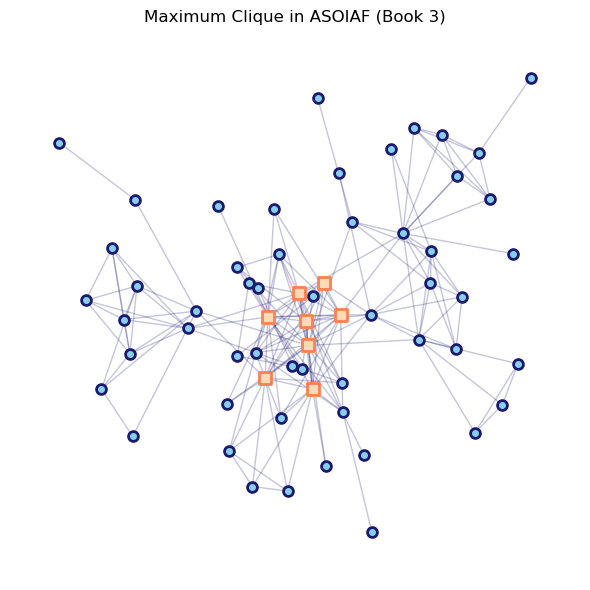}        
        \label{fig:clique-book3}
        }

    \subfloat[Adjacency matrix for Book 5]{
        \includegraphics[width=0.24\linewidth]{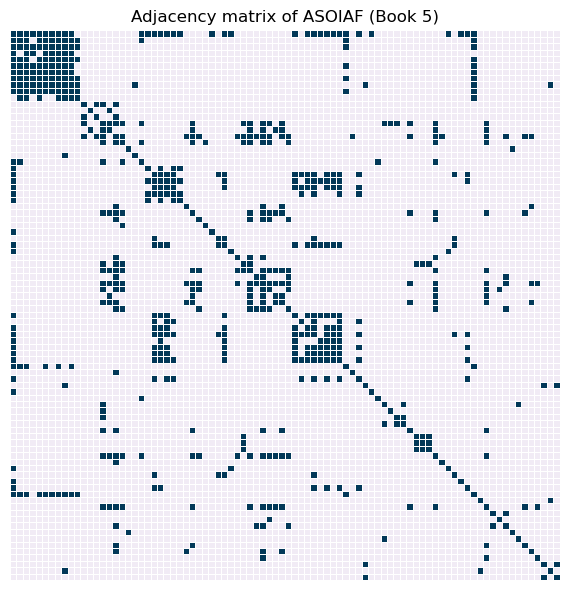}
        \label{fig:A5}
    }
    \subfloat[Solution $\X$ for Book 5]{
        \includegraphics[width=0.24\linewidth]{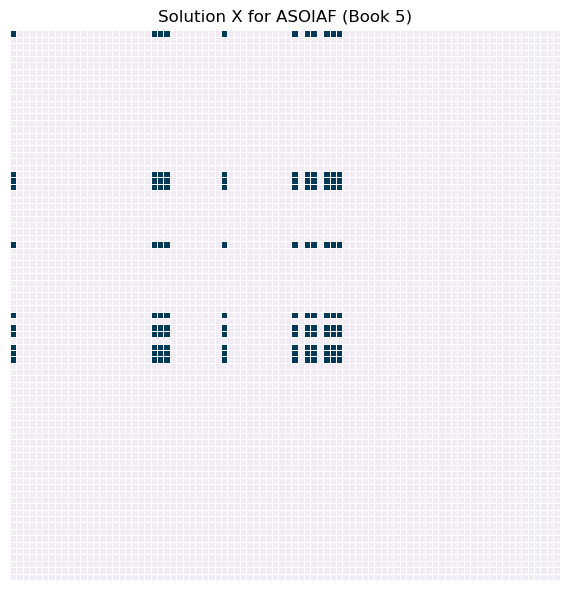}
        \label{fig:X5}
        }    
    \subfloat[Recovered clique for Book 5]{        
        \includegraphics[width=0.24\linewidth]{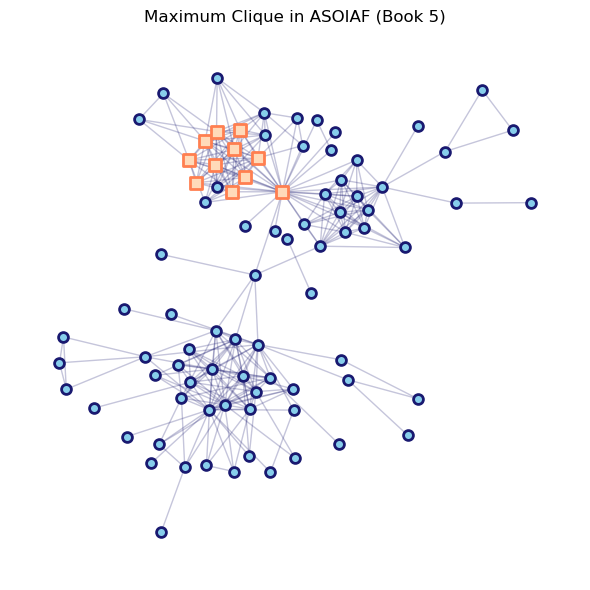}        
        \label{fig:clique-book5}
        }    
    \caption{Adjacency matrices of the character interaction networks for Books 1,2,3, and 5 of the \emph{ASOIAF} saga and recovered maximum cliques.}
    \label{fig:SOIAF1235}
\end{figure}

\begin{figure}[t]
    \centering
    \subfloat[Adjacency matrix for Book 4]{
        \includegraphics[width=0.48\linewidth]{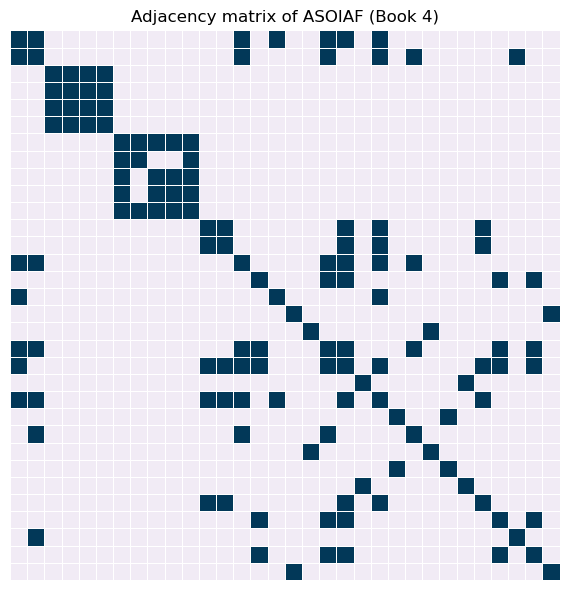}
        \label{fig:A4}
    }
    \subfloat[Solution $\X$ for Book 4]
        {
        \includegraphics[width=0.48\linewidth]{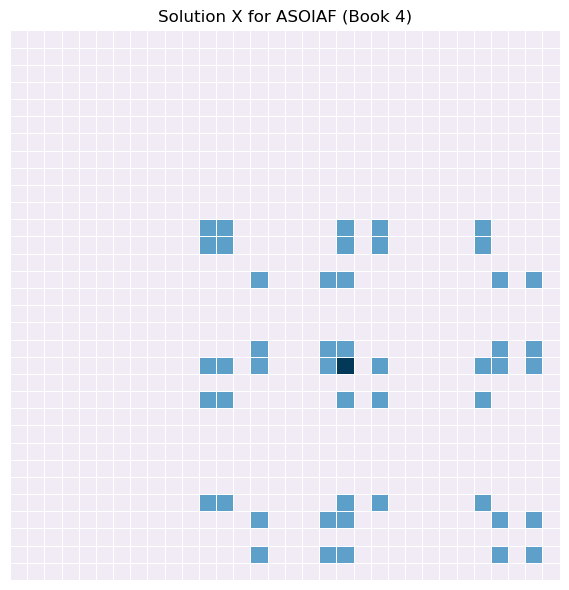}
        \label{fig:X4}
        }
    
    \subfloat[A maximum clique for Book 4]{
        \includegraphics[width=0.48\linewidth]{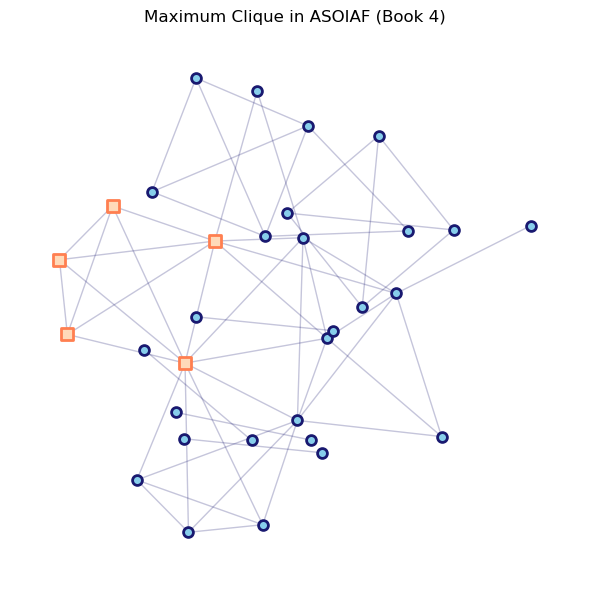}        
        \label{fig:clique-book4a}
        }    
    \subfloat[Another maximum clique for Book 4]{
        \includegraphics[width=0.48\linewidth]{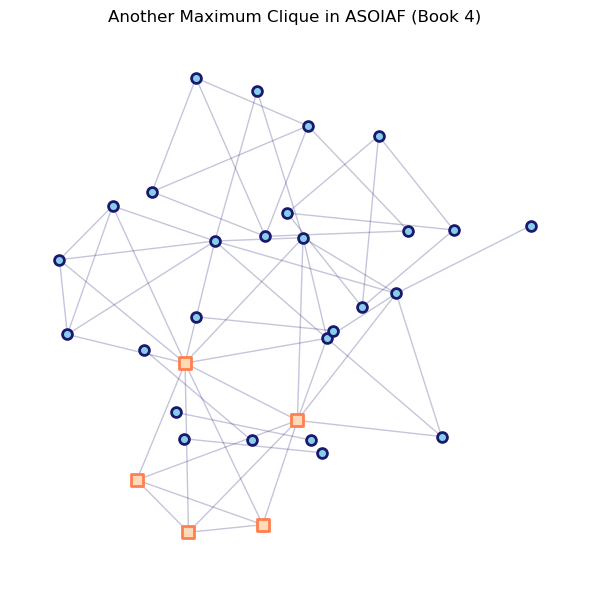}        
        \label{fig:clique-book4b}
    }
    
    \caption{Adjacency matrix of the character interaction network for Book 4  of the \emph{ASOIAF} saga with optimal solution $\X$ of~\eqref{eqn3.2}. We obtain two maximum cliques in this network from $\X$.}
    \label{fig:SOIAF4}
\end{figure}

\begin{table}[h]
    \centering
    \begin{tabular}{ccc} \toprule 
         Book & $\omega(G)$ & Maximum Clique(s)  \\ \midrule
         1 & 18 &  Jon-Arryn, Varys, Barristan-Selmy, Pycelle, Stannis-Baratheon, Jaime-Lannister,\\  
         & & Tywin-Lannister, Petyr-Baelish, Renly-Baratheon, Sansa-Stark, Robert-Baratheon, \\ 
         && Eddard-Stark, Bran-Stark,  Jon-Snow, Tyrion-Lannister, Catelyn-Stark, \\ 
         && Joffrey-Baratheon, Cersei-Lannister \\ \midrule
         2 & 11 & Stannis-Baratheon, Jaime-Lannister, Joffrey-Baratheon, Sansa-Stark,    \\
         && Robb-Stark, Renly-Baratheon, Tyrion-Lannister, Cersei-Lannister, Varys,  \\
         && Sandor-Clegane, Petyr-Baelish \\ \midrule
         3 & 8 & 
         Kevan-Lannister, Jaime-Lannister, Sansa-Stark, Joffrey-Baratheon,   \\ 
         && Tywin-Lannister, Robb-Stark,  Cersei-Lannister, Tyrion-Lannister \\ \midrule 
         4 & 5 & Tommen-Baratheon, Tyrion-Lannister, Tywin-Lannister, Robert-Baratheon,  \\ && Osmund-Kettleblack  \\ 
         && Tommen-Baratheon, Loras-Tyrell, Osney-Kettleblack, High-Sparrow, Taena-of-Myr \\ \midrule 
         5 & 11 & Daario-Naharis, Skahaz-mo-Kandaq, Missandei, Reznak-mo-Reznak, Belwas,  \\ 
         &&Hizdahr-zo-Loraq, Barristan-Selmy, Grey-Worm, Quentyn-Martell, Galazza-Galare, \\ && Daenerys-Targaryen  \\\bottomrule
    \end{tabular}
    \caption{Maximum cliques and size of maximum clique $\omega(G)$ in the character interaction network for each book in the \emph{ASOIAF} series. We report both maximum cliques identified for Book 4.}
    \label{tab:asoiaf-cliques}
\end{table}

\subsection{Choice of Regularization Parameter}

The sufficient conditions for perfect recovery given in Section~\ref{sec:DSM} hold for choice of the regularization parameter $\gamma$ within a certain range. The endpoints of this interval are proportional to $1/\sqrt{m_1 n_1}$ and depend on the probabilities $p_{11}$ and $p^*$. In practice, we may not know these probabilities. 

In our experiments involving synthetic data described in Sections~\ref{sec:expt-1} and~\ref{sec:expt-2}, we use the value of $p_{11}$ and $p^*$ used to construct our random matrices. Specifically, we choose
$$
\gamma = \frac{6}{(p_{11} - p^*)\sqrt{mn}} = \frac{6}{\sqrt{mn}(q-p^*)},
$$
since the numerator $6$ is in the range of coefficients satisfying~\eqref{eq:gamma-range}.
In our analysis of real-world networks in Section~\ref{sec:real-data}, we try to identify the densest submatrix with area equal to the square of the size of the maximum clique of each network. Here, we can assume that $p_{11} = 1$ since the block indexed by the maximum clique is fully dense. On the other hand, we use the rough estimate $p^* \approx 1/2$ in our formula for $\gamma$:
$$
\gamma = \frac{12}{m} = \frac{6}{m/2} \approx \frac{6}{m(p_{11} - p^*)}. 
$$
In practice, if either $p_{11}$ or $p^*$ are not known, we suggest using simple estimation procedures combined with cross-validation to choose $\gamma$. For example, $p_{11}$ can be approximated by normalizing the maximum row or column sum of $\A$, while $p^*$ can be approximated using the density of the graph or a similar function of row and column sums of $\A$. We can then use cross-validation to select $\gamma$ from an interval centered at the choice of $\gamma$ given by these  approximations with a combination of density and rank of solution used as the cross-validation criterion.

\subsubsection{Robustness to Choice of Regularization Parameter}

We performed the following experiment to test the robustness of the recovery guarantees given by Theorem~\ref{thm:suff-cond-random} and Theorem~\ref{thm:suff-cond-adv} to the choice of the regularization parameter $\gamma$:
\begin{itemize}
    \item For each $\gamma$ in a set of potential values of the regularization parameter, we solve~\eqref{eqn3.2} using Alg.~\ref{alg:densub}. 
    \item For each choice of $\gamma$, we stop Alg.~\ref{alg:densub} after a $10^{-4}$-suboptimal solution is found or $2000$ iterations have been performed.
    \item We then compare the recovered solution $\X$ with the indicator matrix $\X_0$ for the known densest submatrix before and after rounding each entry of $\X_0$ to the nearest integer.      
\end{itemize}

\paragraph{Synthetic Data}
\label{sec:gamma-synth}

We first sample  $10$ matrices from the planted submatrix model as follows:
\begin{itemize} 
\item 
    We set $M,N=500$ and partition $[M]$ and $[N]$ into two blocks with $m_1, n_1=400$ and $m_2, n_2 = 100$.
\item We set $p_{11} = 0.75$ and $p_{12} = p_{21} = p_{22} = 0.25$.
\end{itemize}
The sufficient condition given in Theorem~\ref{thm:suff-cond-random} suggests that we should have perfect recovery of the hidden dense $m_1\times n_1$-submatrix for certain choices of $\gamma$ w.h.p.~in this case.

We apply Alg.~\ref{alg:densub} to solve each instance of~\eqref{eqn3.2} corresponding to each 
$$
    \gamma \in \{0, 0.005, 0.01 , 0.015, \dots, 0.145, 0.15\}. 
$$
This set includes the value used in Sections~\ref{sec:expt-1} and~\ref{sec:expt-2}
$$
    \gamma = \frac{12}{\sqrt{mn}} = 0.03.
$$
The result of this experiment is summarized in Figure~\ref{fig:gamma-synth}. Here, we observe that recovery error is below the stopping tolerance $10^{-4}$ for all $\gamma$ chosen within the interval $[0.02, 0.12]$ with poor recovery error outside this interval. This agrees with the claim that we should have perfect recovery (up to error from early termination) for all $\gamma$ within a certain interval.

\begin{figure}[t]
    \centering
    \subfloat[Recovery Error]{        
        \includegraphics[width=0.31\linewidth]{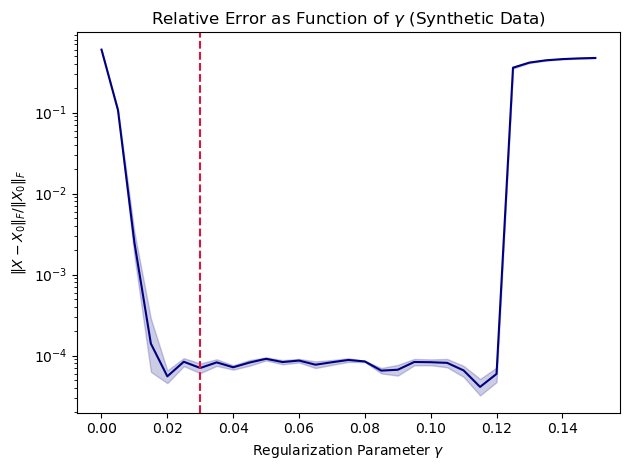}
        \label{fig:gamma-synth-err}
    }
    \subfloat[Recovery Error After Rounding]{
        \includegraphics[width=0.31\linewidth]{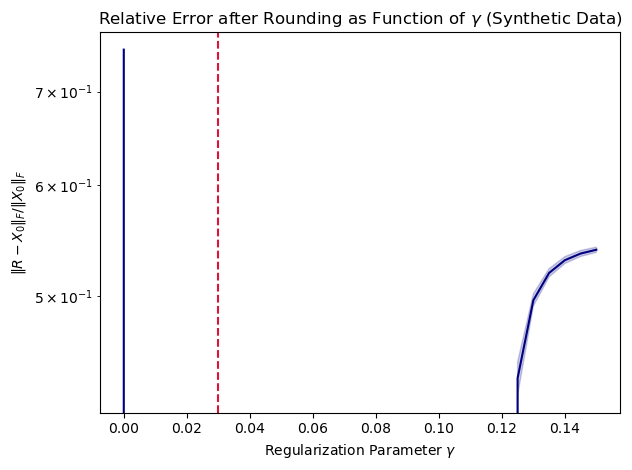}
        \label{fig:gamma-synth-round}
    }
    \subfloat[Run-time (s)]{        
        \includegraphics[width=0.31\linewidth]{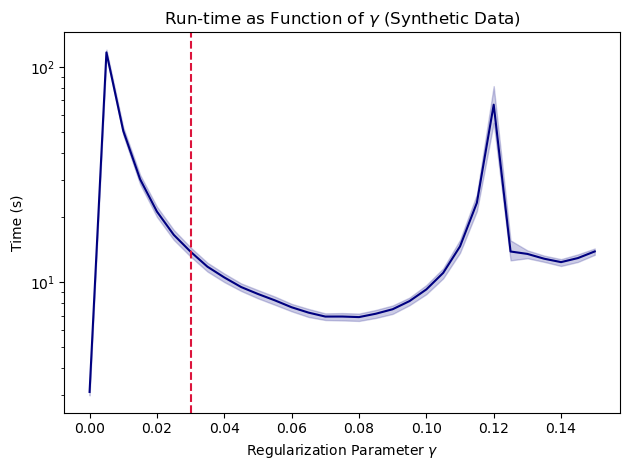}        
        \label{fig:gamma-run-time}
    }    
    \caption{Average recovery error with $95\%$-confidence interval before and after rounding for each choice of $\gamma$, as well as average run-time, for $10$ matrices sampled from the planted submatrix model where we should expect perfect recovery for well-chosen $\gamma$.}
    \label{fig:gamma-synth}
\end{figure}

\paragraph{Real-World Networks}
\label{sec:gamma-real}

We repeated the analysis of Section~\ref{sec:gamma-synth} with two of the networks considered in Sections~\ref{sec:real-data} and~\ref{sec:got-data}:
\begin{itemize}
    \item The JAZZ Network with $m=n=30$ and $\gamma \in \{0,    0.008, 0.016, 0.024, \dots, 0.392,  0.4\}$; and 
    \item Book 1 of the \emph{ASOIAF} series with $m=n=18$ and $\gamma \in \{0,   0.02, 0.04, 0.06, \dots, 0.98, 1 \}$.
\end{itemize}
For each network and choice of $m$ and $\gamma$, we call Alg.~\ref{alg:densub} and compare the calculated solution with the known maximum clique of size $m$.
The recovery error before and after rounding is visualised in Figures~\ref{fig:gamma-jazz} and~\ref{fig:gamma-got}. In both cases, we can observe that we have perfect recovery of the maximum clique for all $\gamma$ above a critical value. For example, we have recovery, before and after rounding, if we use 
\begin{equation}\label{eq:real-gamma-estimate}
\gamma = \frac{6}{m(1 - \bar{p})},
\end{equation}
indicated by the dashed vertical line in Figures~\ref{fig:gamma-jazz} and~\ref{fig:gamma-got},
where we estimate the unknown parameter $p^*$ using the $\bar p = \frac{1}{m^2} \sum_{i,j} a_{ij}$ is the density of the network (with $\A$ denoting the adjacency matrix of the network).

\begin{figure}[t]
    \centering
    \subfloat[Recovery Error]{
        \includegraphics[width=0.31\linewidth]{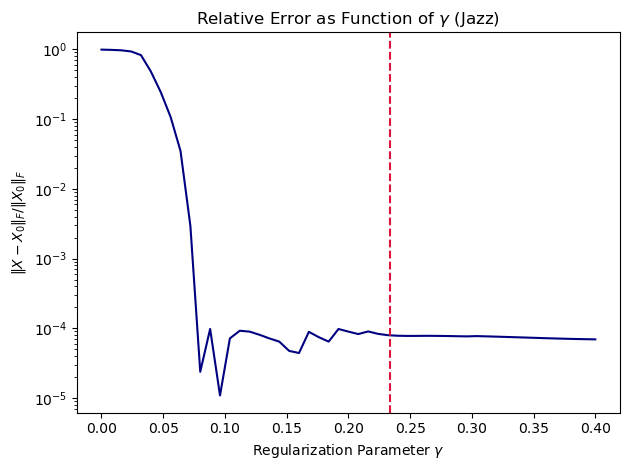}
        \label{fig:gamma-jazz-err}
    }
    \subfloat[Recovery Error After Rounding]{
        \includegraphics[width=0.31\linewidth]{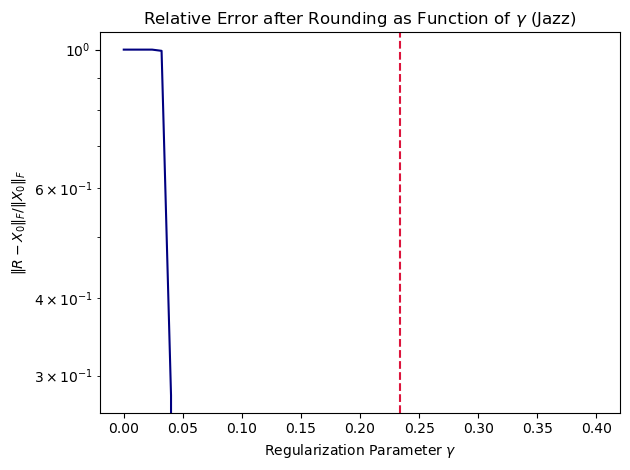}
        \label{fig:gamma-jazz-round}
    }
    \subfloat[Run-time (s)]{
        \includegraphics[width=0.31\linewidth]{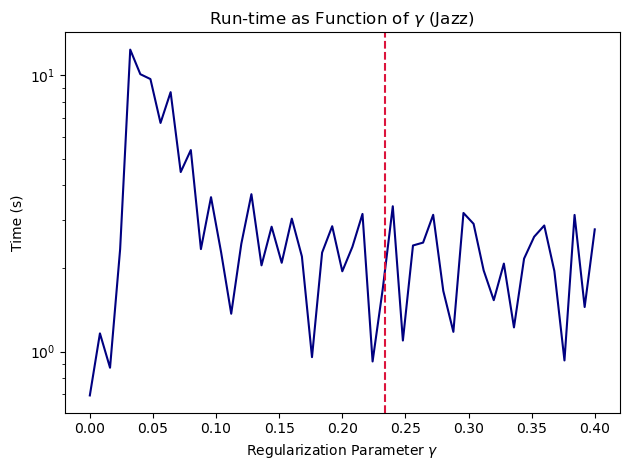}
        \label{fig:gamma-jazztime}
    }
    \caption{Recovery error before and after rounding, and run-time, for each choice of $\gamma$ for the JAZZ collaboration network. The dashed vertical line denotes the choice of $\gamma$ using density as an estimate of $p^*$ according to~\eqref{eq:real-gamma-estimate}.}
    \label{fig:gamma-jazz}
\end{figure}

\begin{figure}[t]
    \subfloat[Recovery Error]{
        \includegraphics[width=0.31\linewidth]{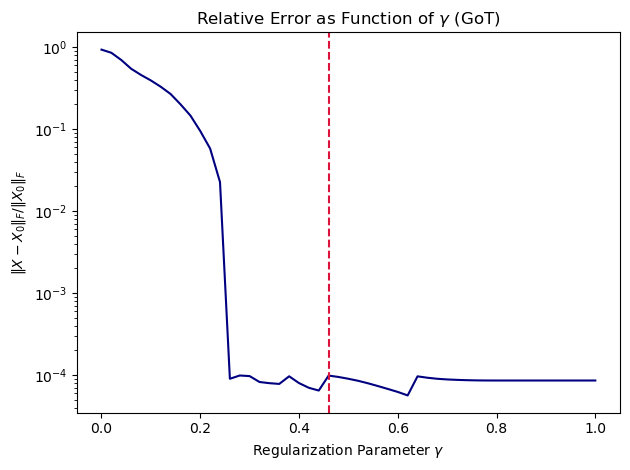}
        \label{fig:gamma-GOT-err}
    }
    \subfloat[Recovery Error After Rounding]{
        \includegraphics[width=0.31\linewidth]{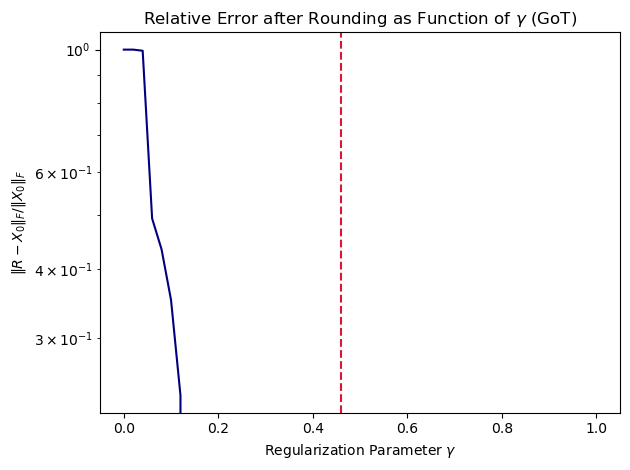}
        \label{fig:gamma-GOT-round}
    }
    \subfloat[Run-time (s)]{
        \includegraphics[width=0.31\linewidth]{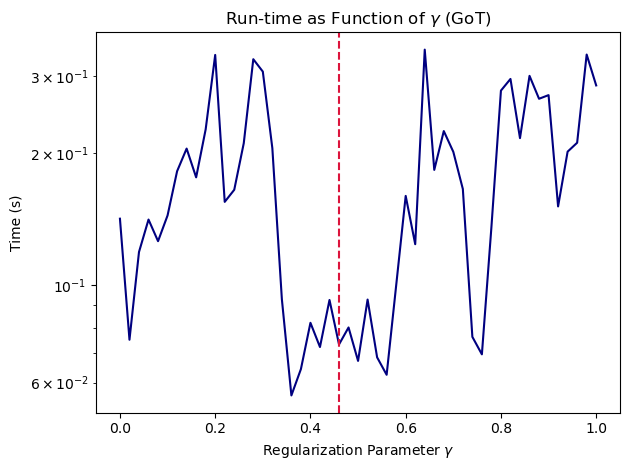}
        \label{fig:gamma-GOTtime}
    }
    \caption{Recovery error before and after rounding, and run-time, for each choice of $\gamma$ for the \emph{ASOIAF} Book 1 interaction network. The dashed vertical line denotes the choice of $\gamma$ using density as an estimate of $p^*$ according to~\eqref{eq:real-gamma-estimate}.}
    \label{fig:gamma-got}
\end{figure}

\paragraph{The Benefits of Rounding}
\label{sec:rounding-gamma}

In each of these experiments, we have perfect recovery after rounding for a wider range of $\gamma$ than without rounding.
When the hypothesis of Theorems~\ref{thm:suff-cond-random} and~\ref{thm:suff-cond-adv} are satisfied the solution at termination of Algorithm~\ref{alg:densub} is close to the rank-one matrix representation of the densest $m\times n$-submatrix; in this case, rounding will give exactly the rank-one indicator matrix. 

When $\gamma$ is chosen outside the range of values given in~\eqref{eq:gamma-range}, the rank-one matrix representation may not be optimal for~\eqref{eqn3.2}. In this case, we still may be able to identify the densest $m\times n$-submatrix using rounding. Figures~\ref{fig:synth-round-gamma},~\ref{fig:gamma-jazz-round}, and~\ref{fig:gamma-GOT-round} illustrate this phenomena. In each case, we find an optimal solution $\X^*$ of the convex relaxation~\eqref{eqn3.2} with strictly fractional entries. Rounding the entries of $\X^*$ to the nearest integer yields the matrix representation of the planted $m\times n$-submatrix in each of these cases; this rank-one binary matrix is suboptimal for~\eqref{eqn3.2} with this choice of $\gamma$ but does reveal the solution for the original combinatorial problem.
However, if $\gamma$ is chosen too far outside this interval, we obtain optimal solutions for the relaxation~\eqref{eqn3.2} which are far from the optimal solution of the original problem, as illustrated by Figures~\ref{fig:Synth-bad-gamma},~\ref{fig:jazz-bad-gamma}, and~\ref{fig:GOT-bad-gamma}. In this case, we cannot extract the densest $m\times n$-submatrix from these optimal solutions of~\eqref{eqn3.2} by rounding.

\begin{figure}[t]
    \centering
    
    \subfloat[$\gamma = 0.005$]{
        \includegraphics[width=0.31\linewidth]{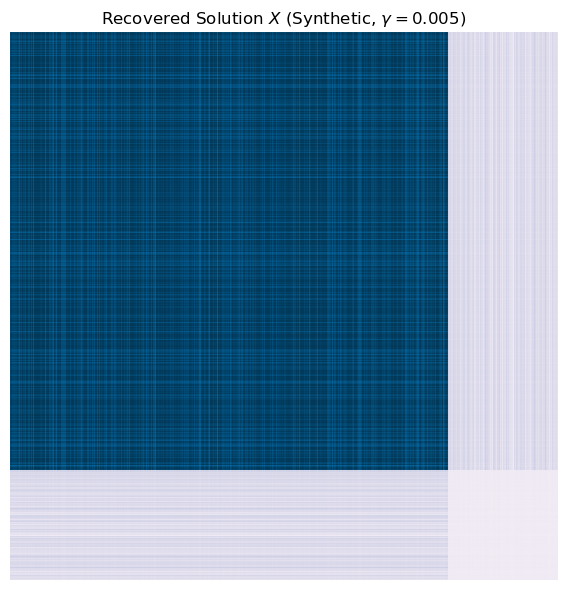}        
        \label{fig:synth-round-gamma}
    }  
    \subfloat[$\gamma=0.03$]{
        \includegraphics[width=0.31\linewidth]{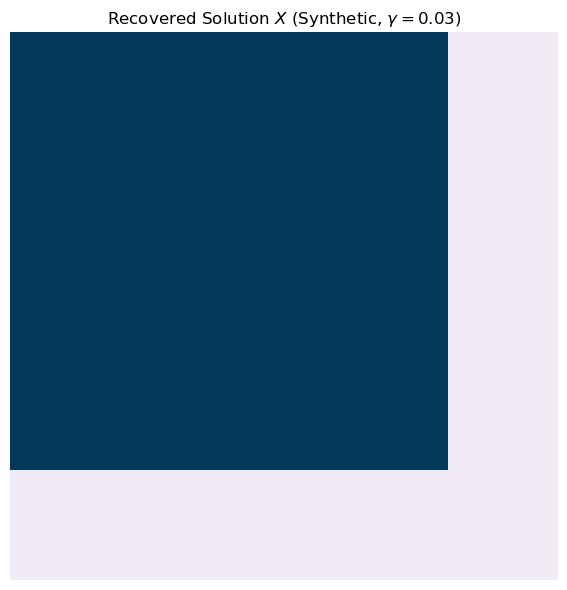}
        \label{fig:Synth-goood}
    }
    %
    \subfloat[$\gamma = 0.15$]{
        \includegraphics[width=0.31\linewidth]{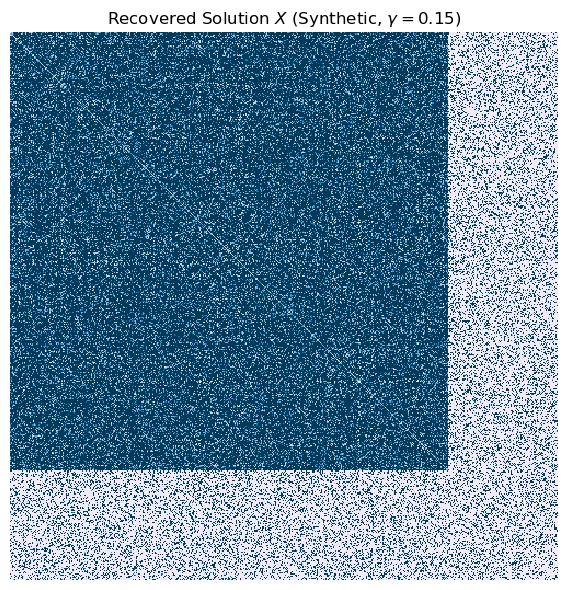}
        \label{fig:Synth-bad-gamma}
    }
    \caption{Recovered solutions for a random matrix sampled from the planted submatrix model from Sect.~\ref{sec:gamma-synth} for different choices of $\gamma$.}
    \label{fig:Synth-rounding}
\end{figure}

\begin{figure}[t]
    \centering
    \subfloat[$\gamma = 0.008$]{
        \includegraphics[width=0.31\linewidth]{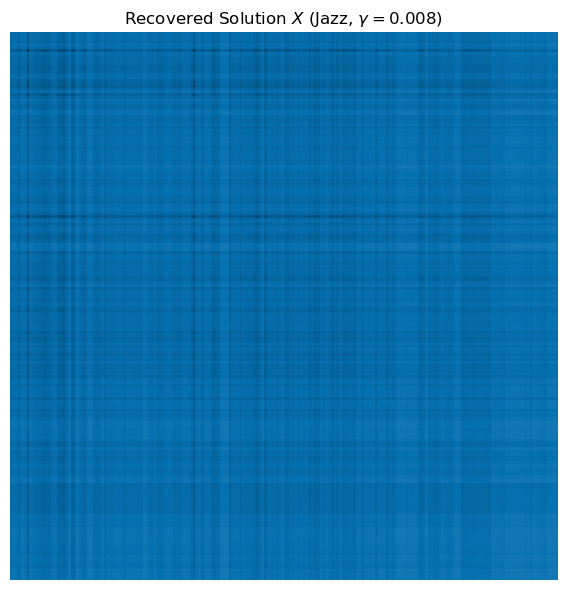}
        \label{fig:jazz-bad-gamma}
        }    
    \subfloat[$\gamma = 0.048$]{
        \includegraphics[width=0.31\linewidth]{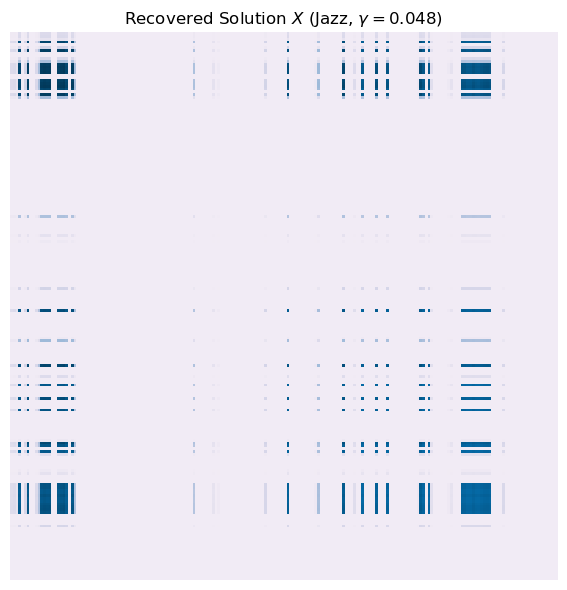}        
        \label{fig:gamma-jazz-round}
        }    
    \subfloat[$\gamma=0.4$]{
        \includegraphics[width=0.31\linewidth]{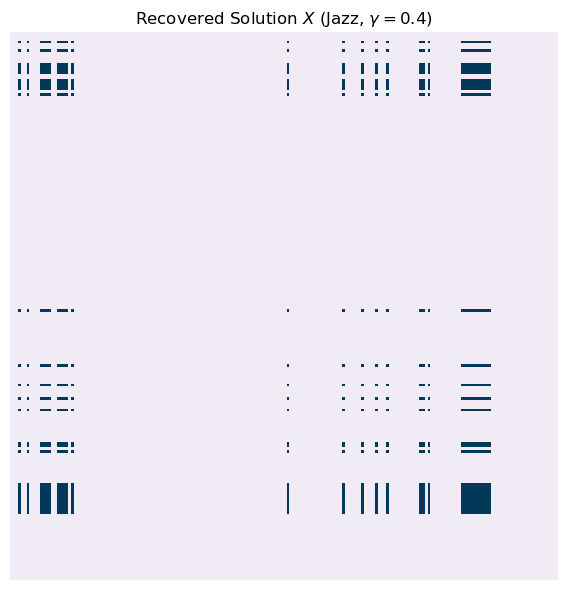}
        \label{fig:gamma-jazz-goood}
    }
    \caption{Recovered solutions for the JAZZ Network for different choices of $\gamma$.}
    \label{fig:Jazz-rounding}
\end{figure}

\begin{figure}[t]
    \centering
    \subfloat[$\gamma = 0.008$]{
        \includegraphics[width=0.31\linewidth]{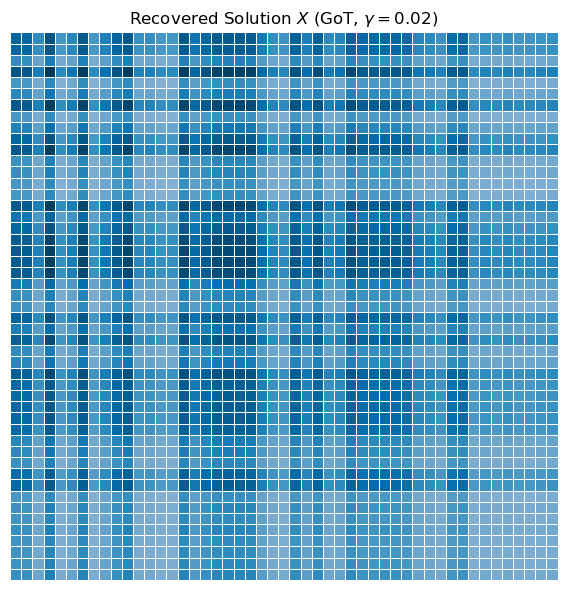}
        \label{fig:GOT-bad-gamma}
        }    
    \subfloat[$\gamma = 0.048$]{
        \includegraphics[width=0.31\linewidth]{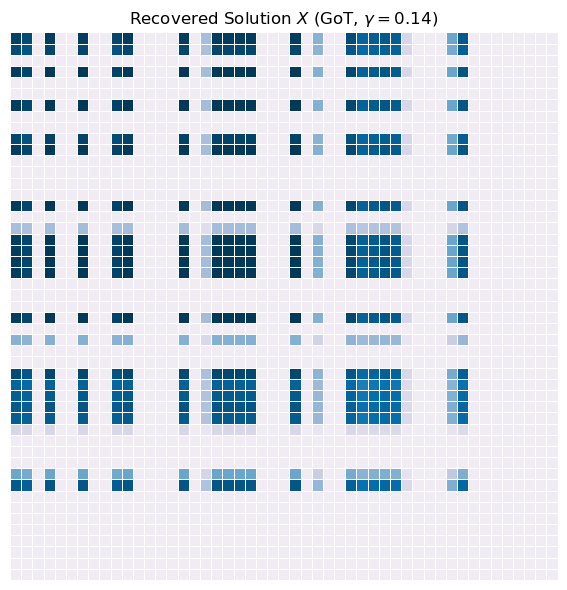}        
        \label{fig:gamma-GOT-round}
        }    
    \subfloat[$\gamma=0.4$]{
        \includegraphics[width=0.31\linewidth]{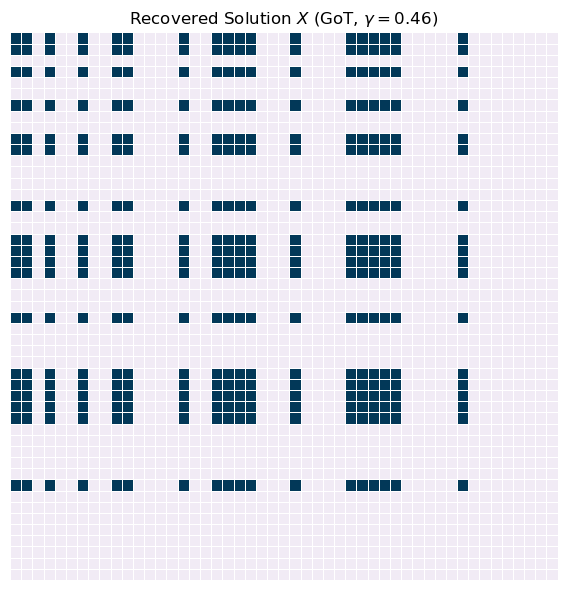}
        \label{fig:gamma-GOT-good}
    }
    \caption{Recovered solutions for the \textit{ASOIAF} Book 1 Network for different choices of $\gamma$.}
    \label{fig:GoT-rounding}
\end{figure}

\section{Conclusions}

We have extended existing convex relaxation approaches for the densest submatrix and planted clique problems to a significantly more general and realistic setting involving multiple dense substructures. By analyzing both probabilistic and adversarial models, we established new sufficient conditions guaranteeing perfect recovery of a planted submatrix in polynomial time. These conditions reveal the precise signal-to-noise trade-offs that govern the phase transition from partial to exact recovery. Further, our results generalize and strengthen several existing recovery guarantees for planted cliques, bicliques, and dense subgraphs as special cases.

Our results demonstrate that nuclear-norm-based convex formulations can effectively identify planted dense blocks even in heterogeneous stochastic block models and in the presence of structured adversarial noise. The accompanying dual-certificate analysis provides rigorous theoretical justification for these recovery guarantees. Empirical experiments conducted on both synthetic and real-world networks substantiate the theoretical predictions, demonstrating the sharpness of the proposed phase transitions and the robustness of the convex relaxation under heterogeneous noise.

The results of our empirical analysis suggest several areas for future research. Of particular interest is the need to strengthen our sufficient conditions for perfect recovery for the case where we have several planted dense blocks of the same size and similar density, e.g., when graphs contain several maximum cliques. We have observed that our nuclear-norm-based approach often recovers a convex combination of the multiple optimal solutions, but we have yet to identify a theoretical phase transition ensuring perfect recovery in this case; this would answer the question posed by Conjecture~\ref{conj:multiple-bicliques-conjecture}.

Moreover, it is necessary to develop scalable algorithmic implementations tailored for large-scale network data. The ADMM approach considered in Section~\ref{sec:ADMM} is considerably more efficient than classical interior point methods for solving our convex relaxation, but still rely on a computationally expensive singular value decomposition, which limits the scale of problems that we can solve in practice. Scalable first-order methods that avoid this SVD step or efficient interior-point algorithms exploiting the sparse structure of the constraints are needed to obtain scalable algorithms for this problem.

\subsection*{Acknowledgements}

The authors acknowledge the use of the IRIDIS High Performance Computing Facility and associated support services at the University of Southampton in the completion of this work.
B.~Ames was supported in part by the National Science Foundation Grants \#20212554 and \#2108645, as well as the University of Alabama Research Grants RG14678 and RG14838.

\bibliographystyle{abbrv}
\bibliography{refs}

@inproceedings{bombina2020convex,
  title={Convex optimization for the densest subgraph and densest submatrix problems},
  author={Bombina, Polina and Ames, Brendan},
  booktitle={SN Operations Research Forum},
  volume={1},
  number={3},
  pages={1--24},
  year={2020},
  organization={Springer}
}

@article{ames2015guaranteed,
  title={Guaranteed recovery of planted cliques and dense subgraphs by convex relaxation},
  author={Ames, Brendan PW},
  journal={Journal of Optimization Theory and Applications},
  volume={167},
  number={2},
  pages={653--675},
  year={2015},
  publisher={Springer}
}

@book{van1996matrix,
  title={Matrix Computations},
  author={Golub, Gene H and Van Loan, Charles F},
  year={2013},
  publisher={JHU Press}
}

@book{boucheron2013concentration,
  title={Concentration inequalities: A nonasymptotic theory of independence},
  author={Boucheron, St{\'e}phane and Lugosi, G{\'a}bor and Massart, Pascal},
  year={2013},
  publisher={Oxford university press}
}

@article{van2017dimension,
  title={The dimension-free structure of nonhomogeneous random matrices},
  author={Lata{\l}a, Rafa{\l} and van Handel, Ramon and Youssef, Pierre},
  journal={Inventiones mathematicae},
  volume={214},
  number={3},
  pages={1031--1080},
  year={2018},
  publisher={Springer}
}

@article{tropp2012user,
  title={User-friendly tail bounds for sums of random matrices},
  author={Tropp, Joel A},
  journal={Foundations of computational mathematics},
  volume={12},
  pages={389--434},
  year={2012},
  publisher={Springer}
}

@article{deng2011generalization,
  title={A generalization of the {Sherman-Morrison-Woodbury} formula},
  author={Deng, Chun Yuan},
  journal={Applied Mathematics Letters},
  volume={24},
  number={9},
  pages={1561--1564},
  year={2011},
  publisher={Elsevier}
}

@incollection{karp2009reducibility,
  title={Reducibility among combinatorial problems},
  author={Karp, Richard M},
  booktitle={50 Years of Integer Programming 1958-2008: from the Early Years to the State-of-the-Art},
  pages={219--241},
  year={2009},
  publisher={Springer}
}

@article{alon2011inapproximability,
  title={Inapproximability of densest $\kappa$-subgraph from average case hardness},
  author={Alon, Noga and Arora, Sanjeev and Manokaran, Rajsekar and Moshkovitz, Dana and Weinstein, Omri},
  journal={Unpublished manuscript},
  volume={1},
  pages={6},
  year={2011}
}

@inproceedings{feige2002relations,
  title={Relations between average case complexity and approximation complexity},
  author={Feige, Uriel},
  booktitle={Proceedings of the thiry-fourth annual ACM symposium on Theory of computing},
  pages={534--543},
  year={2002}
}

@article{feige2000finding,
  title={Finding and certifying a large hidden clique in a semirandom graph},
  author={Feige, Uriel and Krauthgamer, Robert},
  journal={Random Structures \& Algorithms},
  volume={16},
  number={2},
  pages={195--208},
  year={2000},
  publisher={Wiley Online Library}
}

@article{alon1998finding,
  title={Finding a large hidden clique in a random graph},
  author={Alon, Noga and Krivelevich, Michael and Sudakov, Benny},
  journal={Random Structures \& Algorithms},
  volume={13},
  number={3-4},
  pages={457--466},
  year={1998},
  publisher={Wiley Online Library}
}

@article{khot2006ruling,
  title={Ruling out {PTA}S for graph min-bisection, dense $k$-subgraph, and bipartite clique},
  author={Khot, Subhash},
  journal={SIAM Journal on Computing},
  volume={36},
  number={4},
  pages={1025--1071},
  year={2006},
  publisher={SIAM}
}

@article{andersen2007finding,
  title={Finding large and small dense subgraphs},
  author={Andersen, Reid},
  journal={arXiv preprint cs/0702032},
  year={2007}
}

@inproceedings{malod2010maximum,
  title={Maximum cliques in protein structure comparison},
  author={Malod-Dognin, No{\"e}l and Andonov, Rumen and Yanev, Nicola},
  booktitle={Experimental Algorithms: 9th International Symposium, SEA 2010, Ischia Island, Naples, Italy, May 20-22, 2010. Proceedings 9},
  pages={106--117},
  year={2010},
  organization={Springer}
}

@article{gomez2008identification,
  title={Identification of a 5-protein biomarker molecular signature for predicting {Alzheimer}'s disease},
  author={G{\'o}mez Ravetti, Mart{\'\i}n and Moscato, Pablo},
  journal={PloS {O}ne},
  volume={3},
  number={9},
  pages={e3111},
  year={2008},
  publisher={Public Library of Science San Francisco, USA}
}

@article{boginski2006mining,
  title={Mining market data: A network approach},
  author={Boginski, Vladimir and Butenko, Sergiy and Pardalos, Panos M},
  journal={Computers \& Operations Research},
  volume={33},
  number={11},
  pages={3171--3184},
  year={2006},
  publisher={Elsevier}
}

@article{balasundaram2011clique,
  title={Clique relaxations in social network analysis: The maximum k-plex problem},
  author={Balasundaram, Balabhaskar and Butenko, Sergiy and Hicks, Illya V},
  journal={Operations Research},
  volume={59},
  number={1},
  pages={133--142},
  year={2011},
  publisher={INFORMS}
}

@incollection{pattillo2011clique,
  title={Clique relaxation models in social network analysis},
  author={Pattillo, Jeffrey and Youssef, Nataly and Butenko, Sergiy},
  booktitle={Handbook of Optimization in Complex Networks: Communication and Social Networks},
  pages={143--162},
  year={2011},
  publisher={Springer}
}

@inproceedings{jain2003impact,
  title={Impact of interference on multi-hop wireless network performance},
  author={Jain, Kamal and Padhye, Jitendra and Padmanabhan, Venkata N and Qiu, Lili},
  booktitle={Proceedings of the 9th annual international conference on Mobile computing and networking},
  pages={66--80},
  year={2003}
}

@article{schaeffer2007graph,
  title={Graph clustering},
  author={Schaeffer, Satu Elisa},
  journal={Computer Science Review},
  volume={1},
  number={1},
  pages={27--64},
  year={2007},
  publisher={Elsevier}
}

@article{boyd2004convex,
  title={Convex pptimization},
  author={Boyd, Stephen},
  journal={Cambridge UP},
  year={2004}
}

@article{ames2011nuclear,
  title={Nuclear norm minimization for the planted clique and biclique problems},
  author={Ames, Brendan PW and Vavasis, Stephen A},
  journal={Mathematical programming},
  volume={129},
  number={1},
  pages={69--89},
  year={2011},
  publisher={Springer}
}

@article{bandeira2016sharp,
  title={Sharp nonasymptotic bounds on the norm of random matrices with independent entries},
  author={Bandeira, Afonso S and Van Handel, Ramon},
  year={2016}
}

@inproceedings{chen2014statistical,
  title={Statistical-computational phase transitions in planted models: The high-dimensional setting},
  author={Chen, Yudong and Xu, Jiaming},
  booktitle={International conference on machine learning},
  pages={244--252},
  year={2014},
  organization={PMLR}
}

@article{chandrasekaran2011rank,
  title={Rank-sparsity incoherence for matrix decomposition},
  author={Chandrasekaran, Venkat and Sanghavi, Sujay and Parrilo, Pablo A and Willsky, Alan S},
  journal={SIAM Journal on Optimization},
  volume={21},
  number={2},
  pages={572--596},
  year={2011},
  publisher={SIAM}
}

@article{candes2011robust,
  title={Robust principal component analysis?},
  author={Cand{\`e}s, Emmanuel J and Li, Xiaodong and Ma, Yi and Wright, John},
  journal={Journal of the ACM (JACM)},
  volume={58},
  number={3},
  pages={1--37},
  year={2011},
  publisher={ACM New York, NY, USA}
}

@article{doan2013finding,
  title={Finding Approximately Rank-One Submatrices with the Nuclear Norm and $\ell_1$-Norm},
  author={Doan, Xuan Vinh and Vavasis, Stephen},
  journal={SIAM Journal on Optimization},
  volume={23},
  number={4},
  pages={2502--2540},
  year={2013},
  publisher={SIAM}
}

@article{recht2010guaranteed,
  title={Guaranteed minimum-rank solutions of linear matrix equations via nuclear norm minimization},
  author={Recht, Benjamin and Fazel, Maryam and Parrilo, Pablo A},
  journal={SIAM review},
  volume={52},
  number={3},
  pages={471--501},
  year={2010},
  publisher={SIAM}
}

@article{ames2014guaranteed,
  title={Guaranteed clustering and biclustering via semidefinite programming},
  author={Ames, Brendan PW},
  journal={Mathematical Programming},
  volume={147},
  number={1},
  pages={429--465},
  year={2014},
  publisher={Springer}
}

@article{sudoso2025semidefinite,
  title={A semidefinite programming-based branch-and-cut algorithm for biclustering},
  author={Sudoso, Antonio M},
  journal={INFORMS Journal on Computing},
  volume={37},
  number={6},
  pages={1433--1456},
  year={2025},
  publisher={INFORMS}
}

@article{boyd2011distributed,
  title={Distributed optimization and statistical learning via the alternating direction method of multipliers},
  author={Boyd, Stephen and Parikh, Neal and Chu, Eric and Peleato, Borja and Eckstein, Jonathan and others},
  journal={Foundations and Trends{\textregistered} in Machine learning},
  volume={3},
  number={1},
  pages={1--122},
  year={2011},
  publisher={Now Publishers, Inc.}
}

@article{hong2017linear,
  title={On the linear convergence of the alternating direction method of multipliers},
  author={Hong, Mingyi and Luo, Zhi-Quan},
  journal={Mathematical Programming},
  volume={162},
  number={1},
  pages={165--199},
  year={2017},
  publisher={Springer}
}

@misc{abbe2023communitydetectionstochasticblock,
      title={Community Detection and Stochastic Block Models}, 
      author={Emmanuel Abbe},
      year={2023},
      eprint={1703.10146},
      archivePrefix={arXiv},
      primaryClass={math.PR},
      url={https://arxiv.org/abs/1703.10146}, 
}

@article{zachary1977information,
  title={An information flow model for conflict and fission in small groups},
  author={Zachary, Wayne W},
  journal={Journal of anthropological research},
  volume={33},
  number={4},
  pages={452--473},
  year={1977},
  publisher={University of New Mexico}
}

@book{knuth1993stanford,
  title={The Stanford GraphBase: a platform for combinatorial computing},
  author={Knuth, Donald Ervin},
  volume={1},
  year={1993},
  publisher={AcM Press New York}
}

@article{lusseau2003bottlenose,
  title={The bottlenose dolphin community of doubtful sound features a large proportion of long-lasting associations: can geographic isolation explain this unique trait?},
  author={Lusseau, David and Schneider, Karsten and Boisseau, Oliver J and Haase, Patti and Slooten, Elisabeth and Dawson, Steve M},
  journal={Behavioral ecology and sociobiology},
  volume={54},
  number={4},
  pages={396--405},
  year={2003},
  publisher={Springer}
}

@article{gleiser2003community,
  title={Community structure in jazz},
  author={Gleiser, Pablo M and Danon, Leon},
  journal={Advances in complex systems},
  volume={6},
  number={04},
  pages={565--573},
  year={2003},
  publisher={World Scientific}
}

@article{bron1973algorithm,
  title={Algorithm 457: finding all cliques of an undirected graph},
  author={Bron, Coen and Kerbosch, Joep},
  journal={Communications of the ACM},
  volume={16},
  number={9},
  pages={575--577},
  year={1973},
  publisher={ACM New York, NY, USA}
}

@article{tomita2006worst,
  title={The worst-case time complexity for generating all maximal cliques and computational experiments},
  author={Tomita, Etsuji and Tanaka, Akira and Takahashi, Haruhisa},
  journal={Theoretical computer science},
  volume={363},
  number={1},
  pages={28--42},
  year={2006},
  publisher={Elsevier}
}

@article{cazals2008note,
  title={A note on the problem of reporting maximal cliques},
  author={Cazals, Fr{\'e}d{\'e}ric and Karande, Chinmay},
  journal={Theoretical computer science},
  volume={407},
  number={1-3},
  pages={564--568},
  year={2008},
  publisher={Elsevier}
}

@article{beveridge2016network,
  title={Network of {Thrones}},
  author={Beveridge, Andrew and Shan, Jie},
  journal={Math Horizons},
  volume={23},
  number={4},
  pages={18--22},
  year={2016},
  publisher={Taylor \& Francis},
    url={https://github.com/mathbeveridge/mathbeveridge.github.io/blob/master/files/NetworkofThrones.pdf}
}

@misc{networkofthronesNetworkThrones,
	author = {},
	title = {{N}etwork of {T}hrones --- networkofthrones.com},
	howpublished = {\url{https://networkofthrones.com/}},
	year = {},
	note = {[Accessed 24-10-2025]},
}

@incollection{beveridge2018game,
  title={The game of {Game of Thrones}: Networked concordances and fractal dramaturgy},
  author={Beveridge, Andrew and Chemers, Michael},
  booktitle={Reading Contemporary Serial Television Universes},
  pages={201--225},
  year={2018},
  publisher={Routledge}
}

@inproceedings{sohn2025sharp,
  title={Sharp phase transitions in estimation with low-degree polynomials},
  author={Sohn, Youngtak and Wein, Alexander S},
  booktitle={Proceedings of the 57th Annual ACM Symposium on Theory of Computing},
  pages={891--902},
  year={2025}
}

@article{dadon2024detection,
  title={Detection and recovery of hidden submatrices},
  author={Dadon, Marom and Huleihel, Wasim and Bendory, Tamir},
  journal={IEEE Transactions on Signal and Information Processing over Networks},
  volume={10},
  pages={69--82},
  year={2024},
  publisher={IEEE}
}

@article{omer2023harnessing,
  title={Harnessing the mathematics of matrix decomposition to solve planted and maximum clique problem},
  author={Omer, Salma and Ali, Montaz},
  journal={arXiv preprint arXiv:2307.09022},
  year={2023}
}

@article{omer2025maximum,
  title={Maximum edge bi-clique via matrix decomposition},
  author={Omer, Salma and Ali, Montaz},
  journal={Journal of Industrial and Management Optimization},
  volume={21},
  number={11},
  pages={6270--6294},
  year={2025},
  publisher={Journal of Industrial and Management Optimization}
}

@article{lanciano2024survey,
  title={A survey on the densest subgraph problem and its variants},
  author={Lanciano, Tommaso and Miyauchi, Atsushi and Fazzone, Adriano and Bonchi, Francesco},
  journal={ACM Computing Surveys},
  volume={56},
  number={8},
  pages={1--40},
  year={2024},
  publisher={ACM New York, NY}
}

\begin{appendix}

\section{Proof of Lemma~\ref{lem4.4}}
\label{app:MBI}

Lemma~\ref{lem4.4} is a special case of the Matrix Bernstein Inequality applied to $\boldsymbol{Z} =    {\T} - \tilde{\T}$ as found in the hypothesis of Lemma~\ref{lem4.4}.
The Matrix Bernstein Inequality, as presented in \cite{tropp2012user}, is provided below:

\renewcommand{\S}{\boldsymbol{S}}
\begin{theorem}
    Let $\{\S_k\}$ be a finite sequence of independent real $d_{1} \times d_{2}$ random matrices such that 
    $\textbf{E}[\S_{k}]=\boldsymbol{0}$ and $\lVert \S_{k} \rVert \le L$ for all $k$. 
    Introduce the sum $\Z := \sum_{k} \S_{k}$ and let $\nu (\Z)$ denote the variance of the sum:\\
    \begin{align} \label{eq: nu Z def}
        \nu (\Z) = \max\{ \lVert \mathbf{E}[\Z\Z^T] \rVert, \lVert \mathbf{E}[\Z^T\Z] \rVert\} 
        &= \max \left\{ \left\lVert \sum_{k} \mathbf{E}[\S_{k}\S_{k}^T] \right\rVert, 
                        \left\lVert \sum_{k} \mathbf{E}[\S_{k}^{T}\S_{k}] \right\rVert
                    \right\}.
    \end{align}
    Then, for all $t \ge 0$,
    \begin{align} \label{brninq}
     \prob(\lVert \Z \rVert \ge t) 
        &\le  (d_{1} + d_{2}) \exp \left(\frac{-t^{2}/2}{\nu (\Z) + Lt/3}\right). 
    \end{align}
\end{theorem}

\renewcommand{\d}{\boldsymbol{d}}
\renewcommand{\a}{\boldsymbol{a}}
\newcommand{\at}{\boldsymbol{\tilde{a}}}
Consider the matrix
\begin{align}
    \boldsymbol{S}_{j} := \boldsymbol{d}_j\boldsymbol{e}_j^T \label{rf},
\end{align}
where $\d_j := \a_j - \at_j$ is the difference of the $j$th columns $\a_j$, $\at_j$ of $\T$ and $\tilde\T$, respectively, and $\e_j$ denotes the $j$th column of the identity matrix.
Clearly $\Z = \T -\tilde\T = \sum_j \S_j$.
We next calculate $L$ and $\nu(\boldsymbol{Z})$ to substitute in~\eqref{brninq}. 
We start with a bound on $L$.

We start with the following lemma, which bounds $\|\S_j\|$ with high probability.

\begin{lemma} \label{lem: L bound}
    There exist $c_{1} \ge 0$ such that matrices $\boldsymbol{S}_{j}$ defined by \eqref{rf} satisfy
    \begin{align} \label{eq: L bound}
    \|\boldsymbol{S}_j \| 
            \le L := 
            c_1 
            \sqrt{
            \max \{ \tilde \sigma_*^2, 1 \} \log N 
            }
    \end{align}
    for all $j=1,2,\dots N$ with high probability,
    where 
    \begin{equation} \label{eq: sigma star def}
        \tilde \sigma_*^2 := \max_{s=1,2,\dots,k} \tilde \sigma_{s}^2 = \max_{s=1,2,\dots,k} \frac{p_{s}}{1 - p_{s}}.
    \end{equation}    
\end{lemma}
    
\begin{proof}
    Fix $j \in V_s$ for some $s=1,2,\dots, k$. 
    Then, 
    $$\lVert \boldsymbol{S}_{j}\rVert = \lVert \d_{j} \boldsymbol e^{T}_{j}\rVert
    = \lVert \d_{j} \rVert \lVert \boldsymbol e^{T}_{j}\rVert = \lVert \d_{j}\rVert.$$
    We have 
    \begin{align*}
        \lVert \d_{j} \lVert^{2} &\le(n - n_j) \left( -\frac{n_j}{n - n_j} + \frac{p_s}{1-p_s} \right)^{2} 
        = \frac{ (n_j - np_s)^{2}}{ (1 - p_s)^{2} (n - n_j)}.
    \end{align*}
    Note further that
    \begin{align*}
        \frac{d}{dn_j} \|\d_j\|^2 &= \frac{(n p_s - n_j)(n_j - n(2-p_s))}{(n - n_j)^2(1-p_s)^2}, \\
        \frac{d^2}{d^2n_j} \|\d_j\|^2 &= \frac{2n^2}{(n - n_j)^3}. 
    \end{align*}
    The second derivative is strictly positive for all $n_j < n$, which implies that $\|\d_j\|^2$ is strictly convex with respect to $n_j$ on the interval $[0, n)$. Moreover, $\|\d_j\|^2$ is strictly decreasing on $[0, p_sn)$
    and strictly increasing on $(p_sn, n)$. Thus, $\|\d_j\|^2$ is minimized at the stationary point $n_j = p_s n$ and is maximized at one of the two endpoints $a$ and $b$ when restricted to the interval $[a, b]$ for any $0 < a < p_s n < b < n$.
    
    Recall that applying the standard Bernstein inequality \eqref{eq4.5} shows that 
    $$
        p_s n - 6 \max \left\{\sqrt{\sigma_s^2 n \log N }, \log N \right\} < n_j 
        <
        p_s n -  6 \max \left\{\sqrt{\sigma_s^2 n \log N }, \log N \right\}
    $$
    with high probability.
    Substituting these endpoints as $n_j$ in the formula for $\|\d_j\|^2$ shows that
    if $n - n_j \ge \bar c > 0$ then
    \renewcommand{\O}{\mathcal{O}}
    \begin{align*}            
        \lVert \d_{j} \rVert^{2} 
        &\le 
        \frac{ 36 \max\{n \sigma_{p_s}^{2}\log{N}, (\log{N})^{2}\}}
        {(1 - p_s)^{2} (n - n_j ) }\\
        &\le 
        \frac{36 \bar c \max\{n\sigma_{p_s}^{2} \log{N}, (\log{N})^{2}\}}
        {(1 - p_s)^{2} n} \\ 
        &= \O \left(
            \max \{\tilde \sigma_s^2, 1 \} \log N
            \right)
            = 
            \O \left(
            \max \{\tilde \sigma_{*}^2, 1 \} \log N
            \right)
    \end{align*} 
    w.h.p.,
    where the final inequality follows from the assumption $\log N < n (1 - p_s)^2$, which implies that 
    \begin{equation*}
        \frac{(\log N)^2 }{n (1 - p_s)^2 } \le \log N.
    \end{equation*}
    The existence of such a $\bar c > 0$ is guaranteed w.h.p.~by the Bernstein inequality.
    Applying the union bound over all $j$ shows that 
    \begin{equation*}
        \|\boldsymbol{S}_j \| 
        \le L := 
        \O \left( 
            \sqrt{
            \max \{ \tilde \sigma_*^2, 1 \} \log N 
            }
        \right)                        
    \end{equation*}
    with high probability. 
\end{proof}

       
The next lemma provides a bound on $\nu(\boldsymbol{Z})$.

\begin{lemma}\label{lem: nu Z}
    The matrix $\boldsymbol{Z} = \sum_{j}\boldsymbol{S}_{j}$ defined by \eqref{rf} satisfies 
    $$ 
        \nu(\boldsymbol{Z}) \le \bar c \sum_{s=1}^k \tilde \sigma_s^2 n_{s} \le \bar c~\tilde\sigma_*^2N
    $$ 
    with high probability
    for any constant $\bar c>0$ satisfying $n - n_j  \ge \frac{n}{\bar{c}}$ for all $j$.        
\end{lemma}

\renewcommand{\Z}{\boldsymbol{Z}}
\renewcommand{\e}{\boldsymbol{e}}

\renewcommand{\O}{\mathcal{O}}
    
\begin{proof}                                                                                                              
    Constructing upper bounds for both $\lVert \mathbf{E}[\Z\Z^{T}] \rVert$ and $\lVert \mathbf{E}[\Z^{T}\Z] \rVert$ suffices.  Recall that $\boldsymbol{Z} = \sum_{j}\boldsymbol{S}_{j} = \sum_j \d_j \e_j^T$.
    By our choice of $\boldsymbol{Z}$, we have 
    $$
        \boldsymbol{Z}^{T}\boldsymbol{Z} = \sum_{i=1}^N (\d_{i} \e_{i}^{T})^{T} \sum_{j=1}^N(\d_{j} \e_{j}^{T})  
        = \sum_{i=1}^N \sum_{j=1}^N \left(\d_{i}^{T} \d_{j} \right) \e_i \e_{j}^{T}.            
    $$            
    The independence of $\d_i$, $\d_j$ and the fact that $\textbf{E}[\d_j] = \boldsymbol 0$ for all $i,j$ implies that that $\mathbf{E}[\Z^T \Z]$ is a diagonal matrix with $j$th diagonal entry equal to $\textbf{E}\|\d_j\|^2$. 
    Consequently, $$\lVert \EX[\boldsymbol{Z}^{T} \boldsymbol{Z}] \rVert = \max_{j} \mathbf{E}\lVert \d_{j} \rVert^{2}.$$
    For each $j\in V_s$, we have
    \begin{align*}
        \mathbf{E}\lVert \d_{j} \rVert^{2} &= \mathbf{E}\left[ \frac{ (n_j - n p_s)^{2}}{(1 -  p_s)^{2} (n - n_j)}\right].
    \end{align*}                
    Recall that  $n - n_j \ge n/\bar{c}$ for some constant $\bar{c} > 0$ with high probability by the Bernstein bound.
    This implies that
    \begin{align}
        \mathbf{E}\lVert \d_{j} \rVert^{2} 
        & \le \bar{c}~\mathbf{E}  \left[\frac{(n_j - n p_s)^{2}}{(1 -  p_s)^{2} n } \right]
        = \frac{\bar{c} \Var(n_j)}{(1 -  p_s)^{2} n } \nonumber
        \\&= \frac{ \bar{c}~ p_s}{1 -  p_s} = \bar c~\tilde\sigma_s\label{eq:E ZTZ}
    \end{align}      
    with high probability, since $n_j$ is binomially distributed with variance $p_s(1-p_s) n$.
    
    On the other hand, the triangle inequality and Jensen's inequality imply that
    \begin{align}
        \lVert \mathbf{E}[\Z\Z^T] \rVert &\le \sum_{j \notin V_1} \| \EX [ \d_j \d_j^T] \| \nonumber\\ 
        &= \sum_{s=1}^k \sum_{j\in V_s} \| \EX [ \d_j \d_j^T] \| \nonumber\\ 
        &\le \sum_{s=1}^k \sum_{j\in V_s} \EX \|   \d_j  \|^2 \nonumber \\ 
        &\le \bar c \sum_{s=1}^k \left(\frac{p_s}{1 -  p_s} \right) n_{s} 
        = \bar c  \sum_{s=1}^k \tilde\sigma_s^2 n_s \label{eq: E ZZT}                     
    \end{align}
    with high probability.
    Combining \eqref{eq:E ZTZ} and \eqref{eq: E ZZT} yields        
    \[
        \nu(\boldsymbol{Z}) =  \max\Big\{\lVert \mathbf{E}[\Z^{T}\Z] \rVert , \lVert \mathbf{E}[\Z\Z^{T}] \rVert\Big\}
        = \bar c \sum_{s=1}^k \tilde{\sigma}_{s}^2 n_s \le \bar c ~ \tilde\sigma_*^2 N
    \]
    with high probability, where $\tilde\sigma^2_{*} = \max \tilde \sigma^2_{s}$. This completes the proof.        
\end{proof}

\newcommand{\ts}{\tilde\sigma^2}
\newcommand{\tb}{\max \left\{\sqrt{\tilde\sigma^2 N \log N}, (\log N)^{3/2} \right\}}
To complete the proof of Lemma~\ref{lem4.4}, we use this choice of $L$ and $\nu(\Z)$ in the matrix Bernstein inequality to bound the probability that $\|\Z\|$ exceeds
\begin{equation} \label{eq: t def}
    t := c \max \left\{
        \sqrt{\tilde\sigma^2 N \log N}, (\log N)^{3/2} \right\}.
\end{equation}
To do so, We'll examine various cases.

First, suppose that $(\log N)^{3/2} \le \sqrt{\tilde\sigma^2 N \log N}$ so that 
$$
    t = \sqrt{\tilde\sigma^2 N \log N}.
$$
Substituting the values of $L$ and $\nu(\Z)$ from Lemma~\ref{lem: L bound} and Lemma~\ref{lem: nu Z} gives 
\begin{align*}
    \nu(\Z) + \frac{Lt}{3} &= \bar c \ts N + \left(\sqrt{\max\{\ts, 1\} \log N} \right) \left( \sqrt{\tilde\sigma^2 N \log N} \right) \\
    & \le \bar{c} \ts N + \mathcal{O} \left( \ts \log N \cdot \sqrt{N} \right) \\ 
    & = \O \left( \ts N \right).
\end{align*}
It follows that there is some constant $\hat c > 0$ such that
\begin{equation*}
\frac{t^2/2}{\nu(\Z) + \frac{Lt}{3}} \ge  \frac{c^2 \ts N \log N}{\hat c \ts N} = \frac{c^2}{\hat c} \log N.
\end{equation*}
Applying~\eqref{brninq}, we have 
$$
   \prob(\|\Z\| \ge t) \le (n + N) \left(\exp\left(-\frac{c^2}{\hat c} \log N\right) \right) \le 2 N^{-5}
$$
for $c$ sufficiently large.


Next, suppose that 
$$
    t = (\log N)^{3/2} \ge \tb.
$$
Then 
$$
     \nu(\Z) + \frac{Lt}{3} = \O\left((\log N)^2 \right)
$$ 
and 
$$
    \frac{t^2/2}{\nu(\Z) + \frac{Lt}{3}} \ge  \frac{c^2 (\log N)^3}{\hat c(\log N)^2} = \hat c \log N.
$$
Again, it follows that 
$$
   \prob(\|\Z\| \ge t) \le 2 N^{-5}
$$
for sufficiently large $c$ in the choice of $t$.
This completes the proof of Lemma~\ref{lem4.4}.
\qed

\end{appendix}

\end{document}